\documentclass[leqno]{amsart}
\usepackage[utf8]{inputenc}
\usepackage{newclude}
\usepackage{amsfonts,amssymb,amsmath,amsgen}
\allowdisplaybreaks
\usepackage{hyperref,color, xcolor, soul}
\usepackage{pdfsync}

\usepackage[normalem]{ulem}
\usepackage{cancel}

\usepackage{enumerate}

\usepackage{relsize} 
\usepackage{bigints}

\theoremstyle{plain}
\newtheorem{theorem}{Theorem}[section]

\newtheorem{assumption}[theorem]{Assumption}
\newtheorem{lemma}[theorem]{Lemma}
\newtheorem{corollary}[theorem]{Corollary}
\newtheorem{proposition}[theorem]{Proposition}

\theoremstyle{remark}
\newtheorem{remark}[theorem]{Remark}

\newtheorem{example}[theorem]{Example}

\def\C{{\mathbf C}}
\def\R{{\mathbf R}}
\def\N{{\mathbf N}}
\def\Z{{\mathbf Z}}

\def\dd{\mathrm d}

\def\({\left(}
\def\){\right)}
\def\<{\left\langle}
\def\>{\right\rangle}

\def\1{{\mathbf 1}}

\def\eps{\varepsilon}

\newcommand{\Eb}{\mathbb{E}}
\newcommand{\En}{\mathcal{E}}
\newcommand{\Etwo}{\mathbb{E}_{\mathrm{GL}}}
\newcommand{\Etwoc}{\mathcal{E}_{\mathrm{GL}}}
\newcommand{\Hc}{\mathcal{H}}
\newcommand{\Fc}{\mathcal{F}_c}
\newcommand{\Fone}{\mathcal{F}_1}

\newcommand{\Na}{\mathcal{N}_1}
\newcommand{\Nb}{\mathcal{N}_2}
\newcommand{\Nbl}{\mathcal{N}_{2,\mathrm{bd}}}
\newcommand{\Nbh}{\mathcal{N}_{2,\mathrm{int}}}
\newcommand{\psil}{\psi_{\mathrm{bd}}}
\newcommand{\psih}{\psi_{\mathrm{int}}}
\newcommand{\Zt}{Z_T}
\newcommand{\imu}{\mathrm{i}}
\newcommand{\eul}{\mathrm{e}}
\newcommand{\eitD}{\eul^{\frac{\imu}{2}t\Delta}}

\DeclareMathOperator{\RE}{Re}
\DeclareMathOperator{\IM}{Im}

\newcommand{\supp}{\operatorname{supp}}
\newcommand{\loc}{\operatorname{loc}}
\numberwithin{equation}{section}

\date\today

\title[Well-posedness for NLS with non-vanishing conditions at infinity]{Finite energy well-posedness for nonlinear Schr{\"o}dinger equations with non-vanishing conditions at infinity}

\author[P. Antonelli]{Paolo Antonelli}
\address{Gran Sasso Science Institute, viale Francesco Crispi, 7, 67100 L'Aquila, Italy}
\email{paolo.antonelli@gssi.it}

\author[L.E. Hientzsch]{Lars Eric Hientzsch}
\address{
Universit\"at Bielefeld, Fakult\"at f\"ur Mathematik, Postfach 10 01 31, 33501 Bielefeld, Germany \\
Department of Mathematics, Karlsruhe Institute of Technology, Englerstraße 2, 76131 Karlsruhe, Germany
}
\email{lars.hientzsch@kit.edu}

\author[P. Marcati]{Pierangelo Marcati}
\address{Gran Sasso Science Institute, viale Francesco Crispi, 7, 67100 L'Aquila, Italy}
\email{pierangelo.marcati@gssi.it}

\subjclass{Primary: 35Q55; Secondary: 35B30, 37L50.}
 \keywords{nonlinear Schr{\"o}dinger equation, Gross-Pitaevskii, well-posedness,  non-vanishing conditions at infinity}
\begin{document}
\begin{abstract}
Relevant physical phenomena are described by nonlinear Schr\"odinger equations with non-vanishing conditions at infinity. 
This paper investigates the respective 2D and 3D Cauchy problems. Local well-posedness in the (curved) energy space for energy-subcritical nonlinearities, merely satisfying Kato-type assumptions, is proven, providing the analogue of the well-established local $H^1$-theory for solutions vanishing at infinity. The critical nonlinearity will be simply a byproduct of our analysis and the existing literature.
Under an assumption that prevents the onset of a Benjamin-Feir type instability, global well-posedness in the energy space is proven for a) non-negative Hamiltonians, b) sign-indefinite Hamiltonians under additional assumptions on the zeros of the nonlinearity, c) generic nonlinearities and small initial data. The cases b) and c) only concern the 3D case.
\end{abstract}
\maketitle
\section{Introduction}
This paper is devoted to the study of the Cauchy theory for nonlinear Schr{\"o}dinger equations posed on $\R^d$ with $d=2,3$, namely 
\begin{equation}\label{eq:NLS}
    i\partial_t\psi=-\frac{1}{2}\Delta \psi+f(|\psi|^2)\psi,
\end{equation}
equipped with non-trivial boundary conditions at infinity, i.e.
\begin{equation}\label{eq:farfield}
    |\psi(x)|^2\rightarrow \rho_0 \quad \text{as} \quad |x|\rightarrow \infty, 
\end{equation}
and where the nonlinearity satisfies $f(\rho_0)=0$.
Without loss of generality, we assume $\rho_0=1$ as the general case is obtained by a suitable scaling.
The Hamiltonian (coinciding with the total energy, in many relevant physical contexts) associated to \eqref{eq:NLS} is given by
\begin{equation}\label{eq:HamiltonianNLS}
    \Hc(\psi)=\int_{\R^d}\frac{1}{2}|\nabla\psi|^2+F(|\psi|^2)\dd x, \quad \text{with} \quad  F(\rho)=\int_{1}^\rho f(r)\dd r.
\end{equation}
The finite energy assumption encodes \eqref{eq:farfield}. Namely, we deal with infinite energy solutions having finite relative energy with respect to the far-field state.\\  
The system \eqref{eq:NLS}-\eqref{eq:farfield} appears in relevant physical applications. Most prominently, the Gross-Pitaevskii (GP) equation, i.e. $f(\rho)=\rho-1$, is studied as model for Bose-Einstein condensates (BEC) \cite{Gross, P, GR74, PS}, superfluidity in Helium II close to the $\lambda$-point \cite{GinzburgP,P} and for quantum vortices \cite{P}, see also \cite{Berloff}. 
Competing (focusing-defocusing) see e.g. \eqref{eq:cubicquintic}, saturating or exponential nonlinearities for \eqref{eq:NLS}-\eqref{eq:farfield} emerge as models in nonlinear optics \cite{Barashenkov89, KL98, KR95, PelinovskySK}. Further physically relevant models are listed in Example \ref{ex:nonlinearities} below.\\
In the first part of the paper, we establish local well-posedness in the energy space for \eqref{eq:NLS}-\eqref{eq:farfield} with energy-subcritical nonlinear potentials $f$ under Kato-type \cite{Kato87} regularity assumptions. The continuity of the solution map is proven with respect to the topology of the (curved) energy space and not only in affine spaces. {The choice of a suitable functional framework plays a crucial role for the stability analysis of particular solutions \cite{ChironM, Pacherie}}. Indeed, as pointed out in \cite{Gerard08} and \cite[Remark 1.2]{CollotGermainPacherie} the constant solution with $|c|=1$ is linearly unstable in the affine space $1+H^1(\R^d)$ while orbitally stable in the energy space for $d=1,2$. A similar result holds for the Ginzburg-Landau vortex of degree one \cite{Pacherie}.
\\
Second, global well-posedness is proven, provided that $f'(1)>0$, see Assumption \ref{ass:D} below. Specifically, global well-posedness is shown for sign-definite total energies and $d=2,3$, and for sign-indefinite total energies and $d=3$ under  suitable additional assumptions on $f$ and the decay of the initial data at infinity or alternatively for small initial data. \\
Regarding the $3D$-energy critical problem, we remark that global well-posedness is easily achieved relying on the existing literature \cite{KOPV, CKSTT, TaoVisanZhang} combined with our analysis for the sub-critical case, see Section \ref{sec:energy critical}. \\
The mathematical analysis of \eqref{eq:NLS}, with far-field behavior \eqref{eq:farfield}, differs significantly from the usual $H^1$-theory for NLS equations with trivial far-field. Finite energy wave-functions are not integrable and may exhibit non-trivial oscillations at spatial infinity, in particular for $d=2$. \\
System \eqref{eq:NLS}-\eqref{eq:farfield} with defocusing nonlinearity exhibits a very rich dynamics and admits a large variety of special solutions, contrary to the case of vanishing far-field \cite{GV85}. Concerning the GP equation, the existence of sub-sonic traveling waves is known for $d=2$ \cite{BethuelS, BethuelGS} and $d=3$ \cite{BethuelS,BethuelOS, Chiron}, while non-existence in the super-sonic regime is proven in \cite{Gravejat}. 
Traveling waves exist for arbitrarily small energy for $d=2$ \cite{BethuelS}. On the contrary, for $d=3$ non-existence of traveling waves with small energy is proven in \cite{BethuelGS, deLaire}.\\
For general defocusing nonlinearities, including the nonlinearities considered in Assumption \ref{ass:D} below, the existence of sub-sonic traveling waves is investigated in \cite{Maris13, ChironM}. Non-existence in the super-sonic regime is shown in \cite{Maris08}. For $d=2$, traveling waves exist for any, and in particular arbitrarily small energy ruling out scattering, while for $d=3$ there is an energy threshold below which no traveling waves exist. 
We remark that the assumptions given in \cite{Maris13, ChironM} are strongly related with our assumptions on the nonlinear potential $f$.
The stability of multi-dimensional traveling waves is addressed in \cite{Chiron13, LinWZ} and stationary bubbles and their stability in \cite{dBouard}. Transverse instability is studied in \cite{kuznetsov}. The GP equation admits vortex solutions with infinite energy, see \cite{P, BethuelSmets} and \cite{Weinstein, Pacherie} as well as references therein for stability properties. \\
Regarding large time behavior, the existence of global dispersive solutions and small data scattering for the $3D$ and $4D$-GP equation has been investigated in a series of papers \cite{GNTT06, GNTT07, GNTT09, GuoHN}. In \cite{Killip16, Killip18} the final state problem is considered for the 3D defocusing cubic-quintic equation which is energy-critical. For general nonlinear potentials $f$, the respective problems remain open. 
\subsection{Previous well-posedness results}
Local existence of solutions to GP in Zhidkov spaces has been investigated in \cite{Zhidkov87, Zhidkov01} for $d=1$ and \cite{Gallo04} for the multi-dimensional case. 
For $d=1$, GP is known to be completely integrable \cite{zakharovShabat}. The global well-posedness of GP in the energy space is shown in \cite{Zhidkov87} and has recently been proven for fractional Sobolev and low regularity in \cite{KochLiao, KochLiao1}.
While the energy space for GP for $d=1$ coincides with the set of functions in the Zhidkov space, such that $|\psi|^2-1\in L^2(\R)$, this identification does not hold true in the multi-dimensional case, see \cite{Gerard} and Section \ref{sec:energyspace} below. The GP is well-posed in $1+H^1(\R^d)$ for $d=2,3$ \cite{BethuelS}. Global well-posedness in $1+H^s(\R^3)$ with $s\in(5/6,1)$ is proven in \cite{Pecher}. However, the space $1+H^1(\R^d)$ is strictly smaller than the natural energy space $\Eb(\R^d)$, see \eqref{eq:defEb} below.
In fact, there exist traveling waves for GP in the energy space that do not belong to $1+L^2(\R^d)$, see \cite{Gravejat}. Global well-posedness in the energy space for the multi-dimensional GP has been introduced in the seminal paper \cite{Gerard}. One of the major novelties of \cite{Gerard} consists in the precise characterization of the energy space as complete metric space and the action of the free propagator on the energy space. A more general class of defocusing and energy-subcritical $C^3$-nonlinearities has been considered in \cite{Gallo} with subsequent improvement to $C^2$-nonlinearities \cite{Miyazaki}. In \cite{Gallo, Miyazaki}, the respective authors crucially rely on a smooth decomposition of wave-functions in the energy space. Global well-posedness is proven in affine spaces determined by this decomposition which requires the aforementioned regularity assumptions and precise growth conditions for $f$. The result in the affine spaces then implies existence and uniqueness in the energy space. The cubic-quintic equation being energy-critical is studied in \cite{KOPV, Killip16, Killip18}.\\
In \cite{Carles24}, global existence of unique mild solutions to \eqref{eq:NLS}-\eqref{eq:farfield} with a logarithmic nonlinearity is introduced.\\
\subsection{Local well-posedness results}
Our first purpose is to prove local well-posedness assuming merely Kato-type regularity assumptions \cite{Kato87} and with the continuous dependence on the initial data is stated with respect to the topology of the energy space. \\
Let us point out that our well-posedness result will also be useful in the study of a class of quantum hydrodynamic (QHD) systems with non-trivial far-field \cite{AHMX}, see also \cite{AHMZ, H} for some previous results in this direction. The analysis of the Cauchy problem for QHD systems with non-zero conditions at infinity is pivotal to initiate a rigorous study of some relevant physical phenomena described by quantum fluid models, see for instance \cite{Berloff, GR74}.\\
Our main assumptions on the nonlinearity $f$ are the following. 
\begin{assumption}\label{ass:N}
Let $f$ be a real-valued function satisfying the following Kato-type assumptions, namely
\begin{enumerate}[(K1)]
    \item $f\in C([0,\infty))\cap C^1((0,\infty))$ and is such that $f(1)=0$,
     \item the nonlinearity is energy-subcritical, namely there exists $\alpha>0$, with $\alpha<\infty$ for $d=2$ and $\alpha<2$ for $d=3$, such that
     \begin{equation*}
         |f(\rho)|, |\rho f'(\rho)|\leq C(1+\rho^\alpha)
     \end{equation*}
     for all $\rho\geq 0$.
\end{enumerate}
\end{assumption}
The assumptions \emph{(K1), (K2)} are commonly referred to as Kato-type assumptions, see \cite{Kato87, Kato89} and also \cite[Chapter 4]{Cazenave}. For trivial far-field behavior, namely integrable wave-functions $\psi$, these assumptions correspond to the state of the art for the $H^1$-well-posedness for energy-subcritical nonlinearities $f$, see \cite{Cazenave} and references therein for a detailed overview of the theory.\\
The energy-subcritical power-type nonlinearities constitute an example of nonlinearities that satisfy Assumption \ref{ass:N} but in general not covered by \cite{Gallo, Gerard, Miyazaki}.
\begin{example}
The energy-subcritical power-type nonlinearities read
\begin{equation}\label{eq:powernonlinearity}
f(|\psi|^2)=\lambda(|\psi|^{2\alpha}-1), \qquad \text{with} \quad \lambda=\pm 1  \, \, \text{and} \,\,\begin{cases}
    \alpha>0 \quad &\text{for}\, d=2,\\
    0<\alpha<2 \quad &\text{for}\, d=3.
    \end{cases}
\end{equation}
These nonlinearities being included in Assumption \ref{ass:N} merely satisfy $f\in C^{0,\alpha}([0,\infty))$. Previous results require $\lambda=+1$ and $\alpha=1$ \cite{Gerard}, $\lambda>0$, and $f\in C^3([0,\infty))$ \cite{Gallo} or $f\in C^2([0,\infty))$ \cite{Miyazaki}.
The corresponding nonlinear potential energy density reads
\begin{equation}\label{eq:power-potential}
    F(|\psi|^2)=\int_{1}^{|\psi|^2}f(r)\dd r=\frac{\lambda}{\alpha(\alpha+1)}\left(|\psi|^{2(\alpha+1)}-1-(\alpha+1)(|\psi|^2-1)\right).
\end{equation}
For $\lambda=1$, we note that  $F:[0,\infty)\rightarrow \R$ is non-negative, convex and with global minimum achieved by $|\psi|^2=1$. For $\lambda=\alpha=1$, system \eqref{eq:NLS} with nonlinearity \eqref{eq:powernonlinearity} corresponds to the GP-equation
 \begin{equation}\label{eq:GP}
     i\partial_t\psi=-\frac12\Delta\psi+(|\psi|^2-1)\psi,
 \end{equation}
 for which the associated Hamiltonian energy $\mathcal{H}(\psi)$ becomes the well-known Ginzburg-Landau energy functional 
  \begin{equation}\label{eq: EGL}
  \Etwoc(\psi):=\mathcal{H}(\psi)=\int_{\R^d}\frac{1}{2}|\nabla\psi|^2+\frac{1}{2}(|\psi|^2-1)^2\dd x.
\end{equation}
Global well-posedness of \eqref{eq:GP} in the energy space has been established in \cite{Gerard} in the space of states where the associated Hamiltonian is finite, namely
\begin{equation}\label{eq:EGP}
     \begin{aligned}
         \Etwo&=\{\psi\in L_{\loc}^1(\R^d) : \mathcal{H}(\psi)<+\infty \}\\
         &=\{\psi\in L_{\loc}^1(\R^d) : \nabla\psi\in L^2(\R^d), |\psi|^2-1\in L^2(\R^d) \}.
     \end{aligned}
\end{equation}
\end{example}
In the present paper, we define the \emph{energy space} in the spirit of \cite{Zhidkov89, Zhidkov01, ChironM} as
\begin{equation}\label{eq:defEb}
    \Eb(\R^d)=\{\psi\in L_{\loc}^1(\R^d)   :  \mathcal{E}(\psi)<\infty\}
\end{equation}
with
\begin{equation}\label{eq:defE}
    \mathcal{E}(\psi)=\int_{\R^d}|\nabla\psi|^2+\left||\psi|-1\right|^2\dd x.
\end{equation}
As $||\psi|-1|\leq ||\psi|^2-1|$ it follows that $\Etwo \subset\Eb$ and the converse inclusion is straightforward to check, see Lemma \ref{lem:finiteenergy}. Working in $\Eb$ rather than $\Etwo$ is more convenient in several aspects when dealing with a general class of nonlinearities $f$ satisfying Assumption \ref{ass:N}.\\
Wave functions in $\Eb(\R^d)$ may exhibit oscillations at spatial infinity due to the non-vanishing far-field behavior, especially for $d=2$. Since $\psi\notin L^p(\R^d)$ for any $p\geq 1$, the mass is infinite. As its properties are central to the well-posedness theory, a detailed analysis of $\Eb(\R^d)$ is provided in Section \ref{sec:energyspace}. At this stage, we only mention that $\Eb(\R^d)\subset \{\Hc(\psi)<+\infty\}$ and that $\Eb(\R^d)\subset X^1(\R^d)+H^1(\R^d)$, where $X^1$ denotes the Zhidkov space \cite{Zhidkov87,Zhidkov01} defined by
\begin{equation}\label{eq:Zhidkov}
    X^1(\R^d)=\{\psi\in L^{\infty}(\R^d)   :   \nabla\psi\in L^2(\R^d)\}, \quad \|\psi\|_{X^1(\R^d)}:=\|\psi\|_{L^{\infty}(\R^d)}+\|\nabla\psi\|_{L^{2}(\R^d)}.
\end{equation}
While $\Eb$ is not a vector space, we notice that 
\begin{equation}\label{eq:metricE}
    d_{\Eb}(\psi_1,\psi_2)=\|\psi_1-\psi_2\|_{X^1+H^1}+\||\psi_1|-|\psi_2|\|_{L^2}
\end{equation}
defines a metric on $\Eb$ and $(\Eb,d_{\Eb})$ is a complete metric space. We recall that for a sum of Banach spaces, the norm is defined by
\begin{equation*}
    \|\psi\|_{X^1+H^1}=\inf\left\{\|\phi\|_{X_1}+\|u\|_{H^1} \,  : \,  \psi=\phi+u\right\}.
\end{equation*}
 Note, that the two metric spaces $\mathbb E$ and $\Etwo$ turn out to be equivalent, see Lemmata \ref{lem:metric} and \ref{lem:identification Eb} below.\\
Our first main result provides local well-posedness for \eqref{eq:NLS} in the energy space $\Eb$. It suffices to consider positive existence times. Local existence for negative times follows, as usual, from the time reversal symmetry of \eqref{eq:NLS}.
\begin{theorem}\label{thm:mainlocal}
Let $d=2,3$ and let $f$ be as in Assumption \ref{ass:N}. Then \eqref{eq:NLS} is locally well-posed in the energy space $\Eb(\R^d)$. More precisely,
\begin{enumerate}
    \item for any $\psi_0\in \Eb(\R^d)$, there exist a maximal time of existence $T^{\ast}>0$ and a unique solution $\psi\in C([0,T^{\ast});\Eb(\R^d))$ with initial data $\psi(0)=\psi_0$. The following blow-up alternative holds namely, either $T_{\ast}=\infty$ or
    \begin{equation}\label{eq:blow_up_alt}
     \lim_{t\nearrow T^{\ast}}\En(\psi)(t)=+\infty;
    \end{equation}
    \item $\psi-\psi_0\in C([0,T^{\ast});H^1(\R^d))$;
    \item the solution depends continuously on the initial data with respect to the topology induced by the metric $d_{\Eb}$;
    \item the identity $\mathcal{H}(\psi)(t)=\mathcal{H}(\psi_0)$ holds for all $t\in [0,T^{\ast})$;
     \item if in addition $\Delta\psi_0\in L^2(\R^d),$ then $\Delta\psi\in C([0,T^{\ast});L^2(\R^d))$.
\end{enumerate}
\end{theorem}
Note that (2) of Theorem \ref{thm:mainlocal} states that $\psi$ and $\psi_0$ share the same far-field behavior, i.e. they belong to the same connected component of $\Eb(\R^d)$ for all $t\in [0,T^{\ast})$, see Remarks \ref{rem: connected components 3D} and \ref{rem: connected components 2D}. Moreover, it can be shown that the nonlinear flow $\psi-\eitD\psi_0$ belongs to the full range of Strichartz spaces, see Proposition \ref{prop:localWP} and \ref{prop:LWP3d} for $d=2,3$ respectively. The precise notion of continuous dependence on the initial data is given in Proposition \ref{prop:localWP} and \ref{prop:LWP3d}. The topological structure of the metric space $(\Eb(\R^d),d_{\Eb})$ differs for $d=2$ and $d=3$, see \cite{Gerard,Gerard08}. For $d=3$, the energy space $\Eb(\R^3)$ has an affine structure; if $\psi\in \Eb(\R^3)$ then $\psi=c+v$ for some $c\in S^1$, $v\in \dot{H}^1(\R^3)$. For $d=2$, unbounded phase oscillations may occur at spatial infinity that rule out characterizing the energy space with an affine structure. The space $(\Eb(\R^2),d_{\Eb})$ is not separable. Given its relevance for the well-posedness theory, this question is going to be addressed in detail in Section \ref{sec:energyspace}. In particular, one may introduce a weaker topology that restores separability and connectedness. 
Note that this affine structure of the energy space is available for higher dimensions $d\geq 4$ to which our approach adapts. As $\Eb(\R)\subset X^1(\R)$,  the local well-posedness theory simplifies for $d=1$. We expect our approach to extend to $d=1$. Previous results \cite{Gallo04, Gallo, GialelisStratis} do not cover the full generality of Assumption \ref{ass:N}.\\
Assumption \ref{ass:N} is not sufficient in order to prove that the solution map is Lipschitz continuous. This is analogue to the $H^1$-theory for \eqref{eq:NLS} with vanishing far-field behavior. Indeed, for instance for power-law type nonlinearities \eqref{eq:powernonlinearity} Lipschitz continuity of the solution map can only be expected if $\alpha\geq \frac{1}{2}$ for both vanishing and non-vanishing far-field, see \cite[Remark 4.4.5]{Cazenave} and Section \ref{sec:Lipschitz} respectively. 
\begin{theorem}\label{thm:lipschitz}
Let $d=2,3$ and $f$ be as in Assumption \ref{ass:N}. If in addition,
\begin{equation}\label{ass:lipschitz}
f\in C^{1}([0,\infty))\cap C^2((0,\infty)), \,\, \left|\sqrt{\rho}f'(\rho)\right|, \left|\rho^{\frac{3}{2}}f''(\rho)\right|\leq C(1+\rho^{\max\{0,\alpha-\frac12\}}),    
\end{equation}
then the solution map is Lipschitz continuous on bounded sets of $\Eb(\R^d)$.\\
Namely, for any $r,R>0$ and $\psi_0^{\ast}\in \Eb(\R^d)$ such that $\En(\psi_0^{\ast})\leq R$ let $\mathcal{O}_{r}:=\{\psi_0\in \Eb(\R^d) : d(\psi_0, \psi_0^{\ast})\leq r\}$. Then, there exists $T^{\ast}(\mathcal{O}_r)>0$ such that $\psi\in C([0,T^{\ast});\Eb(\R^d))$ for all initial data $\psi(0)=\psi_0\in \mathcal{O}_r$. Moreover, for any $0<T<T^{\ast}(\mathcal{O}_r)$ there exists $C>0$ such that for any $\psi_1,\psi_2\in C([0,T];\Eb(\R^d))$ with initial data $\psi_0^1,\psi_0^2\in \mathcal{O}_r$, we have 
\begin{equation}\label{eq:Lipschitz main}
    \sup_{t\in[0,T]}d_{\Eb}(\psi_1(t),\psi_2(t))\leq C d_{\Eb}(\psi_0^1,\psi_0^2).
\end{equation}
\end{theorem}
Provided that the solutions are global, then the Lipschitz continuity holds for arbitrary times, see Corollary \ref{coro:lipschitz}.\\
\subsection{Global well-posedness results}
The proof of global existence relies on conserved quantities. Compared to the classical $H^1$-theory, the global well-posedness theory for \eqref{eq:NLS} with non-trivial farfield \eqref{eq:farfield} faces the obstacle of the lack of the conservation of mass which is infinite. No suitable notion of a "renormalized" mass being conserved seems to be available. \\
The results are inferred by means of the blow-up alternative stated in \emph{(1)} of Theorem \ref{thm:mainlocal}. 
In the following, we require the nonlinearity $f$ to be defocusing in the following sense.
\begin{assumption}\label{ass:D}
Let $f$ be as in Assumption \ref{ass:N}. Moreover, assume~$f'(1)>0$.
\end{assumption}
This assumption yields that $F$ achieves a local minimum for the constant solution $|\psi|^2=1$. 
In nonlinear optics, his requirement appears in the physical literature to prevent the onset of modulational instability, also known as Benjamin-Feir instability \cite{BenjaminFeir} of the constant equilibrium solution, i.e. the continuous wave background \cite{KivsharAndersonLisak, PelinovskySK}. \\
A sufficient condition allowing for a control of $\En(\psi)$ in terms of $\mathcal{H}(\psi)$ consists in requiring Assumption \ref{ass:D} to hold and the Hamiliton energy to be sign-definite, i.e. the nonlinear potential energy density $F$ to be non-negative.
\begin{theorem}\label{thm:main}
Let $d=2,3$. Let $f$ be such that Assumption \ref{ass:D} is satisfied and the nonlinear potential energy density $F$ defined in \eqref{eq:HamiltonianNLS} is non-negative, i.e. $F\geq 0$, then \eqref{eq:NLS} is globally well-posed in the energy space $\Eb$. 
\end{theorem}
Note that the pure power-type nonlinearities \eqref{eq:powernonlinearity} satisfy $F\geq 0$ for $\lambda>0$. \\
In the case of sign-indefinite Hamiltonian energies, the respective $H^1$-theory for \eqref{eq:NLS} fails in general to provide global existence results without further assumptions. Blow-up occurs for instance for certain focusing nonlinearities, see e.g. \cite{Cazenave}. Similar difficulties occur in the present setting, where in addition we lack the conservation of mass. We provide a global well-posedness result for $d=3$ and a class of competing (focusing-defocusing) nonlinearities $f$ for which the internal energy fails to be non-negative. Such models are of physical relevance for instance in nonlinear optics when self-focusing phenomena in a defocusing background are considered \cite{Barashenkov89, PelinovskySK}. 
We exploit the affine structure of the energy space $\Eb(\R^3)$ that can be identified with the set of functions
\begin{equation*}
 \Eb(\R^3)=\left\{\psi=c+v, \,  c\in \mathbb{C}, \, |c|=1, \,  v\in \Fc\right\}
\end{equation*}
where 
\begin{equation}\label{def:Fc}
    \Fc=\left\{v\in \dot{H}^1(\R^3)\,  :  \,  |v|^2+2\RE(\overline{c}v)\in L^2(\R^3)\right\},
\end{equation}
see \cite{Gerard} and Proposition \ref{prop:energyspace3d}.
\begin{theorem}\label{thm:main3D}
Let $d=3$, $f$ satisfy Assumptions \ref{ass:D} and further such that
\begin{equation*}
    f(r)=a(r^{\alpha_1}-1)+g(r)
\end{equation*}
with $a>0$, $0<\alpha_1<2$ and where $g$ satisfies Assumptions \ref{ass:N} \emph{(K1)} and \emph{(K2)} for some $0\leq\alpha_2<\alpha_1$. In addition, $F$ is such that $F(\rho)>0$ for all $\rho>1$.\\
Then, the solution to \eqref{eq:NLS} given by Theorem \ref{thm:mainlocal} is global provided that the initial data satisfies $\psi_0=c+v_0\in \Eb(\R^3)$ with $\RE(\overline{c}v_0)\in L^2(\R^3)$. 
\end{theorem}
The assumption on the roots of $F$ allows for physically relevant nonlinearities to be studied. It appears from the physics literature \cite{Barashenkov89, KL98, PelinovskySK} that in relevant applications the largest root of $F$ corresponds to the far-field behavior $\rho_0=1$ and constitutes a local minimum of $F$ which is linked to preventing modulational instability of the continuous background wave \cite{KivsharAndersonLisak, PelinovskySK}. To obtain global existence, we rely on the aforementioned affine structure of the energy space $\Eb(\R^3)$ and require that $\RE(\overline{c}v_0)\in L^2(\R^3)$ while $v_0\in \Fc(\R^3)$ only yields $|v_0|^2+2\RE(\overline{c}v_0)\in L^2(\R^3)$. An exponential bound on $\RE(\overline{c}v)(t)\in L^2(\R^3)$ is derived which compensates for the lack of the conserved mass due to the non-trivial farfield. The result of Theorem \ref{thm:main3D} remains valid if the assumption $\RE(\overline{c}v_0)\in L^2$ is replaced by a smallness assumption on $\Hc(\psi_0)$ and $\|\nabla\RE(\overline{c}v_0)\|_{L^2}^2$ depending only on the second largest positive root of $F$, see Remark \ref{rem:smallness}.\\
Finally, we consider the general scenario in which $F$ satisfies Assumption \ref{ass:D} is satisfied but may be unbounded from below, e.g. in the case of competing power-law nonlinearities with the focusing one dominant at large intensities.
\begin{theorem}\label{thm: smalldata GWP}
Let $d=3$, $f$ satisfy Assumptions \ref{ass:D}. There exists $\eps>0$ such that if the initial data $\psi_0$ satisfies $\Hc(\psi_0)\leq\frac{\eps}{4}$ and $\|\nabla \psi_0\|_{L^2}^2\leq \eps$, then $\psi_0\in \mathcal{E}(\R^3)$ and the solution to \eqref{eq:NLS} with $\psi(0,x)=\psi_0(x)$ given by Theorem \ref{thm:mainlocal} is global.
\end{theorem}
It remains an open problem to determine whether small data global well-posedness holds for general subcritical nonlinearities satisfying only Assumptions \ref{ass:N}.\\
To the best of the authors' knowledge, global results for \eqref{eq:NLS} in $d=2,3$ are in general not available in the literature in the case of non sign-definite total energies unless the nonlinear potential energy density is assumed to be bounded from below and in addition more regularity on $f$ \cite{Gallo}, or in the cubic-quintic case a condition on $\RE(v)$ like the ones mentioned are assumed, cf. \cite{KOPV, Killip16} are assumed. In \cite{Killip18} small-data global well-posedness for cubic-quintic nonlinearities is proven also in the case where the quintic nonlinearity is focusing. In \cite{Masaki} the authors consider for $d=1,2$ nonlinear potentials energies unbounded from below for specific regular energy subcritical nonlinearities and prove small data global existence for solutions of the form $\psi=1+u$ in tailored function spaces. 
\\
The main steps of our approach are briefly sketched. First, we identify the suitable mathematical setting for our analysis, namely the energy space $\Eb$, see \eqref{eq:defEb}. We crucially rely on the fact that $(\Eb,d_{\Eb})$ is a complete metric space as well as the properties of the free propagator introduced in \cite{Gerard, Gerard08}. The Hamiltonian $\mathcal{H}$ is well-defined for functions in $\Eb$. While wave-functions in $d=3$ can be decomposed as $\psi=c+v$ with $|c|=1, c\in \C$ and $v\in \dot{H}^1(\R^3)$, for $d=2$ the wave-functions may exhibit unbounded oscillations of the phase at spatial infinity. This motivates treating separately the well-posedness problem for $d=2,3$. In both cases, we show local existence of a solution in the affine space $\psi=\psi_0+H^1(\R^d)$ by a perturbative Kato-type argument \cite{Kato87} and also \cite[Chapter 4]{Cazenave}. Subsequently, uniqueness in $C([0,T];\Eb(\R^d))$ is proven.  The fixed-point argument only provides continuous dependence with respect to perturbations in the space $\psi_0+H^1(\R^d)$. The proof of continuous dependence on the initial data with respect to the topology induced by the metric $d_{\Eb}$ requires additional estimates and differs in a substantial way from the $H^1$-well-posedness theory for NLS-equations with vanishing conditions at infinity. This is due to the non-integrability of wave-functions and the intricate topological structure of the energy space linked to the far-field behavior including oscillations of the phase and the low regularity of the nonlinearity. Global well-posedness is shown relying on the conservation of the Hamiltonian $\mathcal{H}$.
\\
While our method for the $3D$-theory exploits the particular structure of the energy space, the approach used for $d=2$ can easily be adapted to sub-cubic nonlinearities for $d=3$. However, for super-cubic nonlinearities, we exploit the affine structure of $\Eb(\R^3)$. It is then no longer sufficient to work in $L^2$-based spaces as done for $d=2$ but we need that the gradient of the solution belongs to the full range of Strichartz spaces.
\\
In \cite{Gallo,Miyazaki} the authors rely on a decomposition of the initial data as $\psi=\varphi+H^1$ with $\varphi\in C_b^{\infty}$ and develop a well-posedness theory in the affine space $\varphi+H^1$. This approach requires additional regularity assumptions on $f$ not needed for our method. 
\\
For the 3D energy-critical quintic equation, one may proceed as described in Section \ref{sec:energy critical}. 
\\
We conclude this section by providing further examples of physical relevance that enter the class of nonlinearities characterised by Assumption \ref{ass:N}.
 \begin{example}\label{ex:nonlinearities}
Beyond the mentioned power-type nonlinearities, here are some examples of physically relevant nonlinearties and far-field \eqref{eq:farfield}:
\begin{enumerate}
    \item competing nonlinearities $f(\rho)=a\rho^{\alpha_1}-b\rho^{\alpha_2}+c$ with $a,b,c>0$ and $\sigma_1\geq \sigma_2\geq 0$ that arise in the description of self-focusing phenomena in defocusing media \cite{KR95,KL98, PelinovskySK}, see also \cite{SulemSulem, Zhidkov01},
    \item saturated nonlinearities $f(\rho)=\frac{\rho}{1+\gamma \rho}-\frac{1}{1+\gamma}$ with $\gamma>0$, see for instance \cite[Chapter 9.3]{SulemSulem} and references therein,
    \item exponential nonlinearities $f(\rho)=(\eul^{-\gamma}-\eul^{-\gamma\rho})$ with $\gamma>0$ \cite[Chapter 9.3]{SulemSulem},
    \item transiting nonlinearities of the form $f(\rho)=2\rho\left(1+\alpha\tanh\left(\gamma(\rho^2-1\right)\right)$ occurring in nonlinear optics \cite[Section VI]{PelinovskySK},
    \item logarithmic nonlinearities of type $f(\rho)=\rho\log(\rho)$ which arise in the context of dilute quantum gases, see \cite{Carles} and references therein,
    \item the nonlinearity $f(\rho)=\rho^{-1}(\rho-1)$ arises in the study of $1D$-NLS type equations as model for nearly parallel vortex filaments, see \cite{KleinMajdaDamodaran} and \cite[Eq. (1.5)]{BanicaMiot}.
\end{enumerate}
The cubic-quintic equation \eqref{eq:cubicquintic} falls within (1) of the aforementioned list and is also recovered in the small amplitude approximation of (2) and (3) of the above examples \cite[Chapter 9.3]{SulemSulem}. We also refer to \cite{ChironScheid} for a more detailed overview.
 \end{example}
\begin{remark}
Nonlinear Schr\"odinger equations equipped with non-vanishing boundary conditions \eqref{eq:farfield} are also investigated with non-local interaction, e.g. in the context of supersolids, see \cite{lewin} for an overview of relevant models. While we do not pursue this topic here, we expect that our analysis combined with \cite{deLaireNonlocal} allows for a straightforward extension to general non-local nonlinearities. 
\end{remark}
\subsection{The energy-critical equation}\label{sec:energy critical}
We briefly discuss the Cauchy problem for the energy-critical equation for $d=3$, namely the quintic equation
\begin{equation}\label{eq:3Dcritical}
     i\partial_t\psi=-\frac{1}{2}\Delta\psi+(|\psi|^4-1)\psi.
 \end{equation}
The well-posedness of \eqref{eq:3Dcritical} is not addressed by Theorem \ref{thm:mainlocal}. Local well-posedness for small data is introduced in \cite[Theorem 1.3]{Gallo}. Furthermore, note that the cubic-quintic equation
\begin{equation}\label{eq:cubicquintic}
    i\partial_t\psi=-\frac{1}{2}\Delta\psi+\left(\alpha_5|\psi|^4-\alpha_3|\psi|^2+\alpha_1\right)\psi.
\end{equation}
with $\alpha_1,\alpha_3,\alpha_5>0$, $\alpha_3^2-4\alpha_1\alpha_5>0$ and far-field \eqref{eq:farfield} is known to be globally well-posed in the respective energy space due to \cite{KOPV}. The cubic-quintic nonlinearity considered satisfies Assumption \ref{ass:D} and is such that $F(1)=0$ and $F(\rho)>0$ for all $\rho>1$. The authors rely on the affine structure of the respective energy space for $d=3$, the perturbative approach introduced in \cite{TaoVisan, TaoVisanZhang} and the well-posedness of the energy-critical nonlinear Schr\"odinger equation with trivial far-field \cite{CKSTT}. This approach can be adapted to show global well-posedness of \eqref{eq:3Dcritical}. More precisely, it is straightforward to update the perturbative argument, see \cite[Eq. (1.14) and (1.15)]{KOPV} to the respective problem for \eqref{eq:3Dcritical}, see also \eqref{eq:NLS3d}. 
 \subsection{Outline of the paper}
 The remaining part of the paper is structured as follows. Section \ref{sec:energyspace} provides preliminary results on the energy space $\Eb$, its structure and the action of the Schr\"odinger group on $\Eb$. Useful estimates for the nonlinearity are collected. Section \ref{sec:2D} introduces first local and second global well-posedness for $d=2$. More precisely, Theorem \ref{thm:mainlocal} and Theorem \ref{thm:main} are proven for $d=2$. In Section \ref{sec:3d}, we provide the respective proofs for $d=3$. Further, Theorem \ref{thm:main3D} is proven. Finally, Section \ref{sec:Lipschitz} is devoted to the proof of Theorem \ref{thm:lipschitz} and Corollary \ref{coro:lipschitz}.
\subsection{Notations}
We fix some notations. We denote by $\mathcal{L}^d$ the $d$-dimensional Lebesgue measure. The usual Lebesgue spaces are denoted by $L^p(\Omega)$ for $\Omega\subset \R^d$ and Lebesgue exponent $p\in [1,\infty]$. Sobolev spaces are denoted by $H^s(\R^d)$ with norm $\|f\|_{H^s(\R^d)}=\|\left\langle\xi\right\rangle^{s}\hat{f}\|_{L^2}$, where $\hat{f}$ denotes the Fourier transform. For $k\in \Z$ and $r\in [1,\infty]$, we write $W^{k,r}$ for the Sobolev space with norm $\|f\|_{W^{k,r}}=\sum_{|\alpha|\leq k}\|D^{\alpha}f\|_{L^r(\R^d)}$. Mixed space-time Lebesgue or Sobolev spaces are indicated by $L^{p}(I;W^{k,r}(\R^d))$. To shorten notations, we write $L_t^pW_x^{k,r}$ when there is no ambiguity. Further, $C(I;H^s(\R^d))$ and $C(I;\Eb(\R^d))$ denote the space of continuous $H^s$- and $\Eb$-valued functions respectively. 
Finally, $C>0$ denotes any absolute constant.
\section{The energy space and the linear propagator}\label{sec:energyspace}
In the present paper, we define the energy space $\Eb$ as in \eqref{eq:defEb}, see also \cite[Section 2]{ChironM}. For the GP equation \eqref{eq:GP}, being the prototype for \eqref{eq:NLS} with non-vanishing far-field, the energy space considered in \cite{Gerard, Gerard08} consists of the set of wave-functions of finite Ginzburg-Landau energy $\Etwoc(\psi)$ is more convenient when dealing with general nonlinearities $f$. In general, $\Eb\subset \{\mathcal{H}(\psi)<+\infty\}$ while the converse inclusion only holds under further assumptions on $f$. The energy space $(\Eb,d_{\Eb})$, endowed with the metric \eqref{eq:metricE} can be shown to be a complete metric space and may be regarded as the analogue of $H^1$ for NLS equations with trivial far-field. However, $\Eb$ is not a vector space and wave functions $\psi\in\Eb(\R^d)$ may exhibit oscillations at spatial infinity, in particular for low dimensions. A suitable characterization of the energy space and the action of the Schr\"odinger semigroup on $\Eb$ is essential for the subsequent well-posedness theory.
Although many of the facts proven here can be found in the literature \cite{Gerard, Gerard08, ChironM}, we provide a self-contained characterization of the energy space $\Eb$.
\\
To that end, we rely on decompositions of functions by means of smooth cut-off functions that we introduce. Let $\chi\in C_c^{\infty}([0,\infty))$ be a smooth cut-off function such that
\begin{equation}\label{eq:cutoffchi}
      \mathbf{1}_{[0,2]}(r)\leq \chi(r)\leq \mathbf{1}_{[0,3)}(r).
\end{equation}
In particular, given a wave-function $\psi: \R^d\rightarrow \C$ we introduce
\begin{equation}\label{eq:psihighlow}
    \psil:=\chi(|\psi|)\psi, \qquad \psih:=(1-\chi(|\psi|))\psi,
\end{equation}
Further, let $\eta\in C_c^{\infty}([0,\infty))$ with $\supp(\eta)\subset [\frac12,\frac32]$ such that
\begin{equation}\label{eq:eta}
    \mathbf{1}_{[\frac{3}{4},\frac{5}{4}]}(r)\leq \eta(r)\leq \mathbf{1}_{[\frac{1}{2},\frac{3}{2}]}(r),  
\end{equation}
We observe that if $\psi\in \Eb(\R^d)$, then it follows for all $c>0$ from the Chebychev inequality that, for all $c>0$,
\begin{equation}\label{eq:Cheby}
\mathcal{L}^d(\{\left||\psi|-1\right|>c\}\leq \frac{1}{c^2}\||\psi|-1\|_{L^2(\R^d)}^2,
\end{equation}
where $\mathcal{L}^d$ denotes the $d$-dimensional Lebesgue measure.
Consequently,
for all $\psi\in \Eb(\R^d)$ the supports of $(1-\chi(|\psi|))$ and $(1-\eta(|\psi|))$ are of finite Lebesgue measure
\begin{equation}\label{eq:supp eta}
   \mathcal{L}^d(\supp(1-\chi(|\psi|)))\leq \En(\psi), \qquad \mathcal{L}^d(\supp(1-\eta(|\psi|)))\leq 16\En(\psi).
\end{equation}
Following \cite[Lemma 1]{Gerard}, we start by proving that any $\psi\in\Eb(\R^d)$ can be decomposed as sum of a $X^{1}$-function and an $H^1$-function, where the Zhidkov space $X^1(\R^d)$ is defined in \eqref{eq:Zhidkov}.
\begin{lemma}\label{lem:energyspace}
The energy space $(\Eb(\R^d),d_{\Eb})$ with $d_{\Eb}$ defined by \eqref{eq:metricE} is a complete metric space and is embedded in $X^1(\R^d)+H^1(\R^d)$. In particular, for any $\psi\in \Eb$ one has
\begin{equation*}
    \|\psil\|_{X^1(\R^d)}\leq C\left(1+\sqrt{\En(\psi)}\right), \quad \text{and} \quad
    \|\psih\|_{H^1(\R^d)}\leq C\sqrt{\En(\psi)}.
\end{equation*}
Moreover, the energy space is stable under $H^1$ perturbations, in the sense that $\Eb(\R^d)+H^1(\R^d)\subset \Eb(\R^d)$ with
\begin{equation}\label{eq:energyH1}
    \En(\psi+u)\leq 2\En(\psi)+2\|u\|_{H^1(\R^d)}^2.
\end{equation}
\end{lemma}
For $d=1$, one has $\Eb(\R)\subset X^1(\R)$ due to Sobolev embedding. 
\begin{proof}
Given the decomposition \eqref{eq:psihighlow}, we show that $\psil\in X^1(\R^d)$. As $\psil\in L^{\infty}(\R^d)$ it suffices to check that
\begin{equation*}
    \|\nabla\psil\|_{L^2(\R^d)}=\|\chi(|\psi|)\nabla\psi+\psi\chi'(|\psi|)\nabla|\psi|\|_{L^2(\R^d)}\leq C\|\nabla\psi\|_{L^2(\R^d)},
\end{equation*}
where we used $|\nabla|\psi||\leq |\nabla\psi|$ a.e. on $\R^d$.
The bound $\psih\in L^2(\R^d)$ follows from the pointwise inequality $|\psih|\leq C \left||\psih|-1\right|$ valid on the support of $1-\chi(|\psi|)$ and
\begin{equation*}
    \|\nabla\psih\|_{L^2(\R^d)}\leq C\|\nabla\psi\|_{L^2(\R^d)}.
\end{equation*}
To prove \eqref{eq:energyH1}, it suffices to observe that if $\psi\in\Eb(\R^d)$ and $u\in H^1(\R^d)$, then 
\begin{align*}
    \|\nabla(\psi+u)\|_{L^2(\R^d)}^2&\leq 2\|\nabla \psi\|_{L^{2}(\R^d)}^2+ 2\|\nabla u\|_{L^{2}(\R^d)}^2,\\
    \||\psi+u|-1\|_{L^2(\R^d)}^2&\leq 2\||\psi|-1\|_{L^2(\R^d)}^2+2\|u\|_{L^2(\R^d)}^2
\end{align*}
by means of Minkowski's inequality. It remains to prove that $(\Eb,d_{\Eb})$ is a complete metric space. One readily verifies that $d_{\Eb}$ defines a distance function on $\Eb(\R^d)$. To check that $(\Eb,d_{\Eb})$ is complete, let $\{\psi_n\}_{n}\subset \Eb$ be a Cauchy sequence with respect to $d_{\Eb}$. Then, there exists $\psi\in X^1+H^1$ such that $\psi_n\rightarrow \psi$ strongly in $X^1+H^1$. By lower semi-continuity of norms and \eqref{eq:defE} it follows that $\psi\in \Eb$. 
\end{proof}
\subsection{The structure of the energy space depending on the dimension}
The structure of the energy space $\Eb(\R^d)$ is sensitive to the dimension $d$. To illustrate this, we recall the following fact.
Let $\phi\in\mathcal{D}'(\R^d)$, if $\nabla\phi\in L^p(\R^d)$ for some $p<d$, then there exists $c\in\C$ such that $\phi-c\in L^{p^{\ast}}(\R^d)$, where $p^{\ast}=\frac{dp}{d-p}$, see for instance \cite[Theorem 4.5.9]{Hormander}.
Hence, if $\psi\in\Eb(\R^3)$, then $\psi$ admits a decomposition $\psi=c+v$ where $c\in \C$ with $|c|=1$ and $v\in \dot{H}^1(\R^3)$, where
\begin{equation}\label{eq:Hdot}
    \dot{H}^1(\R^3)=\{v\in L^6(\R^3) \,  : \,  \nabla v\in L^2(\R^3)\},
\end{equation}
denotes the completion of $C_0^{\infty}(\R^3)$ with respect to the $L^2$ norm of the gradient.
This observation allows for a equivalent definition of $\Eb(\R^3)$. We equip the space $\Fc$ defined in \eqref{def:Fc} with
\begin{equation*}
    \tilde{\delta}(u,v)=\|\nabla u-\nabla v\|_{L^2(\R^3)}+\||u|^2+2\RE(c^{-1}u)-2\RE(c^{-1}v)-|v|^2\|_{L^2(\R^3)}.
\end{equation*}
It is straightforward to verify that $\delta$ defines distance function on $\Fc$. One has the following characterization given by \cite[Proposition 4.1]{Gerard}. 
\begin{proposition}[\cite{Gerard}]\label{prop:energyspace3d}
For $d=3$, the energy space $\Eb(\R^3)$ can be identified with the set of functions
\begin{equation}
 \Eb(\R^3)=\left\{\psi=c+v, \,  c\in \mathbb{C}, \, |c|=1, \,  v\in \Fc\right\}. 
\end{equation}
Moreover the metric function $d_{\Eb}$ is equivalent to 
\begin{multline}\label{eq:metric3d}
    \delta(c+v,\tilde{c}+\tilde{v})\\=|c-\tilde{c}|+\left\|\nabla v-\nabla\tilde{v}\right\|_{L^2(\R^3)}+\left\||v|^2+2\RE(c^{-1}v)-|\tilde{v}|^2-2\RE(\tilde{c}^{-1}\tilde{v})\right\|_{L^2(\R^3)}.
\end{multline}
\end{proposition}
In \cite{Gerard}, the Proposition is stated for $(\Etwo, d_{\Etwo})$. We prove below, see Lemma \ref{lem:metric}, that the two metric spaces can be identified and that the topologies induced by the respective metrics are equivalent.
\begin{remark}\label{rem: connected components 3D}
We observe that the connected components of $\Eb(\R^3)$ are given by $c+\Fc(\R^3)$ for $c\in \C$ with $|c|=1$. The energy space $\Eb(\R^3)$ is an affine space and the far-field behavior is determined by $c$ corresponding to a phase shift. The affine structure of the energy space allows for an alternative approach to solve the Cauchy Problem for $d=3$, as observed in \cite[Remark 4.5]{Gerard} for \eqref{eq:GP} and exploited in \cite{KOPV} for cubic-quintic nonlinearities and far-field behavior \eqref{eq:farfield}.
\end{remark}
\begin{remark}\label{rem: connected components 2D}
The $2D$ energy space $\Eb(\R^2)$ lacks an affine structure due to non-trivial oscillations at spatial infinity. Indeed, unbounded phase oscillations at spatial infinity may occur, e.g. $\psi(x)=\eul^{\imu(2+\log|x|)^{\beta}}$ with $\beta<\frac12$ is such that $\psi\in \Eb(\R^2)$, see \cite[Remark 4.2]{Gerard}. Moreover, the metric space $(\Eb(\R^2),d_{\Eb})$ is not separable.
We refer to Remark \ref{rem:topology} for a detailed discussion and a weakened topology for which $\Eb(\R^2)$ is connected and separable.
\end{remark}
\subsection{The Hamiltonian for wave-functions in the energy space}\label{sec:H vs E}
The following inequality turns out to be handy for applications in the sequel. For any $q\in[1,\infty)$ there exists $C_q>0$ such that for all $\phi\in L_{\loc}^1(\R^2)$ with $\mathcal{L}^2(\supp(\phi))<+\infty$ and $\nabla\phi\in L^2(\R^2)$, we have
    \begin{equation}\label{eq:finite supp}
        \|\phi\|_{L^q(\R^2)}\leq C_q\|\nabla\phi\|_{L^2(\R^2)}\left(\mathcal{L}^2(\text{supp}(\phi)\right)^{\frac{1}{q}},
    \end{equation}
see for instance \cite[Proof of Lemma 2.1]{ChironM}. For $\psi\in \Eb(\R^2)$, applying \eqref{eq:finite supp} to $\phi=\psi(1-\eta(|\psi|))$ with $\eta$ as defined in \eqref{eq:eta} yields $\psi(1-\eta(|\psi|))\in L^q(\R^2)$ for any $q\in[1,\infty)$. Indeed, in view of \eqref{eq:supp eta} it suffices to check that
\begin{equation*}
    \nabla\left(\psi(1-\eta(|\psi|))\right)=(1-\eta(|\psi|))\nabla\psi-\eta'(|\psi|)\psi\nabla|\psi|\in L^2(\R^2)
\end{equation*}
since $(1-\eta(|\psi|))\in L^{\infty}(\R^2)$, $\psi\eta'(|\psi|)\in L^{\infty}(\R^2)$ as well as $|\nabla|\psi||\leq |\nabla\psi|$ a.e. on $\R^2$.
\\
Next, we show that the functional $\mathcal{H}(\psi)$, introduced in \eqref{eq:defE}, is bounded for all  $\psi\in\Eb(\R^d)$ provided that Assumption \ref{ass:N} holds. Specifically, we have $F(1)=0$ and $F'(1)=f(1)=0$ which allows one to control $F(|\psi|^2)\leq C(|\psi|^2-1)^2$ for $|\psi|^2$ sufficiently close to $1$. Furthermore, if one also requires Assumption \ref{ass:D} ($f'(1)>0$) to be satisfied, then it follows from Taylor expansion that
\begin{equation}\label{eq:Taylor F}
F(r)\simeq \frac12 f'(1)(r-1)^2
\end{equation}
in a small neighborhood of $1$. Hence, there exists $\delta>0$ and $C_1, C_2>0$ such that 
\begin{equation}\label{F:convex}
    \frac{1}{C_2}(|\psi|-1)^2\leq \frac{1}{C_1}(|\psi|^2-1)^2\leq F(|\psi|^2)\leq C_1(|\psi|^2-1)^2\leq C_2(|\psi|-1)^2
    \end{equation}
provided that $||\psi|^2-1|<\delta$. Note that the following Lemma only requires the upper bound on $F$ close to $1$.
\begin{lemma}\label{lem:finiteenergy}
For $d=2,3$ and $f$ satisfying Assumption \ref{ass:N} one has 
\begin{equation*}
    \Eb(\R^d)\subset\left\{\psi \, : \, \left|\Hc(\psi)\right|<+\infty\right\}.
\end{equation*}
\end{lemma}
Note that we do not require Assumption \ref{ass:D} to hold for Lemma \ref{lem:finiteenergy}. 
\begin{proof}
In view of \emph{(K1)} Assumption \ref{ass:N}, a Taylor expansion of $F$ in a small neighborhood $\mathcal{O}$ of $1$ yields that there exist $C,C'>0$ such that 
\begin{equation*}
    F(|\psi|^2)\leq C'(|\psi|^2-1)^2\leq C(|\psi|-1)^2,
\end{equation*}
for all $x\in\R^d$ such that $|\psi|^2\in\mathcal{O}$. Let $\delta>0$ be such that $B(1,\delta)\subset \mathcal{O}$ and $\eta_{\delta}(r):=\eta(\frac{r}{\delta})$ with $\eta$ as in \eqref{eq:eta} and $\psi\in \Eb(\R^d)$, then 
\begin{align*}
    \int_{\R^d}F(|\psi|^2)\dd x=\int_{\R^d}F(|\psi|^2)\eta_\delta(|\psi|)\dd x+\int_{\R^d}F(|\psi|^2)(1-\eta_\delta(|\psi|))\dd x \\
    \leq C \int_{\R^d}\left||\psi|-1\right|^2\dd x+C\int_{\R^d}\left(1+|\psi|^{2\alpha}\right)\left||\psi|^2-1\right|(1-\eta_\delta(|\psi|))\dd x,\\
\end{align*}
where we used \emph{(K2)} Assumption \ref{ass:N} in the last inequality. To control the second term, we consider separately the cases $d=2,3$. For $d=3$, Proposition \ref{prop:energyspace3d} yields that there exist $c\in \C$ with $|c|=1$ and $v\in \Fc(\R^3)$ such that $\psi=c+v$ and
\begin{align*}
    &\int_{\R^3}\left(1+|\psi|^{2\alpha}\right)\left||\psi|^2-1\right|(1-\eta_\delta(|\psi|))\dd x\\ &\leq C \int_{\R^d}(1-\eta_\delta(|\psi|))\chi(\psi)\dd x
    +\int_{\R^3}|c+v|^{2(\alpha+1)}(1-\chi(\psi))\dd x\\
    &\leq C \En(\psi)+\|v\|_{L^6}^{2(1+\alpha)}\En(\psi)^{\frac{2-\alpha}{3}}\leq C\left(\En(\psi)+\En(\psi)^{\frac{5+2\alpha}{3}}\right),
\end{align*}
where we used \eqref{eq:supp eta} in the second to last inequality and that $0<\alpha<2$ for $d=3$. For $d=2$, one has that 
\begin{multline*}
    \int_{\R^2}\left(1+|\psi|^{2\alpha}\right)\left||\psi|^2-1\right|(1-\eta_\delta(|\psi|))\dd x\\
    \leq C \int_{\R^d}(1-\eta_\delta(|\psi|))\chi(\psi)\dd x+\int_{\R^d}\left(1+|\psi|^{2(\alpha+1)}\right)(1-\chi(\psi))\dd x
\end{multline*}
The first integral is bounded by $C\En(\psi)$ and for the second it follows from \eqref{eq:finite supp} that 
\begin{equation*}
    \|\psi(1-\chi(|\psi|))\|_{L^{2(\alpha+1)}(\R^2)}^{2(\alpha+1)}\leq \En(\psi)^{1+\alpha}\mathcal{L}^2(\supp(\psi(1-\chi(\psi))))\leq \En(\psi)^{2+\alpha}.
\end{equation*}
This allows one to bound 
\begin{equation*}
    \int_{\R^2}(1+|\psi|^{2\alpha})\left||\psi|^2-1\right|(1-\eta_\delta(\psi))\dd x\leq \En(\psi)+\En(\psi)^{2+\alpha}.
\end{equation*}
\end{proof}
Next, we identify suitable conditions on $f$ under which the converse inclusion, namely $\{\psi : |\Hc(\psi)|<+\infty\}\subset \Eb(\R^d)$, holds true.
First, we treat the particular case of the Gross-Pitaevskii equation \eqref{eq:GP} for which $\mathcal{H}(\psi)=\Etwoc(\psi)$ and so $\Etwo(\R^d)=\{\mathcal{H}(\psi)<+\infty\}$, see \eqref{eq: EGL} and \eqref{eq:EGP} respectively. 
It has been shown in \cite{Gerard}, see also \cite{Gerard08}, that $(\Etwo,d_{\Etwo})$ with
\begin{equation}\label{eq:metricGP}
    d_{\Etwo}(\psi_1,\psi_2)=\|\psi_1-\psi_2\|_{X^1+H^1}+\||\psi_1|^2-|\psi_2|^2\|_{L^2}.
\end{equation}
is a complete metric space. It is pointed out in \cite[p.13]{ChironM} without proof that $\Eb=\Etwo$ with equivalent respective metrics. We provide a proof for the sake of completeness. 
\begin{lemma}\label{lem:metric}
Let $d\geq 1$, then $\Eb(\R^d)=\Etwo(\R^d)$. Moreover, for $d=2,3$ and any $R>0$, there exists $C=C(R)>0$ such that for any $\psi_1,\psi_2$ with $\En(\psi_i)\leq R$ for $i=1,2$ it holds
\begin{equation}\label{eq:metricXH}
    \frac{1}{C}d_{\Etwo}(\psi_1,\psi_2)\leq d_{\Eb}(\psi_1,\psi_2)\leq C d_{\Etwo}(\psi_1,\psi_2).
\end{equation}
Moreover, there exists $C>0$ such that for $\psi_1,\psi_2\in \Eb(\R^d))$ and $u,v\in H^1(\R^d)$ it holds
\begin{multline}\label{ineq:triangular}
d_{\Eb}(\psi_1+u,\psi_2+v)\\
\leq C\left(1+\sqrt{\En(\psi_1)}+\sqrt{\En(\psi_2)}+\|u\|_{H^1}+\|v\|_{H^1}\right)\left(d_{\Eb}(\psi_1,\psi_2)+\|u-v\|_{H^1}\right).
\end{multline}
\end{lemma}
\begin{remark}\label{rem:topology}
Lemma \ref{lem:metric} allows to infer the topological properties of $(\Eb,d_{\Eb})$ from the results for $(\Etwo(\R^d),d_{\Etwo})$ in \cite{Gerard,Gerard08}.
For instance, the functional $\En$ measures the distance to the circle of constants $S^1=\{\psi\in \Eb : \En(\psi)=0\}$ for $d=3$ but not for $d=2$. Indeed, it follows from Lemma \ref{lem:metric} and \cite[Proposition 4.3]{Gerard} that there exists $A>0$ such that for every $\psi\in \Eb(\R^3)$,
\begin{equation*}
    \frac{1}{A}d_{\Eb}(\psi,S^1)^2\leq \mathcal{E}_{\mathrm{GL}}(\psi)\leq C d_{\Eb}(\psi, S^1)^2.
\end{equation*}
If $d=2$, there exists a sequence $\{\psi_n\}$ in $\Eb(\R^2)$ such that $\En(\psi_n)\rightarrow 0$ but $d_{\Eb}(\psi_n, S^1)\geq c_0>0$. Note that the complete metric space $(\Etwo(\R^2),d_{\Etwo})$ lacks an affine structure and to be separable.
In \cite{Gerard08} a detailed characterisation of $\Etwo(\R^d)$ including a manifold structure for $\Etwo(\R^d)$ is provided. The connected components are characterised by \cite[Theorem 1.8]{Gerard08} and \cite[Proposition 1.10]{Gerard08}. We introduce a (strictly) weaker topology \cite[p. 140]{Gerard08} induced by the metric 
\begin{equation*}
    d_{\Eb}'(\psi_1,\psi_2):=\|\psi_1-\psi_2\|_{L^2(B(1,0))}+\|\nabla\psi_1-\nabla\psi_2\|_{L^2(\R^2)}+\||\psi_1|^2-|\psi_2|^2\|_{L^2(\R^2)}
\end{equation*}
is introduced. It follows that $(\Eb,d_{\Eb}')$ is connected. Relying on the decomposition of elements of $\Eb$ provided by \cite[Theorem 1.8]{Gerard08}, one can show that $(\Eb,d_{\Eb}')$ is separable. If one only requires continuity of the solution map with respect to this weakened topology, the proof of Proposition \ref{prop:localWP} can be simplified. This metric has widely been used in the study of the stability of special solutions for $d=1$. We refer to \cite{KochLiao}, where the authors introduce new energy spaces for \eqref{eq:GP} and $d=1$ in order to tackle global well-posedness in the energy space at $H^s$-regularity.
\end{remark}
\begin{proof}
We start by showing that there exists $C>0$ such that 
\begin{equation*}
    \||\psi_1|-|\psi_2|\|_{L^2(\R^d)}\leq C\left(\||\psi_1|^2-|\psi_2|^2\|_{L^2(\R^d)}+\|\nabla\psi_1-\nabla\psi_2\|_{L^2(\R^d)}\right).
    \end{equation*}
Indeed, let $\chi_6(z)=\chi(6z)$ with $\chi$ defined in \eqref{eq:cutoffchi}, then
\begin{multline*}
    \||\psi_1|-|\psi_2|\|_{L^2(\R^d)}\\
    \leq \||\psi_1|\chi_6(\psi_1)-|\psi_2|\chi_6(\psi_2)\|_{L^2(\R^d)}+\||\psi_1|(1-\chi_6(\psi_1))-|\psi_2|(1-\chi_6(|\psi_2|))\|_{L^2(\R^d)}.
\end{multline*}
The second contribution can be bounded by
\begin{equation*}
    \||\psi_1|(1-\chi_6({\psi_1}))-|\psi_2|(1-\chi_6(\psi_2))\|_{L^2(\R^d)}\leq C \||\psi_1|^2-|\psi_2|^2\|_{L^2(\R^d)}.
\end{equation*}
Next, we notice that for $i=1,2$, the support of $\chi_6(\psi_i)$ is of finite measure as $\psi_i\in\Eb(\R^d)$, see \eqref{eq:Cheby}. 
For $d=2$, by invoking \eqref{eq:finite supp} applied to $\phi=|\psi_1|\chi_6({|\psi_1|})-|\psi_2|\chi_6(|\psi_2|)$, we conclude that 
\begin{multline*}
    \||\psi_1|\chi_6({|\psi_1|})-|\psi_2|\chi_6(|\psi_2|)\|_{L^2(\R^2)}\\
    \leq C\left(\sqrt{\En(\psi_1)}+\sqrt{\En(\psi_2)}\right)\left(\|\psi_1-\psi_2\|_{X^1+H^1(\R^2)}+\||\psi_1|^2-|\psi_2|^2\|_{L^2(\R^2)}\right). 
\end{multline*}
For $d=3$, one proceeds similarly exploiting the decomposition $\psi_i=c_i+v_i$, $v_i\in \Fc(\R^3)$ and Proposition \ref{prop:energyspace3d}. We have 
\begin{multline*}
     \||\psi_1|\chi_6({|\psi_1|})-|\psi_2|\chi_6(|\psi_2|)\|_{L^2(\R^3)}\\ \leq C\left(1+\sqrt{\En(\psi_1)}+\sqrt{\En(\psi_2)}\right)\left(|c_1-c_2|+\|\nabla v_1-\nabla v_2\|_{L^2(\R^3)}\right)
     \leq C(R) d_{\Etwo}(\psi_1,\psi_2).
\end{multline*}
Next, we show that there exists $C=C(R)>0$ such that
\begin{equation*}
    \||\psi_1|^2-|\psi_2|^2\|_{L^2(\R^d)}\leq C_1\left(\||\psi_1|-|\psi_2|\|_{L^2(\R^d)}+\|\psi_{1}-\psi_{2}\|_{X^1+H^1(\R^d)}\right).
\end{equation*}
It suffices to notice that 
\begin{equation*}
   \||\psi_1|^2\chi(\psi_1)-|\psi_2|^2\chi(\psi_2)\|_{L^2(\R^d)}\leq C_1\||\psi_1|-|\psi_2|\|_{L^2(\R^d)},  
\end{equation*}
while
\begin{multline*}
     \||\psi_1|^2(1-\chi(\psi_1))-|\psi_2|^2(1-\chi(\psi_2))\|_{L^2(\R^d)}\\
     \leq C\left(1+\sqrt{\En(\psi_1)}+\sqrt{\En(\psi_2)
     }+\|\psi_{1,\mathrm{int}}\|_{L^4(\R^d)}+\|\psi_{2,\mathrm{int}}\|_{L^4(\R^d)}
     \right)\|\psi_{1,\mathrm{int}}-\psi_{2,\mathrm{int}}\|_{L^4(\R^d)}\\
      \leq 2C\left(1+\sqrt{\En(\psi_1)}+\sqrt{\En(\psi_2)
     }
     \right)\|\psi_{1,\mathrm{int}}-\psi_{2,\mathrm{int}}\|_{L^4(\R^d)}.
\end{multline*}
In the second last inequality, we used that
\begin{equation}\label{eq:pointwise psiq}
    |\psi|^{4}\sqrt{1-\chi(\psi)}\leq C\left|\psih\right|^{4},
\end{equation}
with $\psih$ defined in \eqref{eq:psihighlow} which is only valid provided $(1-\chi(\psi))>\theta$ for some small $\theta>0$. However, this is harmless as
\begin{equation*}
    \mathcal{L}^2\left(\left\{x\in \supp(1-\chi(\psi)) \,  : \,  0<1-\chi(\psi)\leq \theta \right\}\right)\leq \sqrt{\En(\psi)}
\end{equation*}
and $|\psi|\leq 3$ on the respective set. The error can be controlled at the expense of a factor $\sqrt{\En(\psi)}$ in the estimate.
One has that 
\begin{equation*}
    \|\psi_{1,\mathrm{int}}-\psi_{2,\mathrm{int}}\|_{L^4(\R^d)}\leq C \left(\sqrt{\En(\psi_1)}+\sqrt{\En(\psi_2)}\right) \|\psi_{1}-\psi_{2}\|_{X^1+H^1(\R^d)}
\end{equation*}
by means of \eqref{eq:finite supp} for $d=2$ and the decomposition provided by Proposition \ref{prop:energyspace3d} for $d=3$.
Finally,
\begin{align*}
  \||\psi_1|^2&-|\psi_2|^2\|_{L^2(\R^d)}\\
  &\leq C\left(1+\sqrt{\En(\psi_1)}+\sqrt{\En(\psi_2)}\right)\left(\||\psi_1|-|\psi_2|)\|_{L^2(\R^d)}+ \|\psi_{1}-\psi_{2}\|_{X^1+H^1(\R^d)}\right) . 
\end{align*}
It remains to show \eqref{ineq:triangular}. The respective property is known for $d_{\Etwo}$, see \cite[Lemma 2]{Gerard}, and hence follows from the equivalence of metrics. However, we provide a proof to track constants explicitly. Note that
\begin{equation*}
    \||\psi_1+u|-|\psi_2+v|\|_{L^2}\leq  \||\psi_1+u|\chi_6(\psi_1+u)-|\psi_2+v|\chi_6(\psi_2+v)\|_{L^2}+\||\psi_1+u|^2-|\psi_2+v|^2\|_{L^2},
\end{equation*}
by arguing as in the first part of the proof. By invoking \eqref{eq:finite supp}, one has
\begin{multline*}
    \||\psi_1+u|\chi_6(\psi_1+u)-|\psi_2+v|\chi_6(\psi_2+v)\|_{L^2}\\
    \leq C\left(\sqrt{\En(\psi_1)}+\sqrt{\En(\psi_2)}+\|u\|_{H^1}+\|v\|_{H^1}\right)\left(\|\psi_{1}-\psi_{2}\|_{X^1+H^1(\R^d)}+\|u-v\|_{H^1}\right).
\end{multline*}
For the second term, one has
\begin{align*}
    &\||\psi_1+u|^2-|\psi_2+v|^2\|_{L^2}\\
    &\leq \||\psi_1|^2-|\psi_2|^2\|_{L^2}+\||u|^2-|v|^2\|_{L^2}+\|2\RE(\overline{\psi_1}u)-2\RE(\overline{\psi_2}v)\|_{L^2}\\
    &\leq \||\psi_1|^2-|\psi_2|^2\|_{L^2}+\left(\|u\|_{H^1}+\|v\|_{H^1}\right)\|u-v\|_{H^1}\\
    &+2\|\RE\left((\overline{\psi}_{1,\infty}+\overline{\psi}_{1,q})(u-v)\right)\|_{L^2}
    +2\|\RE\left(\left(\overline{\psi}_{1,q}-\overline{\psi}_{2,q}+\overline{\psi}_{1,\infty}-\overline{\psi}_{2,\infty}\right)v\right)\|_{L^2}\\
    &\leq \||\psi_1|^2-|\psi_2|^2\|_{L^2}+\left(\|u\|_{H^1}+\|v\|_{H^1}+1+\En(\psi_1)\right)\|u-v\|_{H^1}
   +\|v\|_{H^1}d_{\Eb}(\psi_1,\psi_2)\\
    &\leq C\left(1+\sqrt{\En(\psi_1}+\sqrt{\En(\psi_2)}+\|u\|_{H^1}+\|v\|_{H^1}\right)\left(d_{\Eb}(\psi_1,\psi_2)+\|u-v\|_{H^1}\right).
\end{align*}
\end{proof}
Next, we provide a sufficient condition on $f$ under which the space of functions with finite Hamiltonian energy is included in $\Eb$. To that end, we require Assumption \ref{ass:D} to be satisfied
so that \eqref{F:convex} holds. The nonlinear potential energy density $F$ is locally convex in a neighborhood of $1$. It was shown in \cite[Lemma 4.8]{ChironM} that requiring in addition that $F\geq 0$ and hence the Hamiltonian energy is sign-definite, implies that $ \Eb=\{\mathcal{H}(\psi)<\infty\}$. Note that the condition $F\geq 0$ is for instance satisfied for the pure power-type nonlinearities in \eqref{eq:powernonlinearity}.
\begin{lemma}\label{lem:identification Eb}
Let $d=2,3$ and suppose that Assumption \ref{ass:D} is satisfied. If in addition $F\geq 0$, then 
\begin{equation*}
   \Eb=\{\mathcal{H}(\psi)<\infty\}.
\end{equation*}
In particular, there exists an increasing function $g: (0,\infty)\rightarrow [0,\infty)$ with $\displaystyle \lim_{r\rightarrow 0}g(r)=0$ such that
\begin{equation}\label{eq:equivEH}
    \En(\psi)\leq g\left(\mathcal{H}(\psi)\right).
\end{equation}
\end{lemma}
By exploiting Lemma \ref{lem:metric} and the conservation of the Hamiltonian along solutions to (1.1), it is then possible to extend the local solutions globally in time. Notice that when in the framework of NLS equations with trivial far-field, the blow-up alternative is given in terms of the $H^1$-norm, whereas here it involves $\mathcal E(\psi)$. In the classical, integrable case, it is possible to infer the analogue of \eqref{eq:equivEH} under less restrictive assumptions on $F$; for instance it is possible to consider mass-subcritical focusing nonlinearities. In this case indeed the analogue of \eqref{eq:equivEH} is derived by exploiting Gagliardo-Nirenberg inequalities. However, the lack of a suitable control of the mass in our case prevents us from considering more general nonlinearities.
\begin{proof}
We sketch of the proof, see \cite{ChironM} for full details. First, we borrow from \cite[Equation (1.18)]{ChironM} the following equivalent definition of $\Etwo(\R^d)=\Eb(\R^d)$. Let $\varphi\in C^{\infty}(\R)$ be such that $\varphi(r)=r$ for $r\in[0,2]$, $0\leq \varphi'\leq 1$ on $\R$ and $\varphi(r)=3$ for $r\geq 4$. We define the modified Ginzburg-Landau energy
\begin{equation*}
    \mathcal{E}_{\mathrm{mGL}}(\psi)=\int_{\R^d}|\nabla\psi|^2+\frac{1}{2}\left(\varphi(|\psi|)^2-1\right)^2\dd x.
\end{equation*}
The functional $\Etwoc$ is well-approximated by $\mathcal{E}_{\mathrm{mGL}}$. Indeed, it is shown in \cite[Section 2]{ChironM} that
\begin{equation*}
    \Etwo(\R^d)=\{\psi\in L_{\loc}^1(\R^d) \,  : \,  \nabla\psi\in L^2(\R^d), \, \varphi(|\psi|)^2-1\in L^2(\R^d)\}.
\end{equation*}
Since $|\varphi(|\psi|)^2-1|\leq 4 ||\psi|-1|$, one has $\varphi(|\psi|)^2-1\in L^2(\R^d)$ if $\psi\in\Eb(\R^d)$. For the converse, see \cite[Lemma 2.1]{ChironM}. We sketch the main idea. On the set where $|\psi(x)|\leq 2$, one has $\varphi(|\psi|)^2=|\psi|^2$ and hence the desired bound follows. Further, $\mathcal{L}^d(\{x\, : \, ||\psi(x)|-1|>\frac32\})<+\infty$ from the Chebychev inequality \eqref{eq:Cheby} if $\varphi(|\psi|)^2-1\in L^2(\R^d)$.  By means of \eqref{eq:finite supp} for $d=2$ and Sobolev embedding for $d=3$ one concludes. Finally, there exist $C>0$ and an increasing function $m: \R_{+}\rightarrow \R_{+}$ with $\displaystyle \lim_{r\rightarrow 0}m(r)=0$ such that
\begin{equation*}
    \frac{1}{4} \mathcal{E}_{\mathrm{mGL}}(\psi)\leq \En(\psi)\leq C m\left(\mathcal{E}_{\mathrm{mGL}}(\psi)\right), 
\end{equation*}
see \cite[Corollary 4.3]{ChironM}.
Second, we note that it suffices to establish inequality \eqref{eq:equivEH} with $\mathcal{E}$ replaced by $\mathcal{E}_{\mathrm{mGL}}$.
In virtue of \eqref{F:convex}, it suffices to consider the region where $\{x : ||\psi|-1|\geq \delta\}$. If $\inf F>0$ on $\{x : ||\psi|-1|\geq \delta\}$, then it is clear that 
\begin{equation*}
    \int_{\{||\psi|-1|\geq \delta\}}\left(\varphi(|\psi|)^2-1\right)^2\dd x \leq C  \int_{\{||\psi|-1|\geq \delta\}}F(|\psi|^2)\dd x.
\end{equation*}
It follows that $\En(\psi)$ can be controlled in terms of $\mathcal{H}(\psi)$. More generally, provided that $F\geq 0$, it follows from \cite[Lemma 4.8]{ChironM} that for all $\psi$ with $|\mathcal{H}(\psi)|<\infty$ there exist $C_1=C_1(\mathcal{H}(\psi))>0$ and $C_2=C_2(\mathcal{H}(\psi))>0$ such that
\begin{equation*}
    C_1\left(\mathcal{H}(\psi)\right)\leq \mathcal{E}_{\mathrm{mGL}}(\psi)\leq C_2\left(\mathcal{H}(\psi)\right).
\end{equation*}
The statement of Lemma \ref{lem:identification Eb} follows.
\end{proof}
\begin{remark}
System \eqref{eq:NLS} is closely related to the QHD system with non-trivial far-field. In a reminiscent analysis, the regularity and integrability properties of its unknowns $(\rho,J)$ corresponding to the mass density $\rho=|\psi|^2$ and momentum density $J=\IM(\overline{\psi}\nabla\psi)$ are then captured in terms of Orlicz spaces, see \cite{AHMZ} and \cite[Chapter 2]{H} as well as \cite{AHM18,HAIMS} for the respective uniform bounds for solutions to the quantum Navier-Stokes equations, a viscous regularization of the QHD system. 
\end{remark}
\subsection{Smooth approximation}
Elements of the energy space can be approximated by smooth functions via convolution with a smooth mollifier.
\begin{lemma}\label{lem:approx}
Let $\psi\in\Eb(\R^d)$, then there exists $\{\psi_n\}_{n\in \N}\subset C^{\infty}(\R^d)\cap \Eb(\R^d)$ such that 
\begin{equation*}
    d_{\Eb}(\psi,\psi_n)\rightarrow 0,
\end{equation*}
as $n\rightarrow 0$. Moreover, for any $\psi\in \Eb(\R^d)$, there exists $\varphi\in C_b^{\infty}(\R^d)\cap \Eb(\R^d)$ such that $\nabla\varphi\in H^{\infty}(\R^d)$ and
\begin{equation}
    \psi-\varphi \in H^1(\R^d).
\end{equation}
\end{lemma}
The first statement is proven in \cite[Lemma 6]{Gerard} by considering the convolution with a standard mollification kernel and the second statement follows from  \cite[Proposition 1.1.]{Gallo}. In \cite{Gerard,Gallo}, the statements are given for $(\Etwo,d_{\Etwo})$ being equivalent to $(\Eb,d_{\Eb})$ by virtue of Lemma \ref{lem:metric}.
\subsection{Action of the linear propagator on the energy space}
The action of the linear Schr{\"o}dinger group on the space $X^k(\R^d)+H^k(\R^d)$ is well-defined, see \cite[Lemma 3]{Gerard} and also \cite{Gerard08}. While the results in \cite{Gerard, Gerard08} are stated for $(\Etwo,d_{\Etwo})$, we state them $(\Eb,d_{\Eb})$ which by Lemma \ref{lem:metric} is equivalent. 
\begin{lemma}[\cite{Gerard}]\label{lem:semigroup}
Let $d$ be a positive integer. For every $k$, for every $t\in \R$, the operator $\eitD$ maps $X^k(\R^d)+H^k(\R^d)$ into itself and it satisfies
\begin{equation}\label{eq:linearflow}
    \|\eitD f\|_{X^k+H^k}\leq C\left(1+|t|\right)^{\frac{1}{2}}\|f\|_{X^k+H^k},
\end{equation}
and 
\begin{equation}\label{eq:linearflow1}
    \|\eitD f-f\|_{L^2}\leq C |t|^{\frac{1}{2}}\|\nabla f\|_{L^2}.
\end{equation}
Moreover, if $f\in X^k(\R^d)+H^k(\R^d)$, the map $t\in \R\mapsto \eitD f\in X^k(\R^d)+H^k(\R^d)$ is continuous.
\end{lemma}
For $d=1$, we notice that $X^k(\R)+H^k(\R)\subset X^k(\R)$ for any $k$ positive integer. The action of $\eitD$ on $X^1(\R)$ has been studied in \cite{Zhidkov87,Zhidkov01}, see also \cite{Gallo04} for the action of the linear propagator on Zhidkov spaces $X^k(\R^d)$ with $d>1$. 
\\
The action of the linear Schr\"odinger group on the space $\Eb(\R^d)$ is described by \cite[Proposition 2.3]{Gerard}.
\begin{proposition}[\cite{Gerard}]\label{prop:linear}
Let $d=2,3$. For every $t\in\R$, the linear propagator $\eitD$ maps $\Eb(\R^d)$ to itself and for every $\psi_0\in\Eb(\R^d)$ the map $t\in \R \mapsto \eitD\psi_0 \in \Eb(\R^d)$ is continuous. Moreover, given $R>0$, $T>0$ there exists $C>0$ such that for every $\psi_0^1,\psi_0^2\in \Eb(\R^d)$ with $\mathcal{E}(\psi_0
^1)\leq R$, $\mathcal{E}(\psi_0^2)\leq R$ one has
\begin{equation}\label{eq:stabilitylinsol}
    \sup_{|t|\leq T}d_{\Eb}(\eitD\psi_0^1,\eitD\psi_0^2)\leq C d_{\Eb}(\psi_0^1,\psi_0^2).
\end{equation}
Furthermore, given $R>0$, there exists $T(R)>0$ such that, for every $\psi_0\in \Eb(\R^d)$ with $\En(\psi_0)\leq R$, we have
\begin{equation}\label{eq:freesolution}
     \sup_{|t|\leq T(R)}\mathcal{E}(\eitD\psi_0)\leq 2R.
\end{equation}
\end{proposition}
\begin{corollary}\label{coro:time derivative}
Let $d=2,3$ and $\psi_0\in \Eb(\R^d)$, then
\begin{equation}\label{eq: derivative semigroup}
    \lim_{t\rightarrow 0}\frac{\eitD\psi_0-\psi_0}{t}=-\frac{i}{2}\Delta\psi_0 \quad \text{in} \, \,  H^{-1}(\R^d).
\end{equation}
In particular, $\eitD\psi_0\in C(\R;\Eb(\R^d))\cap C^{1}(\R,H^{-1}(\R^d))$.
\end{corollary}
\begin{proof}
Note that $\eitD\psi_0-\psi_0\in L^2(\R^d)$ for any finite time $t\in\R$ by virtue of \eqref{eq:linearflow1}. For any $\phi\in H^1(\R^d)$, it follows from Plancherel's identity and the dominated convergence theorem that
\begin{align*}
    &\lim_{t\rightarrow 0}\int_{\R^d}\frac{\eitD \psi_0-\psi_0}{t}\overline{\phi}(x)\dd x=\lim_{t\rightarrow 0}\int_{\R^d}\frac{\eul^{\frac{\imu}{2}t|\xi|^2}\hat{\psi_0}-\hat{\psi_0}}{t}\overline{\hat{\phi}}(\xi)\dd \xi\\
    &=\lim_{t\rightarrow 0}\int_{\R^d}\frac{\imu}{2}|\xi|^2\left(\int_0^1\eul^{\imu ts|\xi|^2}\right)\hat{\psi_0}(\xi)\overline{\hat{\phi}}(\xi)\dd \xi
    =\int_{\R^d}(-\frac{\imu}{2}\Delta \psi_0(x))\overline{\phi}(x)\dd x.
\end{align*}
The identity \eqref{eq: derivative semigroup} follows.
\end{proof}
\subsection{Strichartz estimates}\label{sec:Stichartz}
 We say that a pair $(q,r)$ is (Schr{\"o}dinger) admissible if $q,r\geq 2$ such that 
\begin{equation*}
    \frac{2}{q}+\frac{d}{r}=\frac{d}{2}, \qquad (q,r,d)\neq(2,\infty,2),
\end{equation*}
and we recall the well-known Strichartz estimates, see \cite{KeelTao} and references therein.
\begin{lemma}\label{lem:Strichartz}
Let $d=2,3$ and $(q,r)$ be an admissible pair. Then the linear propagator satisfies,
\begin{equation*}
    \|\eitD u\|_{L^q([0,T];L^r(\R^d))}\leq C \|u\|_{L^2(\R^d)},
\end{equation*}
and for any $(q_1,r_1)$ admissible pair one has
\begin{equation}\label{eq:Strichartznonh}
    \left\|\int_0^t\eul^{\frac{\imu}{2}(t-s)\Delta}f(s)\dd s\right\|_{L^q([0,T];L^r(\R^d)}\leq C\|f\|_{L^{q_1'}([0,T];L^{r_1'}(\R^d))}.
\end{equation}
\end{lemma}
Given a time interval $I=[0,T]$, it is convenient to introduce the Strichartz space $S^0(I\times \R^d)$ characterised by the norm 
\begin{equation*}
    \|u\|_{S^0}:=\sup_{(q,r) admissible}\|u\|_{L^q(I;L^r(\R^d))}.
\end{equation*}
We notice that since $(q,r)=(\infty,2)$ is admissible one has 
\begin{equation}
    \|u\|_{C(I;L^2(\R^d))}\lesssim \|u\|_{S^0}.
\end{equation}
Moreover, we introduce the dual space $N^0=(S^0(I\times \R^d))^{\ast}$ satisfying the estimate
\begin{equation}\label{eq:N0}
    \|f\|_{N^0}\lesssim \|f\|_{L^{q_1'}(I;L^{r_1'}(\R^d))},
\end{equation}
for any admissible pair $(q_1,r_1)$. Further, in order to discuss the well-posedness theory for \eqref{eq:NLS} in the energy space, we also work with the function space $S^1(I\times \R^d)$ and $N^1(I\times \R^d)$ defined by the norms
\begin{equation}\label{eq:S1N1}
    \|u\|_{S^1}=\|u\|_{S^0}+\|\nabla u\|_{S^0}, \qquad \|G\|_{N^1}=\|G\|_{N^0}+\|\nabla G\|_{N^0}.
\end{equation}
While $\psi\not\in S^0$ for any solution to \eqref{eq:NLS} to l in any Strichartz space $S^0$, it will turn out that the nonlinear flow belongs to $S^1$.
\begin{remark}
Let $T>0$ and $\psi_0\in \Eb(\R^d)$, then Lemma \ref{lem:Strichartz} states that for any admissible pair $(q,r)$, we have
\begin{equation}\label{eq:Strichartz gradient}
    \left\|\eitD\nabla\psi_0\right\|_{L^q([0,T];L^r(\R^d))}\leq \|\nabla \psi_0\|_{L^2(\R^d)}.
\end{equation}
Observe that, by virtue of Lemma \ref{lem:semigroup}, one has $\eitD\psi_0-\psi_0\in C([0,T];H^1(\R^d))$ and  $\nabla \eitD\psi_0\in C([0,T];L^2(\R^d))\cap S^0([0,T]\times\R^d))$.
\end{remark}
\subsection{The nonlinearity}
We collect some properties of the nonlinearity $\mathcal{N}(\psi)=f(|\psi|^2)\psi$, with $f$ satisfying Assumption \ref{ass:N}, that will be used in the sequel.  By applying smooth cut-off functions, we separate the behavior close and away from $|\psi|=1$.  Let $\eta\in\C_c^{\infty}(\R_+)$ be given by \eqref{eq:eta}, we define
\begin{equation}\label{eq:N}
    \Na(\psi):=\mathcal{N}(\psi)\eta(|\psi|), \qquad \Nb(\psi):=\mathcal{N}(\psi)(1-\eta(|\psi|)).
\end{equation}
By means of the cut-off $\chi$ defined in \eqref{eq:cutoffchi}, we further split $\Nb$ as 
\begin{equation}\label{eq:Nb}
    \Nbl=\Nb(\psi)\chi(2\psi), \quad \Nbh(\psi)=\Nb(\psi)(1-\chi(2\psi))
\end{equation}
and notice that 
\begin{equation}\label{eq:Npointwise}
    \begin{aligned}
    |\Na(\psi)|&\leq C\left||\psi|-1\right|, \\
    |\Nbl(\psi)|&\leq C(1-\eta(|\psi|), \quad |\Nbh(\psi)|\leq C|\psi|^{2\alpha+1} (1-\chi(\psi)).
      \end{aligned}
\end{equation}
In the case of vanishing boundary conditions and infinity, the strategy developed in \cite{Kato87}, see also \cite[Chapter 4]{Cazenave}, relies on similar pointwise bounds on $\mathcal{N}$. However, here we need to consider additional cut-off functions $\eta$ isolating the behavior close to $1$ in view of the far-field and the related support properties. Note that \eqref{eq:Cheby} yields that the measure of $\supp(\Nb(\psi))$ is bounded by $\En(\psi)$. The quantity $\nabla\mathcal{N}$ can be rigorously defined by means of Nemicki operators, see \cite[Appendix A]{Kato87} and also \cite{Kato89,Cazenave}. It  reads
\begin{equation}\label{eq:nablaN}
    \nabla\mathcal{N}(\psi)=\left(f(|\psi|^2)+f'(|\psi|^2)|\psi|^2\right)\nabla\psi+f'(|\psi|^2)\psi^2\overline{\nabla\psi},
\end{equation}
so that we have
\begin{equation}\label{eq:N_locLip}
|\nabla\mathcal N(\psi)|\lesssim\left(|f(\rho)+\rho f'(\rho)|+|\rho f'(\rho)|\right)|\nabla\psi|.
\end{equation}
Inequalities \eqref{eq:Npointwise} and \eqref{eq:N_locLip} will allow us to infer bounds on the nonlinearity in the Strichartz space $N^1$ defined in \eqref{eq:S1N1}.
Moreover, (\emph{K2}) of Assumption \ref{ass:N} implies that the nonlinearity $\mathcal N(\psi)$ is locally Lipschitz.
More precisely, 
\begin{equation}\label{eq: N loc Lip}
    \left|\mathcal{N}(\psi_1)-\mathcal{N}(\psi_2)\right|\leq C\left(1+|\psi_1|^{2\alpha}+|\psi_2|^{2\alpha}\right)|\psi_1-\psi_2|.
\end{equation}
For general $\psi_1,\psi_2\in \Eb(\R^d)$ one has $\psi_1-\psi_2\notin L^p(\R^d)$ for any $p\geq 1$, unless $\psi_1,\psi_2$ belong to the same connected component of $\Eb(\R^d)$, see Remarks \ref{rem: connected components 3D} and \ref{rem: connected components 2D} for $d=2,3$ respectively. This motivates the following estimates,
\begin{equation}\label{eq:pointwise Nb}
\begin{aligned}
    \left|\Na(\psi_1)-\Na(\psi_2)\right|&\leq C|\psi_1|\left||\psi_1|-|\psi_2|\right|+\left||\psi_2|-1\right|\eta(|\psi_2|)|\psi_1-\psi_2|,\\
    \left|\Nbl(\psi_1)-\Nbl(\psi_2)\right|&\leq C\left|\psi_1-\psi_2\right|,\\
    \left|\Nbh(\psi_1)-\Nbh(\psi_2)\right|&\leq C\left(|\psi_1|^{2\alpha}+|\psi_2|^{2\alpha}\right)\left|\psi_1-\psi_2\right|.
\end{aligned}
\end{equation}
Inequalities \eqref{eq:pointwise Nb} will then lead to respective bounds in Strichartz space $N^0$.
\\
Similarly, we introduce the following estimates for $\nabla\mathcal{N}(\psi)$. One has
\begin{equation*}
   \nabla\mathcal{N}(\psi)=D\mathcal{N}(\psi)\cdot \begin{pmatrix} \nabla \psi \\ \overline{\nabla\psi}\end{pmatrix}=\begin{pmatrix}
  G_1(\psi) \\ G_2(\psi) \end{pmatrix}^T\cdot \begin{pmatrix} \nabla \psi \\ \overline{\nabla\psi}\end{pmatrix},
\end{equation*}
where 
\begin{equation}\label{eq: Gi}
    G_1(\psi)= f(|\psi|^2)+f'(|\psi|^2)|\psi|^2, \quad \text{and} \quad G_2(\psi)=f'(|\psi|^2)\psi^2.
\end{equation}
We define 
\begin{equation*}
    G_{i,\mathrm{bd}}(\psi):= G_i(\psi)\chi(\psi), \quad \text{and} \quad   G_{i,\mathrm{int}}(\psi):= G_{i}(\psi)(1-\chi(\psi)),
\end{equation*}
for $i=1,2$. For the sake of a shorter notation we introduce
\begin{equation}\label{eq: Ginfty Gq}
    G_{\mathrm{bd}}:=G_{1,\mathrm{bd}}+G_{2,\mathrm{bd}},\quad \text{and} \quad  G_{\mathrm{int}}:=G_{1,\mathrm{int}}+G_{2,\mathrm{int}}.
\end{equation}
In particular we observe that Assumption \ref{ass:N} yields that
\begin{equation}\label{eq: estG}
    |G_{\mathrm{bd}}(\psi)|\leq C, \quad \text{and} \quad |G_{\mathrm{int}}(\psi)|\leq C(1+|\psi_{\mathrm{int}}|^{2\alpha})(1-\chi(\psi)).
\end{equation}
 \section{2D well-posedness}\label{sec:2D}
Local well-posedness for energy sub-critical nonlinearities is proven by a perturbative method in the spirit of Kato \cite{Kato87} adapted to the non-trivial farfield behavior.
Subsequently, we prove global well-posedness in Section \ref{sec:global2D}.
\subsection{Local well-posedness} 
First, we provide necessary a priori bounds on the nonlinearity $\mathcal{N}(\psi)$ in the Strichartz norms for $\psi\in \Eb(\R^2)$ that will follow from \eqref{eq:Npointwise} and \eqref{eq:N_locLip}. We notice that $(q_1,r_1)=(\frac{2(\alpha+1)}{\alpha},2(\alpha+1))$ is Strichartz admissible and one has
\begin{equation}\label{eq:q1r1}
(q_1',r_1')=\left(\frac{2(\alpha+1)}{\alpha+2},\frac{2(\alpha+1)}{2\alpha+1}\right).    
\end{equation}
We recall that the space $N^0$ is defined in \eqref{eq:N0}. It suffices to consider positive times of existence as the analogue statements for negative times follow  from the time reversal symmetry of \eqref{eq:NLS}. For $\psi\in L^{\infty}([0,T];\Eb(\R^d))$ we denote
\begin{equation}\label{eq:Zt}
   \Zt:=\|\nabla\psi\|_{L^{\infty}([0,T];L^2(\R^2))}+\||\psi|-1\|_{L^{\infty}([0,T];L^2(\R^2))}
\end{equation}
and note that $\Zt(\psi)\leq 2 \sup_{t\in[0,T]}\sqrt{\En(\psi)(t)}$.
The quantity $\Zt(\psi)$ can be thought of as analogue of the $L^\infty_tH^1_x-$norm for nonlinear Schr\"odinger equations with vanishing conditions at infinity.
\begin{lemma}\label{lem:N}
Let the nonlinearity $f$ be as in Assumption \ref{ass:N} is satisfied, $T>0$, the pair $(q_1',r_1')$ as in \eqref{eq:q1r1} and $\psi\in L^{\infty}([0,T];\Eb(\R^2))$, then the following hold
\begin{equation}\label{eq:boundN}
    \|\mathcal{N}(\psi)\|_{L^1([0,T];L^2(\R^2))}\leq CT\left({\Zt(\psi)}+\Zt(\psi)^{1+2\alpha}\right),
\end{equation}
and 
\begin{equation}\label{eq:boundgradN}
    \|\nabla\mathcal{N}(\psi)\|_{N^0([0,T]\times\R^2)}\leq C\left(T+T^{\frac{1}{q_1'}}\Zt(\psi)^{2\alpha}\right)\|\nabla\psi\|_{L^{\infty}([0,T];L^2(\R^2)}.
\end{equation}
Furthermore, given $\psi\in L^{\infty}([0,T];\Eb(\R^2))$ and $u,v\in L^{\infty}([0,T];H^1(\R^2))$, one has that
\begin{multline}\label{eq:diffNSobolev}
    \|\mathcal{N}(\psi+u)-\mathcal{N}(\psi+v)\|_{N^0([0,T]\times\R^2)}\\
    \leq C\Big(T+T^{\frac{1}{q_1'}}\left(\Zt(\psi+u)^{2\alpha}+\Zt(\psi+v)^{2\alpha}\right)\Big)\|u-v\|_{L^{\infty}([0,T];L^2(\R^2)}.
\end{multline}
\end{lemma}
\begin{proof}
Let $\psi\in\Eb(\R^2)$. To infer \eqref{eq:boundN}, we observe that \eqref{eq:Npointwise} implies
\begin{equation*}
    \|\Na(\psi)\|_{L_t^1L_x^2}\leq CT \left\||\psi|-1\right\|_{L_t^{\infty}L_x^2}\leq CT \Zt(\psi).
\end{equation*}
To obtain the bound of $\Nb(\psi)$, we note that the Chebychev inequality \eqref{eq:Cheby} yields that $\text{supp}(1-\eta(\psi))$ is of finite Lebesgue measure for all $\psi\in \Eb(\R^2)$. It follows then from Lemma \ref{lem:energyspace} and \eqref{eq:Npointwise} that
\begin{equation*}
     \|\Nbl(\psi)\|_{L_t^1L_x^2}\leq C T\mathcal{L}^2\left(\supp(1-\eta(|\psi|)\right)^{\frac12}\leq CT\Zt(\psi).
\end{equation*}
 By exploiting that $\supp(1-\eta(\psi))\subset \supp(1-\chi(\psi))$ for $\psi\in\Eb(\R^2)$ and by \eqref{eq:Cheby}, we bound the third contribution as
\begin{multline*}
    \|\Nbh(\psi)\|_{L_t^1L_x^2}
    \leq C \||\psi|^{2\alpha}|\psi|(1-\chi(\psi))\|_{L_t^1L_x^2}
    \leq C T\Zt(\psi)+CT\|\psih\|_{L_t^{\infty}L_x^{2(1+2\alpha)}}^{1+2\alpha}\\
    \leq CT\left(\Zt(\psi)+\Zt(\psi)^{1+2\alpha}\right),
\end{multline*}
where $\psi_{\mathrm{int}}$ is defined in \eqref{eq:psihighlow}, with $\chi$ given in \eqref{eq:cutoffchi}. In the second inequality, we used that
\begin{equation}\label{eq:pointwise psi}
    |\psi|^{2\alpha+1}(1-\chi(\psi))\leq C\left(\mathbf{1}_{\{0<1-\chi(\psi)\leq 1/4\}}+\left|\psih\right|^{2\alpha+1}\right),
\end{equation}
and 
\begin{equation*}
    \mathcal{L}^2\left(\left\{x\in \supp(1-\chi(\psi)) \, : \, 0<1-\chi(\psi)\leq 1/4\right\}\right) \leq {\Zt(\psi)^2}.
\end{equation*}
To control $\nabla\mathcal{N}(\psi)$, we observe that by using \eqref{eq:N_locLip} and decomposing $\psi=\psil+\psih$, see \eqref{eq:psihighlow}, it follows that
\begin{multline*}
    \left\|\nabla\mathcal{N}(\psi)\right\|_{L_t^1L_x^2+L_t^{q_1'} L_x^{r_1'}}\leq C T \|\nabla\psi\|_{L_t^{\infty}L_x^2}+\||\psih|^{2\alpha}\nabla\psi\|_{L_t^{q_1'} L_x^{r_1'}}\\\leq 
    C\left(T+T^{\frac{1}{q_1'}}\Zt(\psi)^{2\alpha}\right)\|\nabla\psi\|_{L_t^{\infty}L_x^2}.
\end{multline*}
It remains to show \eqref{eq:diffNSobolev}. Let $\psi\in L^{\infty}([0,T];\Eb(\R^2))$ and $u,v\in L^{\infty}([0,T];H^1(\R^2))$. Then, \eqref{eq: N loc Lip} implies the pointwise bound
\begin{equation*}
    \left|\mathcal{N}(\psi+u)-\mathcal{N}(\psi+v)\right|\leq C\left(1+|\psi+u|^{2\alpha}+|\psi+v|^{2\alpha}\right)|u-v|.
\end{equation*}
Exploiting that $\Eb(\R^2)+H^1(\R^2)\subset \Eb(\R^2)$ from Lemma \ref{lem:energyspace}, we proceed as before to infer that for a.e. $t\in[0,T]$ we have
\begin{equation*}
    \||\psi+u|^{2\alpha}\|_{L_x^{\infty}+L_x^{q_1}}+\||\psi+v|^{2\alpha}\|_{L_x^{\infty}+L_x^{q_1}}\\
    \leq C\left(1+\Zt(\psi+u)^{2\alpha}+\Zt(\psi+v)^{2\alpha}\right).
\end{equation*}
It follows that 
\begin{multline*}
    \|\mathcal{N}(\psi+u)-\mathcal{N}(\psi+v)\|_{L_t^1L_x^{2}+L_t^{q_1'}L_x^{r_1'}}\\
    \leq C\left(T+T^{\frac{1}{q_1}'}\left(\Zt(\psi+u)^{2\alpha}+\Zt(\psi+v)^{2\alpha}\right)
    \right)\|u-v\|_{L_t^{\infty}L_x^2},
\end{multline*}
yielding \eqref{eq:diffNSobolev}.
\end{proof} 
With the bounds of Lemma \ref{lem:N} and the Strichartz estimates of Lemma \ref{lem:Strichartz} at hand, we are able to prove existence and uniqueness of solutions to \eqref{eq:NLS}. To that end, we consider the equivalent Duhamel formula
\begin{equation}\label{eq:Duhamel}
    \psi(t)=\eitD\psi_0-i\int_0^t\eul^{\frac{\imu}{2}(t-s)\Delta}\mathcal{N}(\psi)(s)\dd s
\end{equation}
which is justified as identity in $\Eb(\R^3)$ in virtue of the properties of the free solutions from Proposition \ref{prop:linear} and the fact the non-homogeneous terms is bounded in $L_t^{\infty}H_x^1$ by means of the Strichartz estimate \eqref{eq:Strichartznonh} and Lemma \ref{lem:N}. 
\\
 We anticipate that the continuous dependence on the initial data will differ significantly from the classical approach as consequence of the low regularity of the nonlinearity $\mathcal{N}$ combined with the lack of integrability of $\psi$. The constructed solutions are such that $\psi(t)-\psi_0\in H^1(\R^2)$ for all $t$ and hence \eqref{eq:diffNSobolev} suffices to show local existence. Note that in order to show the continuous dependence on the initial data \eqref{eq:diffNSobolev} is not sufficient as in general different initial data possess different far-field behavior, namely belongs to different connected components of $\Eb$, see also Remark \ref{rem: connected components 2D}. Lemma \ref{lem:CDOID} upgrades \eqref{eq:diffNSobolev} to the respective inequality for general initial data.
\\
The following Proposition is stated for positive existence times, the analogous statement for negative times follows from the time reversal symmetry of \eqref{eq:NLS}. 
\begin{proposition}\label{prop:localWP}
Let $d=2$ and $f$ be as in Assumption \ref{ass:N}.
\begin{enumerate}
    \item For any $\psi_0\in \Eb(\R^2)$, there exist $T=T(\En(\psi_0))>0$ and a unique strong solution $\psi\in C([0,T];\Eb(\R^2))$ to \eqref{eq:NLS} with $\psi(0)=\psi_0$. In particular, we have that $\psi-\psi_0\in C([0,T];H^1(\R^2))$;
    \item there exists a maximal existence time  $T^{\ast}=T^{\ast}(\psi_0)>0$, such that $\psi\in C([0,T^{\ast});\Eb(\R^2))$ and the blow-up alternative holds, namely if $T^{\ast}<\infty$ then
    \begin{equation*}
    \lim_{t\nearrow T^{\ast}}\En(\psi)(t)=+\infty.
    \end{equation*}
    \item For any $\psi_0^{\ast}\in\Eb(\R^2)$ there exists a open neighborhood $\mathcal O\subset\mathbb E(\R^2)$ of $\psi_0^{\ast}$ such that
    \begin{equation*}
        T^{\ast}(\mathcal{O})=\inf_{\psi_0\in\mathcal{O}}T^{\ast}(\psi_0)>0,
    \end{equation*}
    and the map $\psi_0^\ast\in \mathcal{O}\mapsto \psi\in C([0,T];\Eb(\R^2))$ is continuous for all $0<T<T^{\ast}(\mathcal{O})$. Moreover, let $\mathcal{O}_r=\{\psi_0\in \Eb(\R^2) : d_{\Eb}(\psi_0^{\ast},\psi_0)<r\}$, then 
    \begin{equation*}
        \liminf_{r\rightarrow 0} T^{\ast}(\mathcal{O}_r)\geq T^{\ast}(\psi_0^{\ast}).
    \end{equation*}
\end{enumerate}
\end{proposition}
Point (1) of Proposition \ref{prop:localWP} is included in (2). Nevertheless, it is stated separately as it proves useful for the proof of continuous dependence property in (3). 
\begin{proof}
\textbf{Local existence.}
We note that $\psi\in C([0,T];\Eb(\R^2))$ is a strong solution to \eqref{eq:NLS} with initial data $\psi_0\in\Eb(\R^2)$ if and only if 
\begin{equation*}
    \psi(t)=\eitD\psi_0-i\int_0^t\eul^{\frac{\imu}{2}(t-s)\Delta}\mathcal{N}(\psi)(s)\dd s
\end{equation*}
for all $t\in[0,T]$. To show the existence of a solution $\psi$ it suffices to implement a fixed-point argument for the solution map
\begin{equation}\label{eq:mapS}
    \Phi(u)(t)= i\int_0^t\eul^{\frac{\imu}{2}(t-s)\Delta}\mathcal{N}(\eul^{\frac{\imu}{2}s\Delta}\psi_0+u(s))\dd s.
\end{equation}
Specifically, let $\psi_0\in\mathbb E$ and $R>0$ such that $\mathcal E(\psi_0)\leq R$ and given $M>0$ and $T>0$, we consider the solution map \eqref{eq:mapS} defined on the function space
\begin{equation*}
    X_T=\left\{u\in C([0,T];H^1(\R^2)) \, : \, u(0)=0, \quad \|u\|_{X_T}\leq M\right\}.
\end{equation*}
For $u,v\in X_T$, we introduce the distance function $d$ as 
\begin{equation*}
    d_X(u,v)=\|u-v\|_{L^{\infty}([0,T];L^2(\R^2))}.
\end{equation*}
It is straightforward to verify that the space $(X_T,d_X)$ is a complete metric space.
\\
Note that $\psi(t)=\eitD\psi_0+u(t)$ satisfies $\psi\in C([0,T];\Eb(\R^2))$ for $u\in X_T$ and 
$\psi_0\in \Eb(\R^2)$. Indeed, it follows from Proposition \ref{prop:linear} that $\eitD\psi_0\in C([0,T];\Eb(\R^2))$ and Lemma \ref{lem:energyspace} yields that $\eitD\psi_0+u\in C([0,T];\Eb(\R^2))$. If $u$ is a fixed-point of \eqref{eq:solutionmap3Dsuper} then $\psi=\eitD\psi_0+u$ is a local strong solution of \eqref{eq:NLS}.
\\
If $\En(\psi_0)\leq R$ and $u\in X_T$, then thanks to the Minkowski inequality and \eqref{eq:freesolution} we obtain
\begin{equation}\label{eq:CMR}
    \Zt(\eitD\psi_0+u)\leq \Zt(\eitD\psi_0)+\|u\|_{L^{\infty}([0,T];H^1(\R^2))}\leq 2 \sqrt{2R}+M,
\end{equation}
provided that $T>0$ sufficiently small.
Next, we show that $\Phi$ defined in \eqref{eq:mapS} maps $X_T$ onto $X_T$. Let $u\in X_T$ and denote $\psi=\eitD\psi_0+u$, then
by virtue of the Strichartz estimate \eqref{eq:Strichartznonh},
\eqref{eq:boundN} and \eqref{eq:CMR} we obtain
\begin{multline}\label{eq:Su1}
    \|\Phi(u)\|_{L^{\infty}([0,T];L^2(\R^2))}\leq \|\mathcal{N}(\psi)\|_{L^1([0,T];L^2(\R^2))}\\ \leq CT\left({\Zt(\psi)}+\Zt(\psi)^{1+2\alpha}\right)\leq C T\left(1+(2\sqrt{2R}+M)^{2\alpha}\right)(2\sqrt{2R}+M).
\end{multline}
To bound $\nabla \Phi(u)$, we apply the Strichartz estimate \eqref{eq:Strichartznonh} concatenated with \eqref{eq:boundgradN} to obtain
\begin{equation}\label{eq:Su2}
\begin{aligned}
        &\|\nabla \Phi(u)\|_{L^{\infty}([0,T];L^2(\R^2))}\leq C 
        \|\nabla\mathcal{N}(\psi)\|_{N^0([0,T]\times\R^2)}\\
        &\leq
        C\left(T+T^{\frac{1}{q_1'}}\Zt(\psi)^{2\alpha}\right)\|\nabla\psi\|_{L_t^{\infty}L_x^2}
        \leq C
      \left(T+T^{\frac{1}{q_1'}}(2\sqrt{2R}+M)^{2\alpha}\right)(2\sqrt{2R}+M).
      \end{aligned}
\end{equation}
We conclude that 
\begin{equation*}
    \Phi(u)\in C([0,T];H^1(\R^2)),
\end{equation*}
and summing up \eqref{eq:Su1} and \eqref{eq:Su2}, we obtain that
\begin{equation*}
    \|\Phi(u)\|_{X_T}
    \leq 
    C\left(T+T^{\frac{1}{q_1'}}(2\sqrt{2R}+M)^{2\alpha}\right)(2\sqrt{2R}+M).
\end{equation*}
 Next, we check that the map $\Phi$ defines a contraction on $(X_T,d_X)$. Let $u_1,u_2\in X_T$ and denote
 \begin{equation*}
 \psi_1=\eitD\psi_0+u_1, \qquad  \psi_2=\eitD\psi_0+u_2.   
 \end{equation*}
Upon applying \eqref{eq:Strichartznonh} followed by \eqref{eq:diffNSobolev} one has
\begin{equation*}
    \begin{aligned}
        d_X(\Phi(u_1),\Phi(u_2))&=\left\|-i\int_0^t\eul^{\frac{\imu}{2}(t-s)\Delta}\left(\mathcal{N}(\psi_1)-\mathcal{N}(\psi_2)\right)(s)\dd x\right\|_{L^{\infty}([0,T],L^2(\R^2))}\\
        &\leq C \left\|\mathcal{N}(\psi_1)-\mathcal{N}(\psi_2)\right\|_{N^0([0,T]\times\R^2)}\\
        &\leq C \left(T+T^{\frac{1}{q_1'}}(2\sqrt{2R}+M)^{2\alpha}\right)d_X(u_1,u_2).
    \end{aligned}
\end{equation*}
We fix $M=\sqrt{2R}$ and notice that there exists a sufficiently small $0<T\leq 1$ small such that 
\begin{equation*}
    C \left(T+T^{\frac{1}{q_1'}}(3\sqrt{2R})^{2\alpha}\right)\leq \frac{1}{3}.
\end{equation*}
Hence, $\Phi$ maps $X_T$ onto $X_T$ and defines a contraction on $X_T$. The Banach fixed-point Theorem yields a unique $u\in X_T$ such that $\eitD\psi_0+u$ is solution to \eqref{eq:Duhamel}. It follows from Lemma \ref{lem:energyspace} and \eqref{eq:freesolution} that $\eitD\psi_0+u\in C([0,T];\Eb(\R^2))$. In particular, $\psi-\psi_0\in C([0,T];H^1(\R^2))$ from \eqref{eq:linearflow1} and $u\in X_T$.
\\
\textbf{Uniqueness.}
Let $\psi_1,\psi_2\in C([0,T],\Eb(\R^2))$ be two solutions to \eqref{eq:NLS} with initial data $\psi_1(0)=\psi_2(0)=\psi_0\in \Eb(\R^2)$. One has that 
\begin{equation}\label{eq:diffuniqueness}
    \psi_1(t)-\psi_2(t)=-i\int_0^t \eul^{\frac{\imu}{2}(t-s)\Delta}\left(\mathcal{N}(\psi_1)-\mathcal{N}(\psi_2)\right)(s)\dd s.
\end{equation}
In particular, as the nonlinear terms are bounded in $L_t^{\infty}H_x^1(\R^2)$, one has $\psi_1-\psi_2\in L^{\infty}([0,T];H^1(\R^2))$. For $(q_1',r_1')$ given by \eqref{eq:q1r1}, the Strichartz estimate \eqref{eq:Strichartznonh} together with \eqref{eq:diffNSobolev} then yields
\begin{align*}
        &\|\psi_1-\psi_2\|_{L_t^{\infty}L_x^2}
        \leq C \|\mathcal{N}(\psi_1)-\mathcal{N}(\psi_2)\|_{N^0([0,T]\times\R^2)}\\
        &\leq C\left(T+T^{\frac{1}{q_1'}}\left(\Zt(\psi_1)^{2\alpha}+\Zt(\psi_2)^{2\alpha}\right)\right)\|\psi_1-\psi_2\|_{L_t^{\infty}L_x^2}.
\end{align*}
Hence, we deduce that there exists $T_1>0$ such that $\psi_1=\psi_2$ a.e. on $[0,T_1]\times\R^2$. As $T_1$ only depends on $\Zt(\psi_1)$, $\Zt(\psi_2)$, one may iterate the argument to obtain uniqueness of the solution on the interval $[0,T]$.
\\
\textbf{Blow-up alternative.} 
Let $\psi_0\in\Eb(\R^2)$ and define
\begin{equation*}
    T^{\ast}(\psi_0)=\sup\left\{T>0 \, : \, \text{there exists a solution to \eqref{eq:NLS} on}\, [0,T]\right\}.
\end{equation*}
Let $ T^{\ast}(\psi_0)<+\infty$ and assume that there exist $R>0$ and a sequence $\{t_n\}_{n\in \N}$ such that $t_n\rightarrow T^{\ast}(\psi_0)$ and $\En(\psi(t_n))\leq R$ for all $n\in\N$. Then, there exists a sufficiently large  $n$ such that the local existence statement allows us to uniquely extend the solution to $[0,t_n+T(R)]$ with $t_n+T(R)> T^{\ast}(\psi_0)$. This violates the maximality assumption and we conclude that 
\begin{equation*}
    \En(\psi(t_n))\rightarrow \infty, \qquad \text{as} \quad t_n\rightarrow  T^{\ast}(\psi_0),
\end{equation*}
if $ T^{\ast}(\psi_0)<+\infty$.\\
 The proof of the continuous dependence on the initial data of the solution requires some auxiliary statements and is postponed after Lemma \ref{coro:nonlinear flow}.
 \end{proof}
We introduce estimates on the nonlinear flow in Strichartz norms that are required for the proof of the continuous dependence on the initial data. The estimates used for the local existence and uniqueness in the proof of Proposition \ref{prop:localWP} are not sufficient since they only allow to control the difference of solutions $\psi_1,\psi_2$ provided that $\psi_1-\psi_2\in L^{\infty}([0,T];L^2(\R^2))$. In addition, as the regularity properties of $\mathcal{N}$ do not suffice to control $\|\nabla\Phi(\psi_1)-\nabla\Phi(\psi_2)\|_{L_t^{\infty}L_x^2}$ for $\psi_1,\psi_2\in C([0,T];\Eb(\R^2))$, we need to rely on a auxiliary metric.
\begin{lemma}\label{lem:CDOID}
Let $f$ satisfy Assumption \ref{ass:N}, $T>0$, $(q_1',r_1')$ as defined in \eqref{eq:q1r1} and $\psi_1,\psi_2\in C([0,T];\Eb(\R^2))$. Then, there exists $\theta\in(0,1]$ such that
\begin{multline*}
    \left\|\mathcal{N}(\psi_1)-\mathcal{N}(\psi_2)\right\|_{N^0([0,T]\times\R^2)}\\
    \leq C T^{\theta}\left(1+\Zt(\psi_1)+\Zt(\psi_2)+\Zt(\psi_1)^{2\alpha}+\Zt(\psi_2)^{2\alpha}\right)\\
    \times\left(\left\||\psi_1|-|\psi_2|\right\|_{L^{2}([0,T]];L^2(\R^2))}+
    \|\psi_1-\psi_2\|_{L^{2}([0,T];L^{\infty}+L^2(\R^2))}\right).
\end{multline*}
\end{lemma}
\begin{proof}
First, we notice that from the first inequality of \eqref{eq:pointwise Nb} and the decomposition provided by Lemma \ref{lem:energyspace}, it follows
\begin{multline*}
    \left\|\Na(\psi_1)-\Na(\psi_2)\right\|_{L_t^1L_x^2+L_t^{\frac43}L_x^{\frac43}}
    \leq C\left(T^{\frac12}+T^{\frac14}\Zt(\psi_1)\right) \left\||\psi_1|-|\psi_2|\right\|_{L_t^{2}L_x^2}\\
    +CT^{\frac12}(1+\Zt(\psi_2))\|\psi_1-\psi_2\|_{L_t^{\infty}(L_x^{\infty}+L_x^2)},
\end{multline*}
where we used that $||\psi_2|-1|\eta(|\psi_2|)\in L^{\infty}([0,T];L^{\infty}(\R^2)\cap L^2(\R^2))$.
Indeed, let the set $\Omega\subset \R^2$ of finite Lebesgue measure and $f\in L^{\infty}(\Omega)+L^p(\Omega)$, then
\begin{equation*}
    \|f\|_{L^p(\Omega)}\leq C\left(1+\mathcal{L}^2(\Omega)^{\frac{1}{p}}\right)\|f\|_{L^p(\Omega)+L^{\infty}(\Omega)}.
\end{equation*}
Second, we observe that $\mathcal{L}^2(\supp(\Nb(\psi_i)))\leq \En(\psi_i)$ for $i=1,2$ from \eqref{eq:Cheby}. From \eqref{eq:pointwise Nb}, we conclude
\begin{equation*}
    \|\Nbl(\psi_1)-\Nbl(\psi_2)\|_{L_t^1L_x^2}\leq CT(1+{\Zt(\psi_1)}+{\Zt(\psi_2)}) \left\|\psi_1-\psi_2\right\|_{L_t^{\infty}(L_x^{\infty}+L_x^2)}.
\end{equation*}
Third, arguing as in the proof of Lemma \ref{lem:N} and using $\mathcal{L}^2(\supp(\Nb(\psi_i)))\leq \En(\psi_i)$ we obtain
\begin{multline*}
     \|\Nbh(\psi_1)-\Nbh(\psi_2)\|_{L_t^1L_x^2+L_t^{q_1'}L_x^{r_1'}}\leq \left\|\mathbf{1}_{\supp(1-\chi(\psi_1))\cup \supp(1-\chi(\psi_2))}|\psi_1-\psi_2|\right\|_{L_t^1L_x^2}\\+\left\|\left(|\psi_{1,\mathrm{int}}|^{2\alpha}+|\psi_{2,\mathrm{int}}|^{2\alpha}\right)|\psi_1-\psi_2|\right\|_{L_t^{q_1'}L_x^{r_1'}}
     \\\leq C(T+T^{\frac{1}{q_1'}})\left({\Zt(\psi_1)}+{\Zt(\psi_2)}+\Zt(\psi_1)^{2\alpha}+\Zt(\psi_2)^{2\alpha}\right)\|\psi_1-\psi_2\|_{L_t^{\infty}(L_x^{\infty}+L_x^2)}.
\end{multline*}
\end{proof}
Concatenating the Strichartz estimate \eqref{eq:Strichartznonh} and Lemma \ref{lem:CDOID} gives the following.
\begin{lemma}\label{coro:nonlinear flow}
Given $\psi_1,\psi_2\in C([0,T];\Eb(\R^2))$ such that $\Zt(\psi_i)\leq M$ for $i=1,2$, there exist $C=C(M)>0$ and $\theta\in(0,1]$ such that 
\begin{equation}\label{eq:Ndiff}
\|\Phi(\psi_1)-\Phi(\psi_2)\|_{S^0([0,T]\times \R^2)}
\leq C_M T^{\theta}\left(\|\psi_1-{\psi_2}\|_{L_t^{\infty}(L_x^{\infty}+L_x^2)}+\||\psi_1|-|\psi_2|\|_{L_t^{2}L_x^2}\right).
\end{equation}
\end{lemma}
We are now in position to complete the proof of Proposition \ref{prop:localWP}. Note that the metric space $(\Eb,d_{\Eb})$ is not separable, see also Remark \ref{rem:topology}. In particular, it is not sufficient to show sequential continuity of the solution map.
\begin{proof}[Proof of Proposition \ref{prop:localWP} continued.]
We prove \textbf{continuous dependence on the initial data}.
Given $\psi_0^{\ast}\in \Eb(\R^2)$, let $R:=\En(\psi_0^{\ast})$ and $r\in (0,\sqrt{R}]$. Denote
\begin{equation}\label{eq:nhd}
    \mathcal{O}_r:=\{\psi_0\in \Eb(\R^2) \, : \, d_{\Eb}(\psi_0^{\ast},\psi_0)<r\}.
\end{equation}
If follows that $\En(\psi_0)\leq 4 \En(\psi_0^{\ast})$ for all $\psi_0\in \mathcal{O}_r$.
The first statement of Proposition \ref{prop:localWP} then yields that there exists $T=T(4\En(\psi_0^{\ast}))>0$ such that for all $\psi_0\in \mathcal{O}_r$ there exists a unique strong solution $\psi\in C([0,T];\Eb(\R^2))$. In particular, for $\psi_0\in \mathcal{O}_r$ the maximal time satisfies
\begin{equation*}
    T^{\ast}(\psi_0)\geq T(4\En(\psi_0^{\ast}))>0
\end{equation*}
by virtue of the blow-up alternative.
Hence, 
\begin{equation*}
    T^{\ast}(\mathcal{O}_r)=\inf_{\psi_0\in \mathcal{O}_r}T^{\ast}(\psi_0)\geq T(4\En(\psi_0^{\ast}))>0.
\end{equation*}
Given $\delta>0$ to be chosen later, let $\mathcal O_\delta$ be defined as in \eqref{eq:nhd}. Let us remark (again) that, for any $\psi_0\in\mathcal O_\delta$, we have $\mathcal E(\psi_0)\leq 2(R+\delta^2)$.  In particular, $T^{\ast}(\psi_0)\geq T^{\ast}(\mathcal{O}_{\delta})>0$.
\\
Let $\psi_0^1\in \mathcal{O}_{\delta}$ and denote by $\psi^{\ast},\psi_1$ the respective solutions with initial data $\psi_0^{\ast}, \psi_0^1$ defined at least up to time $T^{\ast}(\mathcal{O}_{\delta})$. For any $0<T<T^{\ast}(\mathcal{O}_{\delta})$ there exists $M=M(T)>0$ such that
$\Zt(\psi_1)\leq M$ by virtue of the blow-up alternative. From \eqref{eq:stabilitylinsol}, we have that there exists $C=C(R, \delta, T)>0$ such that
\begin{equation}\label{eq:CDlinear}
    \sup_{t\in[0,T]}d_{\Eb}(\eitD\psi_0^1,\eitD\psi_0^{\ast})\leq C d_{\Eb}(\psi_0^1,\psi_0^{\ast})\leq 2C\delta.
\end{equation}
To prove continuous dependence of the solution, we proceed in the following four steps that compensate for the lack of local Lipschitz regularity of $\nabla\mathcal{N}$ that in general does not hold under Assumption \ref{ass:N}.
\begin{enumerate}
     \item There exist $C>0$ and $0<T_1<T^{\ast}(\mathcal{O}_{\delta})$, only depending on $M$ such that
    \begin{multline}\label{eq:CD1}
        \|\psi_1-\psi^{\ast}\|_{L^{\infty}([0,T_1];L^{\infty}+L^2(\R^2))}
        +\||\psi_1|-|\psi^{\ast}|\|_{L^{2}([0,T_1];L^2(\R^2))}\\
        \leq C d_{\Eb}(\psi_0^1,\psi_0^{\ast}).
    \end{multline} 
    \item Provided \eqref{eq:CD1} holds and arguing by contradiction, we show that for all $\eps>0$ there exist $T_2=T_2(M)>0$ and $\delta>0$ such that {$d_{\Eb}(\psi_0^1,\psi_0^{\ast})<\delta$} implies
    \begin{equation}\label{eq:CD2}
        \|\nabla\psi_1-\nabla\psi^{\ast}\|_{L^{\infty}([0,T_2];L^2(\R^2))}<\eps.
    \end{equation}
    \item Provided \eqref{eq:CD1} and \eqref{eq:CD2} hold, for all $\eps>0$ there exists $\delta>0$ such that $d_{\Eb}(\psi_0^1,\psi_0^{\ast})<\delta$ implies 
    \begin{equation}\label{eq:CD3}
        \sup_{t\in[0,T_2]}d_{\Eb}(\psi_1(t),\psi^{\ast}(t))<\eps.
    \end{equation}
    \item By iterating \eqref{eq:CD3}, we prove that for all $0<T<T^{\ast}(\mathcal{O}_{\delta})$ and $\eps>0$, there exists $\delta>0$ such that $d_{\Eb}(\psi_0^1,\psi_0^{\ast})<\delta$ yields 
    \begin{equation}\label{eq:CD4}
        \sup_{t\in[0,T]}d_{\Eb}(\psi_1(t),\psi^{\ast}(t))<\eps,
    \end{equation}
\end{enumerate}
\textbf{Step 1} We show \eqref{eq:CD1}.
Let us consider the first term on the left hand side of \eqref{eq:CD1}. By using \eqref{eq:CDlinear} and from Lemma \ref{coro:nonlinear flow}, we know there exists $\theta>0$ such that
\begin{equation}\label{eq:Step1a}
\begin{aligned}
    &\|\psi_1-\psi^{\ast}\|_{L^{\infty}([0,T];L^{\infty}+L^2(\R^2))}\\
    &\leq \|\eitD\psi_0^1-\eitD\psi_0^{\ast}\|_{L^{\infty}([0,T];L^{\infty}+L^2(\R^2))}+\left\|\Phi(\psi_1)-\Phi(\psi^{\ast})\right\|_{L^{\infty}([0,T],L^{2}(\R^2))}\\
    &\leq C d_{\Eb}(\psi_0^1,\psi_0^{\ast})+C_{M} T^{\theta}\left(\|\psi_1-{\psi^{\ast}}\|_{L^{\infty}([0,T];(L^{\infty}+L^2(\R^2))}+\||\psi_{1}|-|\psi^{\ast}|\|_{L^{2}([0,T];L^2(\R^2)}\right).
    \end{aligned}
\end{equation}
Given $\chi$ satisfying \eqref{eq:cutoffchi}, we define $\chi_6(z):=\chi(6z)$. Arguing as in the proof of Lemma \ref{lem:metric} we notice that
\begin{multline}\label{eq:Step1int}
    \left\||\psi_1|-|\psi^{\ast}|\right\|_{L^2([0,T];L^2(\R^2))}\\
    \leq \left\||\psi_1|^2-|\psi^{\ast}|^2\right\|_{L^2([0,T];L^2(\R^2))}+\|\psi_1\chi_6(\psi_1)-\psi^{\ast}\chi_6(\psi^{\ast})\|_{L^2([0,T];L^2(\R^2))}
\end{multline}
To deal with the first contribution on the right-hand side, we notice that
\begin{multline*}
        \left||\psi_1|^2-|\psi^{\ast}|^2\right|\leq \left||\eitD\psi_0^{1}|^2-|\eitD\psi_0^{\ast}|^2\right|+\left|2\RE\left(\eul^{-\frac{\imu}{2}t\Delta}\overline{\psi_0^{\ast}}\left(\Phi(\psi^{\ast})-\Phi(\psi_1)\right)\right)\right|\\
        +\left|2\RE\left(\eul^{-\frac{\imu}{2}t\Delta}(\overline{\psi_0^{\ast}}-\overline{\psi_0^1})\Phi(\psi_1)\right)\right|
        +\left(|\Phi(\psi_1)|+|\Phi(\psi^{\ast})|\right)\left|\Phi(\psi_1)-\Phi(\psi^{\ast})\right|.
\end{multline*}
We control these four terms separately. First, from \eqref{eq:CDlinear}, one has that 
\begin{equation*}
    \begin{aligned}
        \left\||\eitD\psi_0^{1}|^2-|\eitD\psi_0^{\ast}|^2\right\|_{L_t^{2}L_x^2}\leq C T^{\frac{1}{2}}d_{\Eb}(\psi_0^1,\psi_0^{\ast}).
    \end{aligned}
\end{equation*}
Second, upon splitting $\eitD\psi_0^i\in \Eb(\R^2)$ as in \eqref{eq:psihighlow} we have
\begin{align*}
        &\left\|2\RE\left(\eul^{-\frac{\imu}{2}t\Delta}\overline{\psi_0^{\ast}}\left(\Phi(\psi^{\ast})-\Phi(\psi_1)\right)\right)\right\|_{L_t^{2}L_x^{2}}\\
        &\leq T^{\frac{1}{2}}\|\Phi(\psi^{\ast})-\Phi(\psi_1)\|_{L_t^{\infty}L_x^2}
        +T^{\frac{1}{4}}{\Zt(\eitD\psi_0^{\ast})}\|\Phi(\psi^{\ast})-\Phi(\psi_1)\|_{L_t^{4}L_x^4}\\
      &\leq CM\left(T^{\frac12}\|\Phi(\psi^{\ast})-\Phi(\psi_1)\|_{L_t^{\infty}L_x^2} +T^{\frac{1}{4}}\|\Phi(\psi_1)-\Phi(\psi^{\ast})\|_{L_t^{4}L_x^4}\right).
\end{align*}
Third, proceeding similarly and exploiting \eqref{eq:CDlinear} we have
\begin{align*}
        &\left\|2\RE\left(\eul^{-\frac{\imu}{2}t\Delta}(\overline{\psi_0^{\ast}}-\overline{\psi_0^1})\Phi(\psi_1)\right)\right\|_{L_t^{2}L_x^2}\\
        &\leq C\left(T^{\frac{1}{2}}\|\Phi(\psi_1)\|_{L_t^{\infty}L_x^2}+T^{\frac{1}{4}}\|\Phi(\psi^{\ast})\|_{L_t^{4}L_x^4}\right)d_{\Eb}(\eitD\psi_0^1,\eitD\psi_0^{\ast})
        \\
        &\leq C\left(T^{\frac{1}{2}}\|\Phi(\psi_1)\|_{L_t^{\infty}L_x^2}+T^{\frac{1}{4}}\|\Phi(\psi_1)\|_{L_t^{4}L_x^4}\right)d_{\Eb}(\psi_0^1,\psi_0^{\ast})\\
        &\leq C(T^{\frac12}+T^{\frac14})(M+M^{1+2\alpha})d_{\Eb}(\psi_0^1,\psi_0^{\ast}),
\end{align*}
where we used that $\Phi(\psi_1)\in L^{\infty}([0,T];L^2(\R^2))\cap L^{4}([0,T];L^4(\R^2))$ from \eqref{eq:Strichartznonh}.\\
Fourth, one has
\begin{align*}
    &\|\left(|\Phi(\psi_1)|+|\Phi(\psi^{\ast})|\right)\left|\Phi(\psi^{\ast})-\Phi(\psi_1)\right|\|_{L_t^{2}L_x^2}\\
    &\leq\left(\|\Phi(\psi_1)\|_{L_t^4L_x^4}+\|\Phi(\psi^{\ast})\|_{L_t^4L_x^4}\right)\|\Phi(\psi_1)-\Phi(\psi^{\ast})\|_{L_t^4L_x^4}\\
    &\leq CT\left(M+M^{1+2\alpha}\right) \|\Phi(\psi_1)-\Phi(\psi^{\ast})\|_{L_t^4L_x^4},
\end{align*}
where we used \eqref{eq:boundN} in the last inequality.
Combining the previous inequalities, we infer that there exists $\theta_1>0$ such that
\begin{multline}\label{eq:bound1}
   \left\||\psi_1|^2-|\psi^{\ast}|^2\right\|_{L_t^2L_x^2}\leq CT^{\theta_1}\left(1+{M}+M^{1+2\alpha}\right)\\
    \times\left(d_{\Eb}(\psi_0^1,\psi_0^{\ast})+\|\Phi(\psi_1)-\Phi(\psi^{\ast})\|_{L_t^{\infty}L_x^2}+\|\Phi(\psi_1)-\Phi(\psi^{\ast})\|_{L_t^4L_x^4}\right).    
\end{multline}
The second contribution in \eqref{eq:Step1int} is bounded as follows
\begin{equation}\label{eq:bound2}
\begin{aligned}
    &\|\psi_1\chi_6(\psi_1)-\psi^{\ast}\chi_6(\psi^{\ast})\|_{L_t^2L_x^2}\\
    &\leq CT^{\frac{1}{2}}\left(1+M\right)
    d_{\Eb}(\eitD\psi_0^1,\eitD\psi_0^{\ast})+CT^{\frac{1}{2}}\|\Phi(\psi_1)-\Phi(\psi^{\ast})\|_{L_t^{\infty}L_x^2}\\
    &\leq C T^{\frac{1}{2}}\left(1+M\right)\left(d_{\Eb}(\psi_0^1,\psi_0^{\ast})+\|\Phi(\psi_1)-\Phi(\psi^{\ast})\|_{L_t^{\infty}L_x^2}\right),
    \end{aligned}
\end{equation}
where we exploited that for $\psi\in\Eb(\R^2)$ the measure of the support of $\chi_6(\psi)$ is bounded by $\En(\psi)$, see \eqref{eq:Cheby}. 
It follows from \eqref{eq:Step1int}, \eqref{eq:bound1} and \eqref{eq:bound2} that there exists $\theta_2>0$ such that
\begin{multline}\label{eq:Step1c}
    \left\||\psi_1|-|\psi^{\ast}|\right\|_{L^2([0,T];L^2(\R^2))}
    \leq CT^{\theta_2}\left(1+M+M^{1+2\alpha}\right)\\
    \times\Big(d_{\Eb}(\psi_0^1,\psi_0^{\ast})+\|\Phi(\psi_1)-\Phi(\psi^{\ast})\|_{L^{\infty}L^2}+\|\Phi(\psi_1)-\Phi(\psi^{\ast})\|_{L_t^4L_x^4}\Big).
\end{multline}
Summing up \eqref{eq:Step1a} and \eqref{eq:Step1c} and applying \eqref{eq:Ndiff} yields that there exists $\theta>0$ such that
\begin{multline*}
\|\psi_1-\psi^{\ast}\|_{L_t^{\infty}(L_x^{\infty}+L_x^2)}
        +\left\||\psi_1|-|\psi^{\ast}|\right\|_{L_t^{2}L_x^2}
        \leq C_M T^{\theta}\\
        \times\left(d_{\Eb}(\psi_0^1,\psi_0^{\ast})+C_MT^{\theta}\left(\|\psi_1-\psi^{\ast}\|_{L_t^{\infty}(L_x^{\infty}+L_x^2)}
        +    \||\psi_1|-|\psi^{\ast}|\|_{L_t^{2}L_x^2}\right)\right).    
\end{multline*}
For $T_1>0$ sufficiently small, depending only on $M$, inequality \eqref{eq:CD1} follows.
\\
\textbf{Step 2.} 
Provided that \eqref{eq:CD1} holds, note that
\begin{equation*}
    \nabla\psi_1-\nabla\psi^{\ast}=\eitD\left(\nabla\psi_0^1-\nabla\psi_0^{\ast}\right)-i\int_0^t\eul^{\frac{\imu}{2}(t-s)\Delta}\left(\nabla\mathcal{N}(\psi_1)-\nabla\mathcal{N}(\psi^{\ast})\right)(s)\dd s.
\end{equation*}
We estimate the difference of the free solutions by
\begin{equation}\label{eq:diffgradlCD}
    \left\|\eitD\left(\nabla\psi_0^1-\nabla\psi_0^{\ast}\right)\right\|_{L^{\infty}([0,T],L^2(\R^2))}\leq d_{\Eb}(\psi_0^1,\psi_0^{\ast}),
\end{equation}
exploiting that $\eitD$ is an isometry on $L^2(\R^2)$. We recall from \eqref{eq:nablaN} that
\begin{equation*}
   \nabla\mathcal{N}(\psi)=\left(f(|\psi|^2)+f'(|\psi|^2)|\psi|^2\right)\nabla\psi+f'(|\psi|^2)\psi^2\overline{\nabla\psi},
\end{equation*}
which can be bounded by means of \eqref{eq:N_locLip} as
\begin{equation*}
    \left|\nabla \mathcal{N}(\psi)\right|\leq C (1+|\psi|^{2\alpha})|\nabla \psi|\leq C(1+|\psi_{\mathrm{int}}|^{2\alpha})|\nabla\psi|.
\end{equation*}
We apply estimate \eqref{eq:Strichartznonh} to the non-homogeneous term, where $(q_1,r_1))=(\frac{2(\alpha+1)}{\alpha},2(\alpha+1))$, see also \eqref{eq:q1r1}. We decompose $\nabla\mathcal{N}(\psi_1)-\nabla\mathcal{N}(\psi^{\ast})$ by means of the functions $G_{\mathrm{bd}}, G_{\mathrm{int}}$ defined in \eqref{eq: Ginfty Gq} leading to 
\begin{equation}\label{eq:diffgradnlCD}
    \begin{aligned}
    &\left\|i\int_0^t\eul^{\frac{\imu}{2}(t-s)\Delta}\left(\nabla\mathcal{N}(\psi_1)-\nabla\mathcal{N}(\psi^{\ast})\right)(s)\dd s\right\|_{L^{\infty}([0,T];L^2(\R^2))}\\
    &\leq \left\| (G_{\mathrm{bd}}+G_{\mathrm{int}})(\psi_1)\left|\nabla\psi_1-\nabla\psi^{\ast}\right| \right\|_{N^0}\\
    &\qquad +\left\|\left(  (G_{\mathrm{bd}}+G_{\mathrm{int}})(\psi_1)-  (G_{\mathrm{bd}}+G_{\mathrm{int}})(\psi^{\ast})\right)|\nabla\psi^{\ast}|\right\|_{N^0([0,T]\times\R^2))}\\
    &\leq \|\nabla\psi_1-\nabla\psi^{\ast}\|_{L_t^1L_x^2}+\||\psi_{1,\mathrm{int}}|^{2\alpha}\left|\nabla\psi_1-\nabla\psi^{\ast}\right|\|_{L_t^{q_1'}L_x^{r_1'}}\\
    &+\left\|\left(G_{\mathrm{bd}}(\psi_1)-G_{\mathrm{bd}}(\psi^{\ast})\right)|\nabla\psi^{\ast}|\right\|_{L_t^1L_x^2}
    +\left\|\left(G_{\mathrm{int}}(\psi_1)-G_{\mathrm{int}}(\psi^{\ast})\right)|\nabla\psi_1|\right\|_{L_t^{q_1'}L_x^{r_1'}}\\
   & \leq C\left(T+T^{\frac{1}{q_1'}}\Zt(\psi^{\ast})^{2\alpha})\right)\|\nabla\psi_1-\nabla\psi^{\ast}\|_{L_t^{\infty}L_x^2}\\
 &+\left\|\left(G_{\mathrm{bd}}(\psi_1)-G_{\mathrm{bd}}(\psi^{\ast})\right)|\nabla\psi^{\ast}|\right\|_{L_t^1L_x^2}
    +\left\|\left(G_{\mathrm{int}}(\psi_1)-G_{\mathrm{int}}(\psi^{\ast})\right)|\nabla\psi^{\ast}|\right\|_{L_t^{q_1'}L_x^{r_1'}}\\
    \end{aligned}
\end{equation}
Thus, for $T_2=T_2(M)>0$ sufficiently small so that 
\begin{equation*}
    C\left(T_2+T_2^{\frac{1}{q'}}\Zt(\psi_1)^{2\alpha}\right)\leq C\left(T_2+T_2^{\frac{1}{q'}}M^{2\alpha}\right)\leq \frac{1}{2},
\end{equation*}
we conclude by combining \eqref{eq:diffgradlCD} and \eqref{eq:diffgradnlCD} that
\begin{multline*}
    \| \nabla\psi_1-\nabla\psi^{\ast}\|_{L^{\infty}([0,T_2],L^2(\R^2))}\leq d_{\Eb}(\psi_0^1,\psi_0^{\ast})\\
  +\left\|\left(G_{\mathrm{bd}}(\psi_1)-G_{\mathrm{bd}}(\psi^{\ast})\right)|\nabla\psi^{\ast}|\right\|_{L_t^1L_x^2}
    +\left\|\left(G_{\mathrm{int}}(\psi_1)-G_{\mathrm{int}}(\psi^{\ast})\right)|\nabla\psi^{\ast}|\right\|_{L_t^{q_1'}L_x^{r_1'}}.
\end{multline*}
In order to conclude Step 2, we need to show that the second line above can be made arbitrarily small by choosing a sufficiently small $\delta>0$. We proceed by contradiction, assuming that there exist $\eps>0$, a sequence $\{\delta_n\}_{n\in \N}$ and $\{\psi_0^n\}_{n\in \N}\subset\Eb(\R^2)$ such that $d_{\Eb}(\psi_0^{\ast},\psi_0^n)<\delta_n\rightarrow 0$ and for all $n$ sufficiently large,
\begin{equation}\label{eq:Step2contradiction}
    \left\|\left(G_{\mathrm{bd}}(\psi^\ast)-G_{\mathrm{bd}}(\psi_n)\right)|\nabla\psi^{\ast}|\right\|_{L_t^1L_x^2}
    +\left\|\left(G_{\mathrm{int}}(\psi^\ast)-G_{\mathrm{int}}(\psi_n)\right)|\nabla\psi^{\ast}|\right\|_{L_t^{q_1'}L_x^{r_1'}}\geq \eps,
\end{equation}
where $\psi_n\in C([0,T];\Eb(\R^2))$ denotes the unique maximal solution with $\psi_n(0)=\psi_0^n$.
Inequality \eqref{eq:CD1} implies that, up to extracting a subsequence, {not relabeled}, $\psi_n$ converges to $\psi^{\ast}$ a.e. on $[0,T_1]\times \R^2$. If $0<T_1<T_2$, then set $T_2:=T_1$. By virtue of Assumption \ref{ass:N} on $f$, it follows that $G_{\mathrm{bd}}, G_{\mathrm{int}}$ are continuous and thus
\begin{equation*}
    \begin{aligned}
        \left|\left(G_{\mathrm{bd}}(\psi^{\ast})-G_{\mathrm{bd}}(\psi_n)\right)\right|\,|\nabla\psi^{\ast}|&\rightarrow 0 \quad \text{a.e. in} \quad [0,T_2]\times\R^2,\\
            |G_{\mathrm{int}}(\psi^{\ast})-G_{\mathrm{int}}(\psi_n)|\,|\nabla\psi^{\ast}|&\rightarrow 0 \quad \text{a.e. in} \quad [0,T_2]\times\R^2.\\
    \end{aligned}
\end{equation*}
Since in addition one has
\begin{align*}
    \left\|G_{\mathrm{int}}(\psi_n)\right\|_{L_t^{\infty}L_x^{q_1}(\R^2)}&\leq C  \left\|(1+|\psi_{q,n}|^{2\alpha})(1-\chi(\psi_n)\right\|_{L_t^{\infty}L_x^{q_1}(\R^2)}\\
    &\leq C\left({\Zt(\psi_n)}+\Zt(\psi_n)^{2\alpha}\right)\leq C\left({M}+M^{2\alpha}\right)
\end{align*}
for all $n\in \N$, we obtain from \eqref{eq:CD1} that there exists $\phi\in L^{\infty}([0,T];L^{r_1}(\R^2))$ such that $|\psi_{q,n}|\leq \phi$ a.e. on $[0,T_2)\times\R^2$. Therefore,
\begin{equation*}
    \begin{aligned}
        &\left|\left(G_{\mathrm{bd}}(\psi^{\ast})-G_{\mathrm{bd}}(\psi_n)\right)\right|\, |\nabla\psi^{\ast}|\leq C |\nabla\psi^{\ast}|\in L^1([0,T);L^2(\R^2)),\\
            &\left|\left(G_{\mathrm{int}}(\psi^{\ast})-G_{\mathrm{int}}(\psi_n)\right)\right| \, |\nabla\psi^{\ast}|\leq C\left(|\psi^{\ast}|^{2\alpha}+|\phi|^{2\alpha}\right) |\nabla\psi^{\ast}|\in L^{q_1'}([0,T);L^{r_1'}(\R^2)),\\
    \end{aligned}
\end{equation*}
so that the dominated convergence Theorem implies that \eqref{eq:Step2contradiction} is violated. The inequality \eqref{eq:CD2} follows for the time interval $[0,T_2]$ where we stress that $T_2>0$ only depends on $M$.
\newline\textbf{Step 3.} 
Given that \eqref{eq:CD1} and \eqref{eq:CD2} are satisfied, it suffices to prove that, for any $\eps>0$, there exists $\delta>0$ such that  $d_{\Eb}(\psi_0^1,\psi_0^{\ast})<\delta$ implies
\begin{equation*}
     \left\||\psi_1|-|\psi^{\ast}|\right\|_{L^{\infty}([0,T_2];L^2(\R^2))}<\eps.
\end{equation*}
Note that \eqref{eq:CD1} only yields 
\begin{equation*}
     \left\||\psi_1|-|\psi^{\ast}|\right\|_{L^{2}([0,T_2];L^2(\R^2))}< C \delta.
\end{equation*}
We recall that $\psi(t)=\eitD\psi_0+\Phi(\psi)$, where $\eitD\psi_0\in C([0,T];\Eb(\R^2))$ and $\Phi(\psi_i)\in C([0,T];H^1(\R^2))$ for $\psi=\psi^{\ast},\psi^1$.
More precisely, $\Zt(\eitD\psi_0^\ast)+\Zt(\eitD\psi_0^1)\leq 4\sqrt{2}\sqrt{\En(\psi_i^0)}$. It follows from \eqref{ineq:triangular} that
\begin{align*}
    \left\||\psi_1|-|\psi^{\ast}|\right\|_{L_t^{\infty}L_x^2}\leq C\left(1+\sqrt{\En(\psi_0^1)}+\sqrt{\En(\psi_0^{\ast})}+\|\Phi(\psi_1)\|_{L_t^{\infty}H_x^1}+\|\Phi(\psi^{\ast})\|_{L_t^{\infty}H_x^1}\right)\\
    \times\left(d_{\Eb}(\eitD\psi_0^1,\eitD\psi_0^{\ast})+\|\Phi(\psi_1)-\Phi(\psi^{\ast})\|_{L_t^{\infty}H_x^1}\right)\\
    \leq C(1+2\sqrt{R+\delta}+2M+2M^{1+2\alpha})\left(d_{\Eb}(\psi_0^1,\psi_0^{\ast})+\|\Phi(\psi_1)-\Phi(\psi^{\ast})\|_{L_t^{\infty}H_x^1}\right),
\end{align*}
where we used \eqref{eq:stabilitylinsol} in the last inequality. 
We are left to show that for all $\eps>0$ there exists $\delta>0$ such that $d_{\Eb}(\psi_0^{\ast},\psi_0)<\delta$ yields
\begin{equation*}
    \|\Phi(\psi_1)-\Phi(\psi^{\ast})\|_{L_t^{\infty}H_x^1}<\eps.
\end{equation*}
The statement follows by combining \eqref{eq:Ndiff}, and \eqref{eq:CD1} and observing that
\begin{equation*}
    \|\nabla\Phi(\psi_1)-\nabla\Phi(\psi^{\ast})\|_{L_t^{\infty}L_x^2}\leq \|\nabla\psi_1-\nabla\psi^{\ast}\|_{L_t^{\infty}L_x^2}+\sup_{t\in[0,T_2]}d_{\Eb}(\eitD\psi_0^1,\eitD\psi_0^{\ast})
\end{equation*}
followed by \eqref{eq:CD2} and \eqref{eq:CDlinear}. This completes Step 3.
\\
\textbf{Step 4:} Note that Step 3 yields continuous dependence on the initial data w.r.t. to the topology of $\Eb$ induced by the metric $d_{\Eb}$ on a time interval $[0,T_2]$ where $T_2$ only depends on $M$. 
One may hence cover $[0,T]$ by the union of intervals $[t_{k},t_{k+1}]$ with $t_k=k T_2$ for $k\in \{0, ..., N-1\}$ with $N=\lceil \frac{T}{T_2}\rceil$ finite. For all $\eps>0$, there exists $\delta_N>0$ such that $d_{\Eb}(\psi_{1}(t_{N-1}),\psi^{\ast}(t_{N-1}))<\delta_N$ yields $\sup_{t\in [t_{N-1},T]}d_{\Eb}(\psi_{1}(t),\psi^{\ast}(t))<\eps$. Next, there exists $\delta_{N-1}>0$ such that $d_{\Eb}(\psi_{1}(t_{N-2}),\psi^{\ast}(t_{N-2}))<\delta_{N-1}$ yields $\sup_{t\in[t_{N-2},t_{N-1}]}d_{\Eb}(\psi_{1}(t),\psi^{\ast}(t))<\delta_N$. One may then iterate the scheme finitely many times in order to recover the existence of a $\delta=\delta_1>0$ such that $d_{\Eb}(\psi_1^{0},\psi_0^{\ast})<\delta$ implies $\sup_{t\in [0,T]}d_{\Eb}(\psi_{1}(t),\psi^{\ast}(t))<\eps$.
\\
It remains to show that for $\mathcal{O}_r=\{\psi_0\in \Eb(\R^2) : d_{\Eb}(\psi_0^{\ast},\psi_0)<r\}$ it holds 
    \begin{equation*}
        \liminf_{r\rightarrow 0} T^{\ast}(\mathcal{O}_r)\geq T^{\ast}(\psi_0^{\ast}).
    \end{equation*}
This property is an immediate consequence of Step 4.
\end{proof}
We proceed to show a persistence of regularity property for \eqref{eq:NLS} under the general Assumption \ref{ass:N}. Subsequently, we prove the conservation of the Hamiltonian energy $\mathcal H$.
\begin{lemma}\label{lem:H2solution2D}
Let $f$ be as in Assumption \ref{ass:N} and $\psi_0\in\Eb(\R^2)$ such that $\Delta\psi_0\in L^2(\R^2)$. Then, the unique maximal solution $\psi\in C([0,T^{\ast});\Eb(\R^2))$ to \eqref{eq:NLS} satisfies
\begin{equation*}
    \Delta\psi\in C([0,T^{\ast});L^2(\R^2)), \qquad \partial_t\psi\in C([0,T^{\ast});L^2(\R^2)).
\end{equation*}
Furthermore, the Hamiltonian is conserved, namely
\begin{equation*}
    \mathcal{H}(\psi(t))=\mathcal{H}(\psi_0),
\end{equation*}
for all $t\in [0,T^{\ast})$.
\end{lemma}
\begin{proof}
Let $\psi_0\in \Eb(\R^2)$ such that $\Delta\psi_0\in L^2(\R^2)$. Proposition \ref{prop:localWP} provides a $T^{\ast}>0$ such that there exists a unique maximal strong solution  $\psi\in C([0,T^{\ast});\Eb(\R^2))$ to \eqref{eq:NLS} with initial data $\psi(0)=\psi_0$. The blow-up alternative yields that for any $T\in [0,T^{\ast})$ there exists $M>0$ such that $\Zt\leq M$, defined in \eqref{eq:Zt}. 
\\
First, we show that there exists  $T_1\in(0,T]$ only depending on $\Zt(\psi)$ such that $\partial_t\psi\in C([0,T_1];L^2(\R^2))$. Exploiting that $\psi\in C([0,T];\Eb(\R^2))$ we obtain
\begin{equation*}
    i\partial_t\psi(0)=-\frac{1}{2}\Delta\psi_0+\mathcal{N}(\psi_0).
\end{equation*}
We claim that $\partial_t\psi(0)\in L^2(\R^2)$. We note that $\Delta\psi_0\in L^2(\R^2)$ by assumption yields $\psi_0\in X^2+H^2(\R^2)\subset X^2(\R^2)\subset L^{\infty}(\R^2)$. It follows from \eqref{eq:boundN} that 
\begin{equation*}
    \|\mathcal{N}(\psi_0)\|_{L^2(\R^2)}\leq C\left(\sqrt{\En(\psi_0)}+\En(\psi_0)^{\frac12+\alpha}\right).
\end{equation*}
By differentiating the Duhamel formula \eqref{eq:Duhamel} in time and applying Corollary \ref{coro:time derivative} one has 
\begin{align*}
   \partial_t\psi(t)&=\eitD\left(\frac{\imu}{2}\Delta\psi(0)-i\mathcal{N}(\psi)(0)\right)-i\int_0^t\eul^{\frac{\imu}{2}s\Delta}\partial_t\mathcal{N}(\psi)(t-s)\dd s\\
   &=\eitD(\partial_t\psi(0))+\int_{0}^t \eul^{\frac{\imu}{2}(t-s)\Delta} \left(G_1(\psi)\partial_t\psi+G_2(\psi)\overline{\partial_t\psi}\right)(s) \dd s,
\end{align*}
where $G_1,G_2$ are as defined in \eqref{eq: Gi}.
Hence,
\begin{equation*}
    \|\partial_t\psi\|_{L^{\infty}([0,T];L^2(\R^2))}\leq \|\partial_t\psi(0)\|_{L^2(\R^2)}+\| G_1(\psi)\partial_t\psi+G_2(\psi)\overline{\partial_t\psi}(\partial_t\psi)\|_{N^0([0,T]\times\R^2)}.
\end{equation*}
Upon exploiting the estimates \eqref{eq: estG} on $G_1,G_2$ and following the lines of the proof of Lemma \ref{lem:N}, we conclude that
\begin{multline*}
    \left\|G_1(\psi)\partial_t\psi+G_2(\psi)\overline{\partial_t\psi}\right\|_{N^0([0,T]\times\R^2)}\leq C \|G_{\mathrm{bd}}(\psi)|\partial_t\psi|\|_{L_t^1L_x^2}+\|G_{\mathrm{int}}(\psi)|\partial_t\psi|\|_{N^0}\\
    \leq C \|\partial_t\psi\|_{L_t^1L_x^2}+\|
    \left(1+|\psi|^{2\alpha}\right)|\partial_t\psi|\|_{N^0}
    \leq  C T\|\partial_t\psi\|_{L_t^{\infty}L_x^2}+T^{\frac{1}{q_1'}}\Zt(\psi)^{2\alpha}\|\partial_t\psi\|_{L_t^{\infty}L_x^2}.
\end{multline*}
Thus, there exists $0<T_1<T$ only depending on $\Zt(\psi)$ such that
\begin{equation*}
    \left(T_1+T_1^{\frac{1}{q'}}\right)\left(1+\Zt(\psi)^{2\alpha}\right)<\frac{1}{2},
\end{equation*}
and 
\begin{equation*}
     \|\partial_t\psi\|_{L^{\infty}([0,T_1];L^2(\R^2))}\leq 2 \|\partial_t\psi(0)\|_{L^2(\R^2)}.
\end{equation*}
Second, we deduce a space-time bound for $\Delta\psi$. More precisely,
\begin{multline*}
    \|\Delta\psi\|_{L^{\infty}([0,T_1];L^2(\R^2))}\leq  \|\partial_t\psi\|_{L^{\infty}([0,T_1];L^2(\R^2))}+\|\mathcal{N}(\psi)\|_{L^{\infty}([0,T_1];L^2(\R^2))}\\
    \leq \|\partial_t\psi\|_{L^{\infty}([0,T_1];L^2(\R^2))}+\left(T_1+T_1^{\frac{1}{q_1'}}\right)\left({\Zt(\psi)}+\Zt(\psi)^{2\alpha}\right),
\end{multline*}
by virtue of \eqref{eq:boundN}.
As $\partial_t\psi\in C([0,T_1];L^2(\R^2))$ it then follows $\Delta\psi\in C([0,T_1];L^2(\R^2))$.
\\
Third, we show that $\mathcal{H}(\psi(t))=\mathcal{H}(\psi_0)$ for all $t\in[0,T_1]$. To that end, we compute the $L^2$-scalar product of \eqref{eq:NLS} with $\partial_t\psi$ and take the real part to infer
\begin{equation*}
    0=\RE\left\langle i\partial_t\psi,\partial_t\psi\right\rangle=\RE\left\langle -\frac12\Delta\psi+\mathcal{N}(\psi),\partial_t\psi\right\rangle,
\end{equation*}
for any $t\in [0,T_1]$. We notice that all terms are well-defined and 
conclude that for all $t\in [0,T_1]$ the Hamiltonian energy is conserved, namely
\begin{equation*}
    0=\frac{\dd}{\dd t}\int_{\R^d}\frac{1}{2}|\nabla\psi|^2+F(|\psi|^2)\dd x.
\end{equation*}
As $T_1>0$ only depends on $\Zt(\psi)$, the procedure above may be implemented starting from any $t_0\in [0,T-T_1]$ covering the time interval $[0,T]$ by finitely many sub-intervals. It follows that $\mathcal{H}(\psi)$ is constant in time on each of them. Since $\psi\in C([0,T];\Eb(\R^2))$, by continuity one concludes that $\mathcal{H}(\psi)(t)=\mathcal{H}(\psi_0)$ for all $t\in[0,T]$.
\end{proof}
The results of this Section then yield the proof of Theorem \ref{thm:mainlocal} for $d=2$.
\begin{proof}[Proof of Theorem \ref{thm:mainlocal} in 2D]
For $d=2$, the first three statements follow from Proposition \ref{prop:localWP}, while the fourth and fifth are provided by Lemma \ref{lem:H2solution2D}.
\end{proof}
\subsection{Global well-posedness}\label{sec:global2D}
Assuming the internal energy in \eqref{eq:HamiltonianNLS} to be non-negative, we show that the Cauchy problem associated to \eqref{eq:NLS} is globally well-posed in the space $\mathbb E(\R^2)$ which completes the proof of Theorem \ref{thm:main} for $d=2$. First, we show that the regular solutions provided by Lemma \ref{lem:H2solution2D} are global. 
\begin{corollary}\label{coro : GWP 2D}
Under the same assumptions of Lemma \ref{lem:H2solution2D}, let in addition the nonlinear potential energy density $F$, defined in \eqref{eq:HamiltonianNLS} be non-negative, namely $F\geq 0$. Then, the solution constructed in Lemma \ref{lem:H2solution2D} is global, i.e. $T^{\ast}=+\infty$.
\end{corollary}
\begin{proof}
Let $\psi\in C(0,T^{\ast};\Eb(\R^2))$ denote the unique maximal solution to \eqref{eq:NLS} with initial data $\psi(0)=\psi_0\in \Eb(\R^2)$. Since $\mathcal{H}(\psi)(t)=\mathcal{H}(\psi_0)$ for all $t\in [0,T^{\ast})$ it follows from Lemma \ref{lem:identification Eb} that there exists an increasing function $g:(0,\infty)\rightarrow (0,\infty)$ with $\displaystyle \lim_{r\rightarrow 0}g(r)=0$ such that 
\begin{equation}\label{eq:control En}
  \En(\psi)(t)\leq g\left(\mathcal{H}(\psi)(t)\right)=g\left(\mathcal{H}(\psi)(0)\right)=g\left(\mathcal{H}(\psi_0)\right) <+\infty
\end{equation}
for all $t\in [0,T^{\ast})$. The blow-up alternative then yields that $T^{\ast}=+\infty$. In addition, $\psi$ enjoys the bounds $\partial_t\psi\in C([0,T];L^2(\R^2))$ and $\Delta\psi\in C([0,T];L^2(\R^2))$ for any $T>0$ as well as $\mathcal{H}(\psi(t))=\mathcal{H}(\psi_0)$ for all $t\in [0,\infty)$.
\end{proof}
Second, we prove Theorem \ref{thm:main} for $d=2$. More precisely, by exploiting continuous dependence on the initial data we show that the Hamiltonian energy is conserved for solutions in the energy space and deduce global existence.
\begin{proof}[Proof of Theorem \ref{thm:main}]
Note that to complete the proof of the theorem it suffices to prove that the Hamiltonian energy is conserved for all solutions $\psi\in C([0,T^{\ast});\Eb(\R^2))$. Global existence then follows by arguing as in the proof of Corollary \ref{coro : GWP 2D}. To that end, given initial data $\psi_0\in \Eb(\R^3)$ and the unique solution $\psi\in C([0,T^{\ast});\Eb(\R^2))$ to \eqref{eq:NLS} such that $\psi(0)=\psi_0$, we observe that thanks to Lemma \ref{lem:approx} there exists $\{\psi_0^n\}\subset \Eb(\R^2)\cap C^{\infty}(\R^2)$ such that $\Delta\psi_0^n\in L^2(\R^2)$ and $d_{\Eb}(\psi_0,\psi_0^n)$ converges to $0$ as $n$ goes to infinity. Lemma  \ref{lem:H2solution2D} provides a sequence of unique global solutions $\psi_n\in C(\R,\Eb(\R^2))$ such that $\mathcal{H}(\psi_n)(t)=\mathcal{H}(\psi_0^n)$ for all $n$. Relying on the continuous dependence on the initial data, we conclude that for any $0<T<T^{\ast}$ one has
\begin{equation*}
\sup_{t\in [0,T]}d_{\Eb}(\psi(t),\psi_n(t))\rightarrow 0 \qquad \text{as} \quad n\rightarrow \infty.
\end{equation*}
Hence, $\En(\psi_n)(t)\rightarrow \En(\psi(t))$ for all $t\in [0,T]$. Similarly, conservation of the Hamiltonian energy $\mathcal{H}(\psi)$ follows from $\mathcal{H}(\psi_n)(t)\rightarrow \mathcal{H}(\psi)(t)$ for all $t\in[0,T]$. In particular, Lemma \ref{lem:identification Eb} yields an increasing function $g:(0,\infty)\rightarrow (0,\infty)$ with $\displaystyle \lim_{r\rightarrow 0}g(r)=0$ such that 
\begin{equation*}
    \En(\psi)(t)\leq 2\En(\psi_n)(t)\leq 2g\left(\mathcal{H}(\psi_n)(t)\right)=2g\left(\mathcal{H}(\psi_0^n)\right)\leq C,
\end{equation*}
for all $t\in [0,T]$ and $n$ sufficiently large. By means of the blow-up alternative we conclude that the solution is global, namely $\psi\in C(\R,\Eb(\R^2))$.
\end{proof}
\section{3D well-posedness}\label{sec:3d}
The approach to prove well-posedness for $d=3$ differs from the one for $d=2$ in two aspects.
First, we need to exploit that the nonlinear flow belongs to the full range of Strichartz spaces $S^1([0,T]\times\R^3))$, defined in \eqref{eq:S1N1}. In particular, exploiting also \eqref{eq:Strichartz gradient} we use that $\nabla \psi\in L^q([0,T];L^r(\R^3))$ for some $r>2$. For $d=3$, it is not sufficient to work in $L^2$-based function spaces - at least for super-cubic nonlinearities.
Second, Proposition \ref{prop:energyspace3d} yields an affine structure for the energy space $\Eb(\R^3)$. This allows for several simplifications of the well-posedness proofs, compared to Proposition \ref{prop:localWP}.
In this section, let
\begin{equation}\label{eq:qr3D}
    (q,r)=\left(\frac{4(\alpha+1)}{3\alpha},2(\alpha+1)\right)
\end{equation}
and note that 
$(q,r)$ is Schr\"odinger admissible. We recall that the Strichartz spaces $N^0$ and $N^1$ are defined in \eqref{eq:N0} and \eqref{eq:S1N1} respectively and the 
quantity $\Zt(\psi)$ in \eqref{eq:Zt}. 
\begin{proposition}\label{prop:LWP3d}
Let $d=3$ and $f$ be as in Assumption \ref{ass:N}.
\begin{enumerate}
    \item For any $\psi_0\in \Eb(\R^3)$ there exists a maximal existence time  $T^{\ast}=T^{\ast}(\psi_0)>0$ and a unique maximal solution $\psi\in C([0,T^{\ast});\Eb(\R^3))$ of \eqref{eq:NLS}. The blow-up alternative holds, namely if $T^{\ast}<\infty$ then
    \begin{equation*}
    \lim_{t\nearrow T^{\ast}}\En(\psi)(t)=+\infty;
    \end{equation*}
        \item For any $0<T<T^{\ast}(\psi_0),$ it follows
    \begin{equation*}
        \psi-\psi_0\in C([0,T];H^1(\R^3)), \quad \nabla\psi\in S^0([0,T]\times \R^3)),
    \end{equation*}
    moreover, the nonlinear flow satisfies
    \begin{equation*}
        \psi(t)-\eitD\psi_0\in C([0,T];H^1(\R^3))\cap S^1([0,T]\times \R^3);
    \end{equation*}
    \item the solution depends continuously on the initial data, namely if $\{\psi_0^n\}_{n\in \N}\subset \Eb(\R^3)$ is such that $d_{\Eb}(\psi_0^n,\psi_0)\rightarrow 0,$ then for any $0<T<T^{\ast}(\psi_0)$ it follows that $\sup_{t\in [0,T^{\ast})}d_{\Eb}(\psi_n(t),\psi(t))\rightarrow 0$, where $\psi_n$ denotes the unique local solution such that $\psi_n(0)=\psi_0^n$.
    \end{enumerate}
\end{proposition}
The affine structure of the energy space, see Proposition \ref{prop:energyspace3d} allows one to reduce the wellposedness of Cauchy Problem for \eqref{eq:NLS} to the wellposedness of an affine problem in $\mathcal{F}_c(\R^3)$, see Lemma \ref{lem:affine} and Remark \ref{rem: WP affine space} below. However, we only exploit this property for the proof of the continuous dependence on the initial data. Note that due to the affine structure it suffices to show sequential continuity. 
\begin{proof}
 To show existence of a local strong solution $\psi$, it suffices to implement a fixed-point argument for the map
\begin{equation}\label{eq:solutionmap3Dsuper}
    \Phi(u)(t)= i\int_0^t\eul^{\frac{\imu}{2}(t-s)\Delta}\mathcal{N}(\eul^{\imu\frac{s}{2}\Delta}\psi_0+u(s))\dd s.
\end{equation}
Indeed, if $u\in C([0,T];H^1(\R^3))$ is a fixed-point of \eqref{eq:solutionmap3Dsuper} then $\psi(t)=\eitD \psi_0+u(t)$ is such that $\psi\in C([0,T];\Eb(\R^3))$ due to Lemma \ref{lem:energyspace} and $\psi$ is a local strong solution of \eqref{eq:NLS}.
\\
\textbf{Local existence}
Fixed $(q,r)$ as in \eqref{eq:qr3D}, we implement a fixed-point argument for \eqref{eq:solutionmap3Dsuper} in
\begin{equation*}
    X_T=\{u\in C([0,T];H^1(\R^3))\cap L^q([0,T];W^{1,r}(\R^3)), \quad u(0)=0, \, \, \|u\|_{X_T}\leq M\}
\end{equation*}
with 
\begin{equation*}
    \|\cdot\|_{X_T}=\|\cdot\|_{L^{\infty}([0,T];H^1(\R^3))}+\|\cdot\|_{L^{q}([0,T];W^{1,r}(\R^3))}.
\end{equation*}
Equipped with the distance function
\begin{equation*}
    d_X(u,v)=\|u-v\|_{L^{\infty}([0,T];L^2(\R^3))}+\|u-v\|_{L^q([0,T];L^r(\R^3))},
\end{equation*}
the space $(X_T,d)$ is a complete metric space.
Let $\psi_0\in\Eb(\R^3)$ with $\En(\psi_0)\leq R$, where $M>0$ and $0<T\leq 1$ are to be fixed later. First, we verify that $\Phi: X_T \rightarrow X_T$. To that end, we recall that for $T=T(R)>0$ sufficiently small
\begin{equation*}
    \Zt(\eitD\psi_0+u)\leq \Zt(\eitD)+\|u\|_{H^1(\R^2)}\leq 2 \sqrt{2\En(\psi_0)}+M\leq 2\sqrt{2R}+M,
\end{equation*}
where $\Zt$ is defined in \eqref{eq:Zt} and \eqref{eq:energyH1} and \eqref{eq:freesolution} have been applied in the first and second inequality respectively.
It follows from \eqref{eq:Strichartznonh} that
\begin{equation*}
    \|\Phi(u)(t)\|_{L_t^{\infty}L_x^2}+ \|\Phi(u)(t)\|_{L_t^{q}L_x^r}\leq 2 \|\mathcal{N}(\eitD\psi_0+u)\|_{N^0}.
\end{equation*}
Defining $\mathcal{N}_1,\mathcal{N}_2$ as in \eqref{eq:N} and exploiting the pointwise bounds \eqref{eq:Npointwise}, we infer
\begin{equation*}
    \left\|\Na(\eitD\psi_0+u)\right\|_{L_t^1L_x^2}\leq CT \Zt(\eitD\psi_0+u)\leq CT\left(2\sqrt{2R}+M\right).
\end{equation*}
Next, using again \eqref{eq:Npointwise} and the Chebychev inequality \eqref{eq:Cheby} one has
\begin{equation*}
    \left\|\Nbl(\eitD\psi_0+u)\right\|_{L_t^1L_x^2}\leq CT \mathcal{L}^3\left(\supp(1-\eta(\eitD\psi_0+u))\right)^{\frac{1}{2}}\leq C T\left(2\sqrt{2R}+M\right)
\end{equation*}
and
\begin{multline*}
    \left\|\Nbh(\eitD\psi_0+u)\right\|_{L_t^1L_x^2+L_t^{q'}L_x^{r'}}\\
    \leq \left\|(1+|\eitD\psi_0+u|^{2\alpha})|\eitD\psi_0+u|(1-\chi(\eitD\psi_0+u))\right\|_{L_t^1L_x^2+L_t^{q'}L_x^{r'}} \\
    \leq CT(2\sqrt{2R}+M)+ \left\|\left|\left(\eitD\psi_0+u\right)(1-\chi(\eitD\psi_0+u))\right|^{2\alpha+1}\right\|_{L_t^{q'}L_x^{r'}}\\
    \leq CT(2\sqrt{2R}+M)+C{T}^{\frac{q-q'}{qq'}} \left\|(\eitD\psi_0+u)_q\right\|_{L^{\infty}L^r}^{2\alpha}\left\|(\eitD\psi_0+u)_q\right\|_{L_t^qL_x^r}
    \\
    \leq C\left(T+T^{\frac{q-q'}{qq'}}\left(2\sqrt{2R}+M\right)^{2\alpha}\right)\left(2\sqrt{2R}+M\right).
\end{multline*}
Moreover, Assumption \ref{ass:N}, see also \eqref{eq:nablaN}, implies the bound
\begin{equation*}
    \left|\nabla\mathcal{N}(\psi)\right|\leq C(1+|\psi|^{2\alpha})|\nabla \psi|,
\end{equation*}
which allows one to infer that 
\begin{multline*}
    \left\|\nabla\Na(\eitD\psi_0+u)+\nabla \Nbl(\eitD\psi_0+u)\right\|_{L_t^1L_x^2}\\
    \leq C T\left(\|\nabla\psi_0\|_{L_t^{\infty}L_x^2}+\|\nabla u\|_{L_t^{\infty}L_x^2}\right)\leq CT\left(2\sqrt{2R}+M\right).
\end{multline*}
To control $\nabla\Nbh$, note that $\eitD\nabla\psi_0\in L^q([0,T];L^r(\R^3))$ for any admissible pair $(q,r)$ from Lemma \ref{lem:Strichartz} and $\En(\psi_0)\leq R$. Therefore, 
\begin{multline*}
    \|\nabla \Nbh(\eitD\psi_0+u)\|_{L_t^1L_x^2+L_t^{q'}L_x^{r'}}
\leq CT\left(\|\nabla\psi_0\|_{L^2}+\|u\|_{X_T}\right)\\+ C\left(\left\||(\eitD\psi_0+u)_q|^{2\alpha}\nabla \eitD\psi_0\right\|_{L_t^{q'}L_x^{r'}}+\left\||(\eitD\psi_0+u)_q|^{2\alpha}\nabla u\right\|_{L_t^{q'}L_x^{r'}}\right)\\
\leq CT(2\sqrt{2R}+M)+C T^{\frac{q-q'}{qq'}}(2\sqrt{2R}+M)^{2\alpha}\left(\|\nabla\psi_0\|_{L_x^2}+\|\nabla u\|_{L_t^qL_x^{r}}\right)\\
\leq C\left(T+T^{\frac{q-q'}{qq'}}(2\sqrt{2R}+M)^{2\alpha}\right)\left(2\sqrt{2R}+M\right).    
\end{multline*}
Finally,
\begin{equation*}
    \|\Phi(u)\|_{X_T}\leq C\left(T+T^{\frac{q-q'}{qq'}}(2\sqrt{2R}+M)^{2\alpha}\right)\left(2\sqrt{2R}+M\right).    
\end{equation*}
We proceed to show that $\Phi$ defines a contraction on $X_T$. Let $\psi_0\in \Eb(\R^3)$ such that $\En(\psi_0)\leq R$ and $u,v\in X_T$. Then,
\begin{equation*}
    d_X\left(\Phi(u),\Phi(v)\right)\leq \left\|\mathcal{N}\left(\eitD\psi_0+u\right)-\mathcal{N}\left(\eitD\psi_0+v\right)\right\|_{N^0}
\end{equation*}
Inequality \eqref{eq: N loc Lip} implies that 
\begin{equation*}
    \left\|\Na\left(\eitD\psi_0+u\right)-\Na\left(\eitD\psi_0+v\right)\right\|_{L_t^1L_x^2}
    \leq C T d_X(u,v).
\end{equation*}
and 
\begin{equation*}
    \left\|\Nbl\left(\eitD\psi_0+u\right)-\Nbl\left(\eitD\psi_0+v\right)\right\|_{L_t^1L_x^2}
    \leq C T d_X(u,v).
\end{equation*}
Again inequality \eqref{eq: N loc Lip} allows us to control the remaining term as follows
\begin{align*}
    &\left\|\Nbh\left(\eitD\psi_0+u\right)-\Nbh\left(\eitD\psi_0+v\right)\right\|_{L^1L^2+L_t^{r'}L_x^{q'}}\\
    &\leq CT\|u-v\|_{L_t^{\infty}L_x^2}\\
    & \qquad + CT^{\frac{q-q'}{qq'}}\left(\Zt\left(\eitD\psi_0+u\right)^{2\alpha}+\Zt\left(\eitD\psi_0+v\right)^{2\alpha}\right)\|u-v\|_{L_t^qL_x^r}\\
    &\leq C\left(T+T^{\frac{q-q'}{qq'}}(2\sqrt{2R}+M)^{2\alpha}\right)d_X(u,v).
\end{align*}
Finally,
\begin{equation*}
     d_X\left((\Phi(u),\Phi(v)\right)\leq C\left(T+T^{\frac{q-q'}{qq'}}(2\sqrt{2R}+M)^{2\alpha}\right)d_X(u,v).
\end{equation*}
Therefore, it suffices to set $M=\sqrt{R}$ and choose $T=T(M)>0$ sufficiently small in order to conclude that $\Phi:X_T\rightarrow X_T$ and $\Phi$ defines a contraction on $X_T$. The Banach fixed-point Theorem yields a unique solution $u\in X_T$ to \eqref{eq:solutionmap3Dsuper}. In particular, $\psi(t)=\eitD\psi_0+u(t)$ solves \eqref{eq:NLS} with $\psi\in C([0,T];\Eb(\R^3))$.\\
\textbf{Uniqueness}
Let $R>0$ be fixed. Let $\psi_0\in \Eb(\R^3)$ with $\En(\psi_0)\leq R$ and $\psi_1,\psi_2\in C([0,T];\Eb(\R^3))$ two solutions to \eqref{eq:NLS} such that $\psi_1(0)=\psi_2(0)=\psi_0$. We note that $\psi_1-\psi_2\in S^1([0,T]\times\R^3)$. In particular, from the Strichartz estimate \eqref{eq:Strichartznonh} and arguing as for the local existence we obtain that
\begin{align*}
    d_{X}(\psi_1,\psi_2)&\leq \left\|\mathcal{N}(\psi_1)-\mathcal{N}(\psi_2)\right\|_{N^0([0,T]\times \R^3)}\\
    &\leq C\left(T+T^{\frac{q-q'}{qq'}}(\Zt(\psi_1)^{2\alpha}+\Zt(\psi_2)^{2\alpha}\right)d_{X}(\psi_1,\psi_2).
\end{align*}
Thus, there exists a sufficiently small $T_1>0$  such that $\psi_1=\psi_2$ a.e. on $[0,T_1]\times \R^3$. As $T_1$ only depends on $\Zt(\psi_i)$ with $i=1,2$ one may iterate the argument. This yields uniqueness in $C([0,T];\Eb(\R^3))$.
\\
\textbf{Blow up alternative}
The proof of the blow-up alternative follows \emph{verbatim} the proof of the respective statement for $d=2$, see Proposition \ref{prop:localWP} and thus it is omitted.
\\
\textbf{Membership in Strichartz spaces}
Statement (2) of Proposition \ref{prop:LWP3d} follows directly from the local existence argument and the properties of the free solution, see \eqref{eq:linearflow1} and \eqref{eq:Strichartz gradient}.
\\
The proof of the continuous dependence on the initial data requires some preliminary properties and is postponed after Lemma \ref{lem:CDOID3D}.
\end{proof}
In view of the equivalent characterisation of the energy space $\Eb(\R^3)$ provided by Proposition \ref{prop:energyspace3d}, the well-posedness for \eqref{eq:NLS} can be reduced to the well-posedness of the following "affine" problem.
\begin{lemma}\label{lem:affine}
Given $\psi_0\in \Eb(\R^3)$, let $\psi\in C([0,T^{\ast});\Eb(\R^3))$ be the unique maximal solution to \eqref{eq:NLS} with initial data $\psi_0$. Then, there exist $|c|=1$ and  $v\in C([0,T^{\ast});\Fc)$ such that $\psi(t)=c+v(t)$ for all $t\in [0,T^{\ast})$ and where $v$ is a solution to 
\begin{equation}\label{eq:NLS3d}
    i\partial_t v=-\frac{1}{2}\Delta v+f(|c+v|^2)(c+v), \quad v(0)=v_0.
\end{equation}
\end{lemma}
\begin{proof}
The unique maximal solution exists in virtue of Proposition \ref{prop:LWP3d}, Proposition \ref{prop:energyspace3d} yields the decomposition $\psi(t)=c(t)+v(t)$ for some $|c(t)|=1$ and $v(t)\in \Fc$ for all $t\in[0,T^{\ast})$. In particular, $c(0)=c$ and $v(0)=v_0$. It suffices to show that $c(t)=c$ for all $t\in [0,T^{\ast})$. From (\emph{2}) Proposition \ref{prop:LWP3d} we infer $\psi-\psi_0\in C([0,T];H^1(\R^3))$ for all $0<T<T^{\ast}$, namely $\psi(t)=c(t)+v(t)$ and $\psi_0=c+v_0$ share the same far-field behavior for all $t\in [0,T]$. It follows that $c(t)=c$ for all $t\in [0,T]$ with $0<T<T^{\ast}$.
\end{proof}
Given initial data $\psi_0=c+v_0$, the solution $\psi$ satisfies $\psi=\eitD\psi_0+\Phi(\psi)\in \{c\}+\Fc(\R^3)+H^1(\R^3)$. The connected component of $\Eb(\R^3)$ the solution $\psi$ belongs to is determined by the constant $c$, see Remark \ref{rem: connected components 3D}. Moreover, if $\psi=c+v\in C([0,T);\Eb(\R^3))$ such that $v$ solves \eqref{eq:NLS}, then $\widetilde{\psi}=\overline{c}\psi=1+\overline{c}v$ solves \eqref{eq:NLS} and $\widetilde{v}=\overline{c}v$ solves 
\begin{equation}\label{eq:c1}
    i\partial_t \tilde{v}=-\frac{1}{2}\Delta \tilde{v}+f(|1+\tilde{v}|^2)(1+\tilde{v}), \quad \tilde{v}(0)=\overline{c}v_0.
\end{equation}
It therefore suffices to consider $c=1$.
\begin{remark}\label{rem: WP affine space}
Note that Lemma \ref{lem:affine} reduces the well-posedness of \eqref{eq:NLS} in $\Eb(\R^3)$ to solving the affine problem \eqref{eq:NLS3d} in $\Fc$ where the constant $c$ is determined by the choice of the initial data. In particular, the continuous dependence on the initial data can be stated equivalently in terms of the metric \eqref{eq:metric3d} with the constants $c_1$ and $c_2$ determined by the initial data.
\\
If the nonlinearity is such that $f$ satisfies \eqref{ass:lipschitz}, then it is convenient to implement the well-posedness result in homogeneous spaces by exploiting Strichartz estimates on the gradient, see also \cite[Remark 4.5]{Gerard} for \eqref{eq:GP} and \cite[Proposition 1.1.18]{H} for \eqref{eq:NLS} with nonlinearity \eqref{eq:powernonlinearity}. Indeed, Assumption \eqref{ass:lipschitz} ensures that $\nabla\mathcal{N}$ is locally Lipschitz. A suitable choice of the functional spaces for the local well-posedness is given by
\begin{equation*}
    X_T=C([0,T];\Fc(\R^3))\cap L^q([0,T];\dot{W}^{1,r}(\R^3)),
\end{equation*}
where the Strichartz admissible pair is for instance $(q,r)=(10,\frac{30}{13})$, see \cite[Proposition 1.1.18]{H}.
\\
However, note that in the framework of Assumption \ref{ass:N}, this is ruled out by the lack of regularity of the nonlinearity $f$. More precisely, for $\nabla \mathcal{N}$ to be locally Lipschitz we require \eqref{ass:lipschitz}.
\end{remark}
We proceed to the proof of continuous dependence on the initial data for which we exploit the decomposition of $\psi$ given by Lemma \ref{lem:affine}.
\begin{lemma}\label{lem:CDOID3D}
Let $f$ satisfy Assumption \ref{ass:N}, $T>0$, $(q,r)$ as defined in \eqref{eq:qr3D} and $\psi_1,\psi_2\in C([0,T];\Eb(\R^3))$ such that $\psi_i=c_i+v_i$ with $c_i\in \C$, $|c_i|=1$ and $v_i\in C([0,T];\mathcal{F}_{c_i})$ for $i=1,2$. Then, there exists $\theta\in(0,1]$ such that
\begin{multline*}
    \left\|\mathcal{N}(\psi_1)-\mathcal{N}(\psi_2)\right\|_{N^0([0,T]\times\R^3)}\\
    \leq C T^{\theta}\left(1+{\Zt(\psi_1)}+{\Zt(\psi_2)}+\Zt(\psi_1)^{2\alpha}+\Zt(\psi_2)^{2\alpha}\right)\\
    \times\left(|c_1-c_2|+\|v_1-v_2\|_{L_t^{2}L_x^6}+\left\||\psi_1|-|\psi_2|\right\|_{L_t^{2}L_x^2}
    \right).
\end{multline*}
\end{lemma}
\begin{proof}
First, we notice that for $\mathcal{N}_1,\mathcal{N}_2$ defined in \eqref{eq:N} it follows from the first inequality of \eqref{eq:pointwise Nb} and the decomposition $\psi_i=c_i+v_i$ provided by Lemma \ref{lem:affine} that
\begin{multline*}
    \left\|\Na(\psi_1)-\Na(\psi_2)\right\|_{L_t^1L_x^2+L_t^{\frac43}L_x^{\frac32}}\\
    \leq \left\||c_1+v_1|\left||\psi_1|-|\psi_2|\right|\right\|_{L_t^1L_x^2+L_t^{\frac43}L_x^{\frac32}}+\left\|\left||\psi_2|-1\right|\left|c_1-c_2+v_1-v_2|\right|\right\|_{L_t^1L_x^2+L_t^{\frac43}L_x^{\frac32}}\\
    \leq C\left(T^{\frac12}+T^{\frac14}\Zt(\psi_1)\right) \left\||\psi_1|-|\psi_2|\right\|_{L_t^{2}L_x^2}\\
    +C T^{\frac{1}{4}}\Zt(\psi_2)\left|c_1-c_2\right|+C T^{\frac{1}{2}}\Zt(\psi_2))\|v_1-v_2\|_{L_t^{2}L_x^6}.
\end{multline*}
 Second, we observe that $\mathcal{L}^3(\supp(\Nb(\psi_i)))\leq \Zt(\psi_i)^2$ for $i=1,2$ from \eqref{eq:Cheby}. From \eqref{eq:pointwise Nb}, we conclude
\begin{multline*}
    \|\Nbl(\psi_1)-\Nbl(\psi_2)\|_{L_t^1L_x^2}\leq  CT\left(\Zt(\psi_1)+\Zt(\psi_2)\right)|c_1-c_2|\\
    +CT^{\frac12}\left(\Zt(\psi_1)^{\frac{2}{3}}+{\Zt(\psi_2)}^{\frac{2}{3}}\right) \left\|v_1-v_2\right\|_{L_t^{2}L_x^{6}}.
\end{multline*}
Third, we establish the desired bound for $\Nbh(\psi_1)-\Nbh(\psi_2)$. As $|\psi_i|\geq\frac32$ on $\supp(\Nbh(\psi_i))$, it follows from \eqref{eq:pointwise Nb} that
\begin{align*}
    \left|\Nbh(\psi_1)-\Nbh(\psi_2)\right|\leq C\left(1+|\psi_1|^{2\alpha}+|\psi_2|^{2\alpha}\right)|\psi_1-\psi_2|\\
    \leq C\left(|\psi_1|^{\beta}+|\psi_2|^{\beta}\right)|\psi_1-\psi_2|,
\end{align*}
with $\beta=\max\{2,2\alpha\}$. Hence, it suffices to consider $\alpha\in [1,2)$. We observe that 
\begin{equation*}
     \left|\Nbh(\psi_1)-\Nbh(\psi_2)\right|\leq C\left(1+|\psi_{1,\mathrm{int}}|^{\alpha}+|\psi_{2,\mathrm{int}}|^{\alpha}\right)|\psi_1-\psi_2|,
\end{equation*}
see also \eqref{eq:pointwise psi}. Using again that $\mathcal{L}^3(\supp(\Nb(\psi_i)))\leq \Zt(\psi_i)^2$, one recovers
\begin{align*}
      &\left\|\Nbh(\psi_1)-\Nbh(\psi_2)\right\|_{N^0}
      \leq \|\psi_1-\psi_2\|_{L_t^1L_x^2}+\|(|\psi_{1,\mathrm{int}}|^{2\alpha}+|\psi_{2,\mathrm{int}}|^{2\alpha})|c_1-c_2|\|_{L_t^{\frac43}L_x^{\frac32}}\\
      &+\|(|\psi_{1,\mathrm{int}}|^{2\alpha}+|\psi_{2,\mathrm{int}}|^{2\alpha})|v_1-v_2|\|_{L_t^{\frac{2}{3-\alpha}}L_x^{\frac{6}{2\alpha+1}}}
      \\&\leq C T\left(\Zt(\psi_1)+\Zt(\psi_2)\right)|c_1-c_2|+CT^{\frac{1}{2}}\left(\Zt(\psi_1)^{\frac23}+\Zt(\psi_2)^{\frac23}\right)\|v_1-v_2\|_{L_t^2L_x^6}\\
      &+C\left(\Zt(\psi_1)^{2\alpha}+\Zt(\psi_2)^{2\alpha}\right)T^{\frac{3}{4}}|c_1-c_2|+T^{\frac{2-\alpha}{2}}\|v_1-v_2\|_{L_t^2L_x^6}
\end{align*}
Combining the previous estimates, one concludes that there exists $\theta\in(0,1]$ such that
\begin{multline*}
    \left\|\mathcal{N}(\psi_1)-\mathcal{N}(\psi_2)\right\|_{N^0}\leq C T^{\theta}\left(1+\Zt(\psi_1)+\Zt(\psi_2)+\Zt(\psi_1)^{2\alpha}+\Zt(\psi_2)^{2\alpha}\right)\\
    \times\left(|c_1-c_2|+\|v_1-v_2\|_{L_t^2L_x^6}+\||\psi_1|-|\psi_2|\|_{L_t^2L_x^2}\right).
\end{multline*}
\end{proof}
We now prove continuous dependence on the initial data. As in the proof of Proposition \ref{prop:localWP}, we rely on an auxiliary metric to compensate for the lack of regularity of the nonlinearity $f$ and to deal with the non-integrability of the wave-functions. However, by virtue of Lemma \ref{lem:affine}, it suffices to consider the affine problem \eqref{eq:NLS3d}. This decomposition enables us to implement an argument in $L^2([0,T];L^6(\R^3))$. In particular, it is sufficient to prove sequential continuity.
\begin{proof}[Proof of Proposition \ref{prop:LWP3d} continued.]
Let $R>0$, $\psi_0\in \Eb(\R^3)$ with $\En(\psi_0)\leq R$ and $\psi_0^n\in \Eb(\R^3)$ such that $\En(\psi_0^n)\leq R$ and $d_{\Eb}(\psi_0,\psi_0^n)\rightarrow 0$. In particular, there exist complex constants $|c|=1$, $|c_n|=1$ and $v_0,v_0^n\in \Fc$ such that
\begin{equation*}
    \psi_0=c+v_0, \quad \psi_0^n=c_n+v_0^n.
\end{equation*}
It follows from the equivalence of metrics, see Proposition \ref{prop:energyspace3d}, that 
\begin{equation*}
    \delta(c+v_0,c_n+v_0^n)\rightarrow 0,
\end{equation*}
where $\delta$ is defined in \eqref{eq:metric3d}. There exists $T=T(2\En(\psi_0)>0$ such that the unique solutions $\psi,\psi_n\in C([0,T];\Eb(\R^3))$ to \eqref{eq:NLS} with initial data $\psi_0,\psi_0^n$ respectively satisfy
\begin{equation*}
    \Zt(\psi)+\Zt(\psi_n)\leq M
\end{equation*}
for sufficiently large $n$.
Then, Lemma \ref{lem:affine} implies that there exist $v,v_n\in C([0,T];\Fc)$ such that 
\begin{equation*}
    \psi=c+v, \qquad \psi_n=c_n+v_n.
\end{equation*}
The proof follows the same lines as the proof of Proposition \ref{prop:localWP}. We proceed in three steps corresponding to \eqref{eq:CD1}, \eqref{eq:CD2} and \eqref{eq:CD3} respectively.
\\
\textbf{Step 1:} We show that there exists $T_1=T_1(M)>0$ such that  
\begin{equation}\label{eq:3DCD1}
    \|v-v_n\|_{L^2([0,T_1];L^6(\R^3))}+\||\psi|-|\psi_n|\|_{L^2([0,T_1];L^2(\R^3))}\leq C\delta(c+v_0,c_n+v_0^n).
\end{equation}
For the first contribution, we observe that
\begin{align*}
    &\|v-v_n\|_{L^2([0,T];L^6(\R^3))}\\
    &= \left\|\eitD\psi_0-c+\Phi(\psi)-\eitD\psi_0^n+c_n-\Phi(\psi_n)\right\|_{L_t^2L_x^6}\\
    &\leq\left\|\eitD(\psi_0-\psi_0^n)-(\psi_0-\psi_0^n)\right\|_{L_t^2L_x^6}+\|v_0-v_0^n\|_{L_t^2L_x^6}+\left\|\mathcal{N}(\psi)-\mathcal{N}(\psi_n)\right\|_{N^0}\\
    &\leq C(T+T^{\frac{1}{2}})\delta(c+v_0,c_n+v_0^n)+\left\|\mathcal{N}(\psi)-\mathcal{N}(\psi_n)\right\|_{N^0},
\end{align*}
where we used \eqref{eq:Strichartznonh} in the second to last inequality and \eqref{eq:linearflow1} to control the difference of the free solutions in the last inequality. More precisely,
\begin{multline*}
    \left\|\eitD(\psi_0-\psi_0^n)-(\psi_0-\psi_0^n)\right\|_{L_t^2L_x^6}
    \leq T^{\frac{1}{2}}\left\|\eitD(\nabla\psi_0-\nabla\psi_0^n)-(\nabla\psi_0-\nabla\psi_0^n)\right\|_{L_t^{\infty}L_x^2}\\
    \leq CT\|\nabla\psi_0-\nabla\psi_0^n\|_{L_x^2}\leq C T \delta(c+v_0,c_n+v_0^n).
\end{multline*}
To bound the second contribution in \eqref{eq:3DCD1}, we proceed as in \eqref{eq:Step1int}. More precisely, we observe that \eqref{eq:Step1c} remains valid upon replacing the admissible Strichartz pair $(4,4)$ for $d=2$ with $(\frac83,4)$ for $d=3$.
Hence, the respective version of \eqref{eq:Step1c} reads that there exists $\theta_2\in (0,1]$ such that
\begin{multline}\label{eq:3DStep1c}
    \left\||\psi|-|\psi_n|\right\|_{L^2([0,T];L^2(\R^3))}
    \leq CT^{\theta_2}\left(1+M+M^{1+2\alpha}\right)\\
    \times\Big(\delta(c+v_0,c_n+v_0^n)+\|\Phi(\psi)-\Phi(\psi^n)\|_{S^0}\Big).
\end{multline}
Summing up and applying the Strichartz estimate \eqref{eq:Strichartznonh}, we conclude from Lemma \ref{lem:CDOID3D} that there exist $C=C(M)>0$ and $\theta>0$ such that
\begin{multline*}
\|v-v_n\|_{L^{2}([0,T_1];L^6(\R^3))}
        +\left\||\psi_n|-|\psi|\right\|_{L^{2}([0,T_1];L^2(\R^3))}
        \leq C_MT^{\theta}\\
        \times\left(\delta(c+v_0,c_n+v_0^n)+C_MT^{\theta}\left(\|v-v_n\|_{L_t^{2}L_x^6}
        +    \||\psi_n|-|\psi|\|_{L_t^{2}L_x^2}\right)\right).    
\end{multline*}
For $T_1>0$ sufficiently small depending only on $M$, inequality \eqref{eq:3DCD1} follows and Step 1 is complete. 
\\
\textbf{Step 2} We show that \eqref{eq:3DCD1} implies that there exists $T_2=T_2(M)>0$ such that 
\begin{equation}\label{eq:3DCD2}
    \|\nabla v-\nabla v_n\|_{L^{\infty}([0,T_2];L^2(\R^3))}+\|\nabla v-\nabla v_n\|_{L^{q}([0,T_2];L^r(\R^3))}\rightarrow 0,
\end{equation}
as $n\rightarrow \infty$ and where $(q,r)$ as in \eqref{eq:qr3D}. The proof follows closely the one of \eqref{eq:CD2} to which we refer for full details. In view of the Strichartz estimates of Lemma \ref{lem:Strichartz} it follows that
\begin{multline}\label{eq:3Ddifflinear}
    \|\nabla \eitD(c+v_0)-\nabla \eitD(c+v_0)\|_{L_t^{\infty}L_x^2}+ \|\nabla \eitD(c+v_0)-\nabla \eitD(c+v_0)\|_{L_t^{q}L_x^r}\\
    \leq C\|\nabla v_0-\nabla v_0^n\|_{L_x^2}.
\end{multline}
To control the non-homogeneous term, we recall that \eqref{eq:N_locLip} yields
\begin{equation*}
    \left|\nabla \mathcal{N}(\psi)\right|\leq C (1+|\psi|^{2\alpha})|\nabla \psi|\leq C(1+|\psi_{\mathrm{int}}|^{2\alpha})|\nabla\psi|.
\end{equation*}
More precisely, for $G_{\mathrm{bd}}, G_{\mathrm{int}}$ defined in \eqref{eq: Ginfty Gq} and upon applying \eqref{eq:Strichartznonh}, we split the non-homogeneous term:
\begin{equation}\label{eq:3DdiffgradnlCD}
    \begin{aligned}
    &\left\|i\int_0^t\eul^{\frac{\imu}{2}(t-s)\Delta}\left(\nabla\mathcal{N}(\psi)-\nabla\mathcal{N}(\psi_n)\right)(s)\dd s\right\|_{S^0([0,T]\times \R^3)}\\
    &\leq \|(G_{\mathrm{bd}})(\psi)|\nabla v-\nabla v_n|\|_{L_t^1L_x^2}+\|G_{\mathrm{int}}(\psi)|\nabla v-\nabla v_n|\|_{L_t^{q'}L_x^{r'}}\\
    &+\left\|\left(G_{\mathrm{bd}}(\psi)-G_{\mathrm{bd}}(\psi_n)\right)|\nabla v|\right\|_{L_t^1L_x^2}+\left\|\left(G_{\mathrm{int}}(\psi)-G_{\mathrm{int}}(\psi_n)\right)|\nabla v|\right\|_{L_t^{q'}L_x^{r'}}\\
   & \leq CT\|\nabla v-\nabla v_n\|_{L_t^{\infty}L_x^2}+CT^{\frac{q-q}{qq'}}\Zt(\psi)^{2\alpha}\|\nabla v -\nabla v_n\|_{L_t^qL_x^r}\\
  &+\left\|\left(G_{\mathrm{bd}}(\psi)-G_{\mathrm{bd}}(\psi_n)\right)|\nabla v|\right\|_{L_t^1L_x^2}+\left\|\left(G_{\mathrm{int}}(\psi)-G_{\mathrm{int}}(\psi_n)\right)|\nabla v|\right\|_{L_t^{q'}L_x^{r'}}.
    \end{aligned}
\end{equation}
Thus, for $T_2>0$ sufficiently small so that 
\begin{equation*}
    C\left(T_2+T_2^{\frac{q-q'}{qq'}}\Zt(\psi)^{2\alpha}\right)\leq \frac{1}{2},
\end{equation*}
we conclude from \eqref{eq:3Ddifflinear} and \eqref{eq:3DdiffgradnlCD} that
\begin{multline*}
    \| \nabla v-\nabla v_n\|_{L^{\infty}([0,T_2],L^2(\R^3))}+\| \nabla v-\nabla v_n\|_{L^{q}([0,T_2],L^r(\R^3))}\leq C \delta(c+v_0,c_n+v_0^n)\\
  +\left\|\left(G_{\mathrm{bd}}(\psi)-G_{\mathrm{bd}}(\psi_n)\right)|\nabla v|\right\|_{L_t^1L_x^2}+\left\|\left(G_{\mathrm{int}}(\psi)-G_{\mathrm{int}}(\psi_n)\right)|\nabla v|\right\|_{L_t^{q'}L_x^{r'}}.
\end{multline*}
To conclude that \eqref{eq:3DCD2}  holds, it suffices to show that the second line of the right-hand side converges to $0$ as $n$ goes to infinity. We proceed by contradiction assuming that there exists a subsequence still denoted $\psi_n$ such that there exists $\eps>0$ such that for all $n$ sufficiently large,
\begin{equation}\label{eq:Step2contradiction3D}
   \left\|\left(G_{\mathrm{bd}}(\psi)-G_{\mathrm{bd}}(\psi_n)\right)|\nabla v|\right\|_{L_t^1L_x^2}+\left\|\left(G_{\mathrm{int}}(\psi)-G_{\mathrm{int}}(\psi_n)\right)|\nabla v|\right\|_{L_t^{q'}L_x^{r'}}\geq \eps.
\end{equation}
Inequality \eqref{eq:3DCD1} implies that up to extracting a further subsequence, still denoted $\psi_n$, that $\psi_n=c_n+v_n$ converges to $\psi=c+v$ a.e. on $[0,T)\times \R^3$. By virtue of the Assumption \ref{ass:N}, one has that $G_{\mathrm{bd}}, G_{\mathrm{int}}$ are continuous. Therefore,
\begin{equation*}
    \begin{aligned}
        \left|\left(G_{\mathrm{bd}}(\psi)-G_{\mathrm{bd}}(\psi_n)\right)\right|\, |\nabla v|&\rightarrow 0 \quad \text{a.e. in} \quad [0,T)\times\R^3,\\
            \left|\left(G_{\mathrm{int}}(\psi)-G_{\mathrm{int}}(\psi_n)\right)\right|\,|\nabla v|&\rightarrow 0 \quad \text{a.e. in} \quad [0,T)\times\R^3.\\
    \end{aligned}
\end{equation*}
Further,
\begin{align*}
   &\|G_{\mathrm{int}}(\psi_n)\|_{L_t^{\infty}L_x^{\frac{2(\alpha+1)}{2\alpha}}(\R^3)}\leq C \||\psi_n|^{2\alpha}(1-\chi(\psi_n))\|_{L_t^{\infty}L_x^{\frac{2(\alpha+1)}{2\alpha}}(\R^3)}\\
   &\leq \mathcal{L}^3\left(\supp(1-\chi(\psi_n))\right)^{\frac{\alpha}{\alpha+1}}+\|\psi_{q,n}\|_{L^{\infty}L^{2(\alpha+1)}}^{2\alpha}\leq C\left(\Zt(\psi_n)^{\frac{2\alpha}{1+\alpha}}+\Zt(\psi_n)^{2\alpha}\right)\\
   &\leq C(M^{\frac{2\alpha}{\alpha+1}}+M^{2\alpha}),
\end{align*}
for all $n\in \N$, where we exploited \eqref{eq:Cheby}, namely that the measure of $\supp(1-\chi(\psi_n))$ is finite.
We obtain that there exists $\phi\in L^{\infty}([0,T];L^{2(\alpha+1)}(\R^3))$ such that $|\psi_{q,n}|\leq \phi$ a.e. on $[0,T)\times\R^3$. Therefore, we control
\begin{equation*}
    \begin{aligned}
        &\left|\left(G_{\mathrm{bd}}(\psi)-G_{\mathrm{bd}}(\psi_n)\right)\right|\,|\nabla v|\leq C |\nabla\psi|\in L^1([0,T);L^2(\R^3)),\\
            &\left|\left(G_{\mathrm{int}}(\psi)-G_{\mathrm{int}}(\psi_n)\right)\right|\,|\nabla v|\leq C\left(|\psi|^{2\alpha}+|\phi|^{2\alpha}\right) |\nabla\psi|\in L^{q'}([0,T);L^{r'}(\R^3)).\\
    \end{aligned}
\end{equation*}
The dominated convergence Theorem then implies that \eqref{eq:Step2contradiction3D} is violated, \eqref{eq:CD2} follows and Step 2 is complete.\\
\textbf{Step 3.} It remains to show that 
\begin{equation}\label{eq:3DCD3}
    \left\||\psi|-|\psi_n|\right\|_{L^{\infty}([0,T];L^2(\R^3))}\rightarrow 0.
\end{equation}
More precisely, we need to upgrade
\begin{equation*}
    \left\||\psi|-|\psi_n|\right\|_{L^{2}([0,T];L^2(\R^3))}\rightarrow 0,
\end{equation*}
so that the convergence holds for almost all times $t\in[0,T]$.
The proof follows closely the respective proof for $d=2$, namely the proof of \eqref{eq:CD3}. We omit the details.
\end{proof}
Next, we show a persistence of regularity property and that the Hamiltonian energy $\mathcal{H}$ is conserved for regular solutions.
The proof is completely analogous to the one for $d=2$, except that here we can exploit the affine structure of the energy space $\Eb$ and Sobolev embeddings which depend on the dimension. For the sake of clarity, we provide the proof of this lemma.
\begin{lemma}\label{lem: H23D}
Let $d=3$, $f$ as in Assumption \ref{ass:N} and $\psi_0\in\Eb(\R^3)$ such that $\Delta\psi_0\in L^2(\R^3)$. Then, the unique maximal solution $\psi\in C([0,T^{\ast});\Eb(\R^3))$ satisfies
\begin{equation*}
    \Delta\psi \in C([0,T];L^2(\R^3)), \qquad  \partial_t\psi \in C([0,T];L^2(\R^3))
\end{equation*}
for all $T\in [0,T^{\ast})$. Moreover, $\mathcal{H}(\psi)(t)=\mathcal{H}(\psi_0)$ for all $t\in [0,T^{\ast}))$. 
\end{lemma}
\begin{proof}
In view of Lemma \ref{lem:affine}, one has $\psi(t)=c+v(t)$ for all $t\in [0,T^{\ast})$ and it suffices to consider $v\in C([0,T^{\ast});\Fc(\R^3))$ solution to \eqref{eq:NLS3d}. The assumption  $v_0\in \Fc(\R^3)\cap \dot{H}^2(\R^3)$ yields that $\partial_tv(0)\in L^2(\R^3)$. Indeed, by continuity in time one has
\begin{equation*}
    i\partial_t v(0)=-\frac{1}{2}\Delta v(0)+\mathcal{N}(c+v)(0).
\end{equation*}
As $v(0)=v_0\in \Fc(\R^3)\cap \dot{H}^2(\R^3) \subset L^{\infty}(\R^3)$ it follows that $\Na(c+v_0)\in L^2(\R^3)$ from \eqref{eq:Npointwise} and $\Nb(c+v_0)\in L^{\infty}(\R^3)$ and hence in $L^2(\R^3)$ by means of \eqref{eq:Cheby}.
By differentiating the Duhamel formula in time and applying Corollary \ref{coro:time derivative}, it follows that
\begin{align*}
    i\partial_t v(t)&=\eitD\left(\frac{i}{2}\Delta v(0)-i\mathcal{N}(c+v)(0)\right)-i\int_{0}^t\eul^{\frac{\imu}{2}s\Delta}\partial_t\left(\mathcal{N}(c+v)(t-s)\right)\dd s\\
    &=\eitD\partial_tv-i\int_{0}^t\eul^{\frac{\imu}{2}(t-s)\Delta}\left(G_1(c+v)\partial_t v+G_2(c+v)\overline{\partial_tv}\right)(s)\dd s
\end{align*}
By means of the Strichartz estimates of Lemma \ref{lem:Strichartz}, for the admissible pair $(q,r)$ as in \eqref{eq:qr3D} and any $0<T<T^{\ast}$ we have that
\begin{multline*}
    \|\partial_t v\|_{L^{\infty}([0,T];L^2(\R^3))}+\|\partial_t v\|_{L^q([0,T];L^r(\R^3))}\\
    \leq 2\|\partial_t v(0)\|_{L^2(\R^3)}+\left\|G_1(c+v)\partial_t v+G_2(c+v)\overline{\partial_tv}\right\|_{N^0([0,T]\times\R^3)},
\end{multline*}
with $G_1,G_2$ defined in \eqref{eq: Gi}. Upon splitting $G_i$ in $G_{i,\infty}$ and $G_{i,\mathrm{int}}$, as in \eqref{eq: Ginfty Gq}, it follows that
\begin{align*}
    &\|G_i(c+v)|\partial_t v|\|_{N^0([0,T]\times\R^3)}\\
    &\leq C T \|\partial_t v\|_{L^{\infty}([0,T];L^2(\R^3))}+\||c+v|^{2\alpha}(1-\chi(c+v))|\partial_t v|\|_{N^0([0,T]\times\R^3)}\\
    &\leq C T \|\partial_t v\|_{L^{\infty}([0,T];L^2(\R^3))}+\left\||(c+v)_{q}|^{2\alpha}|\partial_t v \right\|_{L^{q'}([0,T];L^{r'}(\R^3)}\\
    &\leq C T \|\partial_t v\|_{L^{\infty}([0,T];L^2(\R^3))}+T^{\frac{q-q'}{qq'}}\Zt(c+v)^{2\alpha}\|\partial_t v\|_{L^q([0,T];L^r(\R^3))}
\end{align*}
Therefore, 
\begin{align*}
     \|\partial_t v\|_{L^{\infty}([0,T];L^2(\R^3))}&+\|\partial_t v\|_{L^q([0,T];L^r(\R^3))}\leq  2\|\partial_t v(0)\|_{L^2(\R^3)}\\
     &+C T \|\partial_t v\|_{L^{\infty}([0,T];L^2(\R^3))}+T^{\frac{q-q'}{qq'}}\Zt(c+v)^{2\alpha}\|\partial_t v\|_{L^q([0,T];L^r(\R^3))}.
\end{align*}
For $0<T_1<T^{\ast}$ sufficiently small, it holds
\begin{equation*}
    \|\partial_t v\|_{L^{\infty}([0,T_1];L^2(\R^3))}+\|\partial_t v\|_{L^q([0,T_1];L^r(\R^3))}\leq  4\|\partial_t v(0)\|_{L^2(\R^3)}.
\end{equation*}
Further,
\begin{align*}
    &\|\Delta v\|_{L^{\infty}([0,T_1];L^2(\R^3))}\leq 2\|\partial_t v\|_{L^{\infty}([0,T_1];L^2(\R^3))}+2\|\mathcal{N}(c+v)\|_{L^{\infty}([0,T_1];L^2(\R^3))}\\
    &\leq 2\|\partial_t v\|_{L^{\infty}([0,T_1];L^2(\R^3))}+4\Zt(c+v)+\||(c+v)_q)|^{2\alpha+1}\|_{L^{\infty}([0,T_1];L^2(\R^3))}
\end{align*}
Note that $|(c+v)_q|\geq 2$  and $|v|\geq 1$ on $\supp(1-\chi(c+v))$. If $\alpha\in (0,1]$, then 
\begin{equation*}
    \||(c+v)_q|^{2\alpha+1}\|_{L^{\infty}([0,T_1];L^2(\R^3))}\leq C \|v\|_{L^{\infty}([0,T_1];L^{6}(\R^3)}^{1+2\alpha}\leq C \Zt(c+v)^{1+2\alpha}.
\end{equation*}
If $\alpha\in (1,2)$, then we apply the Gagliardo-Nirenberg inequality to obtain that
\begin{equation*}
    \||(c+v)_q|^{2\alpha+1}\|_{L^{\infty}([0,T_1];L^2(\R^3))}\leq C \|v\|_{L^{\infty}([0,T_1];L^6(\R^3))}^{2-\alpha}\|\Delta v\|_{L^{\infty}([0,T_1];L^2(\R^3))}^{\alpha-1},
\end{equation*}
where we note that $0<\alpha-1<1$. It follows
\begin{equation*}
    \Delta v\in C([0,T_1];L^2(\R^3)). 
\end{equation*}
Finally, we conclude that $\mathcal{H}(c+v)(t)=\mathcal{H}(c+v_0)$ by performing an analoguous argument as in the proof of Lemma \ref{lem:H2solution2D} for $d=2$.
\end{proof}
\begin{proof}[Proof of Theorem \ref{thm:mainlocal} in 3D]
It only remains to show that the Hamiltonian energy is conserved for all solutions $\psi\in C([0,T^{\ast}),\Eb(\R^3))$ which follows from Proposition \ref{prop:LWP3d} and an approximation by smooth solutions by means of Lemma \ref{lem:approx} together with Lemma \ref{lem: H23D}.
\end{proof}
\subsection{Global well-posedness}\label{sec:3dglobal}
Similar to the $2D$ case, the lack of a suitable notion of (renormalized) mass and the lack of sign-definiteness of the Hamiltonian energy $\mathcal{H}$ constitute the main obstacles for proving global existence. 
\\
Assuming that $F\geq 0$ allows one to control the functional $\En(\cdot)$, in terms of which the blow-up alternative in Proposition \ref{prop:LWP3d} is stated, by $\mathcal{H}(\cdot)$, see Lemma \ref{lem:identification Eb}. Global existence is proven following closely the method detailed in Section \ref{sec:global2D} for $d=2$.
\begin{corollary}\label{coro:FnGWP3D}
Let Assumption \ref{ass:D} be satisfied and in addition the nonlinear potential energy density $F$, defined in \eqref{eq:HamiltonianNLS} be non-negative, namely $F\geq 0$. Then, the unique solution constructed in Proposition \ref{prop:LWP3d} is global, i.e. $T^{\ast}=+\infty$.
\end{corollary}
This proves Theorem \ref{thm:main} for $d=3$.
\\
Exploiting the affine structure of the energy space $\Eb(\R^3)$, we also prove global well-posedness for a class of equations for which the associated nonlinear potential energy density $F(|\psi|^2)$ fails to be non-negative. Such equations arise for instance in nonlinear optics to investigate self-focusing phenomena in a defocusing medium, see \cite{Barashenkov89, KL98, PelinovskySK}. A showcase model for such phenomena is \eqref{eq:NLS3d} with competing subcritical power-type nonlinearities satisfying Assumption \ref{ass:D} and having the form
\begin{equation}\label{eq:competing}
     f(r)=a_1(r^{\alpha_1}-\rho_0)-a_2(r^{\alpha_2}-\rho_0),
\end{equation}
where $a_1,a_2>0$ and $0<\alpha_2<\alpha_1<2$. The defocusing nonlinearity is dominant for large intensities $|\psi|^2>>\rho_0$ and focusing phenomena occur for small intensities $|\psi|^2\leq \rho_0$ where $\rho_0$ is determined by the far-field. 
The case $\alpha_1=2$, $\alpha_2=1$ corresponds to the energy-critical cubic-quintic nonlinearity and is investigated in \cite{Killip16, KOPV}. As before, we set $\rho_0=1$, further as in \eqref{eq:c1}, it suffices to consider $c=1$ and upon scaling space and time $a_1=1$. We are hence led to consider nonlinearities of the type 
\begin{equation}\label{eq:competing-scaled}
     f(r)=(r^{\alpha_1}-1)-a_2(r^{\alpha_2}-1),
\end{equation}
where Assumption \ref{ass:D} implies $\frac{\alpha_1}{\alpha_2}>a_2$. Furthermore, we may assume that $a_2>1$ as otherwise $F\geq 0$. Indeed, the following hold
\begin{enumerate}
    \item if $a_2\leq 0$, then it follows from \eqref{eq:power-potential} that $F(\rho)>0$ for all $\rho\geq 0$ with $\rho\neq 1$,
    \item if $0<a_2\leq 1$, then $f$ admits only one positive real root for $r=1$ corresponding to a global minimum of $F$. Hence, $F(\rho)>0$ for all $\rho\geq 0$ with $\rho\neq 1$,
    \item if $a_2>1$, then $f$ admits two positive real roots $\rho_1,1$ with  $0<\rho_1<1$ and $F$ displays a local minimum in $\rho=1$ and a local maximum in $\rho=\rho_1$. Depending on the location of the root $\rho_1$ two scenarios may occur:
    \begin{enumerate}[(a)]
        \item the root $\rho_1$ is sufficiently close to $0$ such that $F(\rho)\geq 0$ for all $\rho\geq 0$,
        \item the root $\rho_1$ is sufficiently close to $1$ such that there exists $\rho_2$ with $F(\rho)<0$ for all $0\leq \rho<\rho_2$.  
    \end{enumerate}
\end{enumerate} 
Thus, it suffices to study the case (3b), in particular $\frac{\alpha_1}{\alpha_2}>a_2>1$. The behavior of the competing power-type nonlinearities motivates the following assumptions.
\begin{assumption}\label{ass:global3D}
Let $f$ be a real-valued function satisfying Assumption \ref{ass:D} and having the form 
\begin{equation*}
    f(r)=(r^{\alpha_1}-1)+g(r)
\end{equation*}
where $0<\alpha_1<2$ and where $g\in C^{0}([0,\infty))\cap C^1(0,\infty)$ is such that
\begin{equation*}
|g(\rho)|,|\rho g'(\rho)|\leq C(1+\rho^{\alpha_2})
\end{equation*}
with $0\leq\alpha_2<\alpha_1$ for all $\rho\geq 0$.
In addition, $F(\rho)>0$ for all $\rho>1$. 
\end{assumption}
Local well-posedness for \eqref{eq:NLS3d} is provided by Theorem \ref{thm:mainlocal}. The assumptions yield that the nonlinear potential energy density $F$ is well-approximated by the one of the Ginzburg-Landau energy for $\rho$ close to $1$, see \eqref{eq:Taylor F} and coercive. Further, there exists $0\leq \rho_2<1$ such that the negative part $F_{-}$ satisfies
\begin{equation}\label{eq:negF}
    \supp (F_{-})\subset [0,\rho_2].
\end{equation}
For \eqref{eq:competing-scaled}, let $0<\rho_1<1$ denote the smaller root of $f$. Then, $0\leq \rho_2<\rho_1<1$.
\begin{proposition}\label{prop:3Dglobal}
    Let $f$ satisfy assumption \ref{ass:global3D} and $v_0\in \Fone(\R^3)$ with $\RE(v_0)\in L^2(\R^3)$. Then the unique local solution $v\in C([0,T);\Fone(\R^3))$ to \eqref{eq:c1} with initial data $v(0)=v_0$ provided by Proposition \ref{prop:LWP3d} is global.
\end{proposition}
In particular, Theorem \ref{thm:main3D} follows upon considering the phase shift given by multiplication of the datum with $\overline{c}$, see \eqref{eq:c1}. 
In order to compensate for the lack of sign-definiteness of the total energy, we restrict our analysis to the subspace of $\Fone(\R^3)$ such that $\RE(v)\in L^2(\R^3)$. Following \cite{KOPV}, for any $v\in \Fone(\R^3)$ with $\RE(v)\in L^2(\R^3)$, we define for $\psi=1+v$ the functional 
\begin{equation*}
    M(\psi)=\mathcal{H}(\psi)+C_0\int_{\R^3}\left|\RE(v)\right|^2\dd x,
\end{equation*}
for a suitable $C_0>0$ to be determined. 
To prove global well-posedness, we show coercivity of  $M$ and then conclude global existence by means of the Gronwall inequality, see Lemma \ref{lem:EM} and Lemma \ref{lem:bound M} respectively.
\begin{lemma}\label{lem:EM}
    Let $v\in \Fone(\R^3)$ such that $\RE(v)\in L^2(\R^3)$. Then, $M(1+v)$ is well-defined, in particular for all $C_0>0$ there exists an increasing function $h: (0,\infty)\rightarrow [0,\infty)$ with $\displaystyle \lim_{r\rightarrow 0}h(r)=0$ such that
    \begin{equation*}
       M(1+v) \leq h(\En(1+v))+C_0\|\RE(v)\|_{L^2}^2.
    \end{equation*}
    Moreover, there exist $C_0(\rho_1)>0$ and $C>0$ such that 
      \begin{equation}\label{eq:bound-EM}
        \En(1+v)\leq C M(1+v).
    \end{equation}
\end{lemma}
The constant $C_0>0$ only depends on $\rho_1$ being the second largest root of $F$ as in \eqref{eq:negF}.
\begin{proof}
The first inequality immediately follows from Lemma \ref{lem:finiteenergy}. To show the second inequality, it suffices to prove that there exist $C_2,C_0>0$ such that
\begin{multline*}
    \En(1+v)+C_2\int_{\R^3}F_{-}(|1+v|^2)\dd x\\
    \leq C_2\left(\frac12 \|\nabla v\|_{L^2(\R^3)}^2+\int_{\R^3}F_{+}(|1+v|^2)\dd x+C_0 \left\|\RE(v)\right\|_{L^2(\R^3)}^2\right).
\end{multline*}
Let $\delta\in (0,1)$ be such that the expansion \eqref{eq:Taylor F} of $F$ yields that\begin{equation*}
    \|\left(|1+v|-1\right)\mathbf{1}_{\{\left||1+v|^2-1\right|<\delta\}}\|_{L^2(\R^3)}^2\leq {C_l} \int_{\R^3}F(|1+v|^2)\mathbf{1}_{\{\left||1+v|^2-1\right|<\delta\}}\dd x.
\end{equation*}
for some $C_l>0$. On the other hand, by Assumption \ref{ass:global3D} the nonlinear potential energy is coercive and there exists $R_0>>1$ such that, 
\begin{align*}
    \left||1+v|-1\right|^2\leq C F(|1+v|^2),
\end{align*}
for all $|1+v|^2\geq R_0$. For $1+\delta\leq |1+v|^2\leq R_0$, it suffices to notice that $F$ is bounded from above and below away from $0$ to conclude that there exists $C_h>0$ such that
\begin{equation*}
\int_{\R^3}\left||1+v|-1\right|^2\mathbf{1}_{\{|1+v|^2\geq 1+\delta\}}\dd x\leq C_h  \int_{\R^3}F(|1+v|^2)\mathbf{1}_{\{|1+v|^2\geq 1+\delta\}}\dd x
\end{equation*}
by Assumption \ref{ass:global3D}. Let $C:=\max\{C_l,C_h\}$.
It remains to bound the negative part of $F$. One has 
\begin{equation*}
\supp(F_{-}(|1+v|^2))\subset \{|1+v|^2\leq \rho_2\}\subset \{|1+v|^2<1-\delta\}.
\end{equation*}
If $v$ is in the latter set, then necessarily $\RE(v)\in (-1-\sqrt{1-\delta},-1+\sqrt{1-\delta})$. In particular,
\begin{equation*}
    \{|1+v|^2<1-\delta\}\subset \{|\RE(v)|>\eta, \,\, \text{with} \, \eta:=1-\sqrt{1-\delta}\},
\end{equation*}
from which we conclude
\begin{equation*}
     \int_{\R^3}\left(\left||1+v|-1\right|^2+CF_{-}(|1+v|^2)\right)\mathbf{1}_{\{|1+v|^2\leq 1-\delta\}}\dd x\leq \frac{1+C}{\eta^2}\int_{\R^3}\left|\RE(v)\right|^2\dd x.
\end{equation*}
We observe that $\delta,\eta>0$ only depend on $0<\rho_1<1$ (and more precisely $\rho_2$) being the root of $f$ closest to but smaller than $1$. The expansion \eqref{eq:Taylor F} which is determined by $\alpha_1$ and $g$ guaranties that $f$ has an isolated root in $1$. Hence, there exists $C_0=C_0(\eta)>0$ such that the claim follows.
\end{proof}
\begin{remark}
    Note that in the case of a competing power-type nonlinearity \eqref{eq:competing-scaled} the constant $C_0>0$ only depends on $\alpha_1,\alpha_2$ and $a_2$ satisfying $\frac{\alpha_1}{\alpha_2}>a_2>1$.
\end{remark}
\begin{lemma}\label{lem:bound M}
Let $f$ satisfy Assumption \ref{ass:global3D}, $v_0\in \Fone(\R^3)$ such that $\RE(v_0)\in L^2(\R^3)$ and $v\in C([0,T^{\ast});\Fone)$ be the unique maximal solution to \eqref{eq:c1} with initial data $v_0$. Then there exists $C>0$ such that
\begin{equation*}
    M(1+v)(t)\leq  \eul^{C t} M(1+v_0))
\end{equation*}
for all $t\in [0,T^{\ast})$. In particular, there exists $D=D(\En(1+v_0),\|\RE(v_0)\|_{L^2}^2)>0$ such that 
\begin{equation*}
\En(1+v)(t)\leq D \eul^{C t}.
\end{equation*}
for all $t\in [0,T^{\ast})$.
\end{lemma}
\begin{proof}
First, let $v_0\in \Fone$, i.e. $\psi_0:=1+v_0\in \Eb(\R^3)$ and $\RE(v_0)\in L^2(\R^3)$, such that $\Delta v_0\in L^2(\R^3)$, then $\psi=1+v\in C([0,T^{\ast});\Eb(\R^3))$ and $\Delta v\in C([0,T];L^2(\R^3))$ for all $0<T<T^{\ast}$ by virtue of Theorem \ref{thm:mainlocal}. It follows that
\begin{equation*}
    \frac{\dd}{\dd t}M(\psi)(t)=C_0\frac{\dd}{\dd t}\int_{\R^3}\left|\RE(v)\right|^2\dd x,
\end{equation*}
where we exploited that $\mathcal{H}(\psi)(t)=\mathcal{H}(\psi_0)$ for all $t\in [0,T]$ from (4) Theorem \ref{thm:mainlocal}. Therefore,
\begin{equation*}
    \begin{aligned}
    &\frac{\dd}{\dd t}\int_{\R^3}\left|\RE(v)\right|^2\dd x
    =-2\int_{\R^3}\RE(v)\IM(\Delta v)\dd x+2\int_{\R^3}f(|1+v|^2)\RE(v)\IM(1+v)\dd x\\
    &\leq \int_{\R^3}\left|\nabla v\right|^2\dd x+2\int_{\R^3}f(|1+v|^2)\RE(v)\IM(v)\dd x,
    \end{aligned}
\end{equation*}
upon integrating by parts and using Young's inequality. The second term is decomposed as
\begin{multline*}
    2\int_{\R^3}f(|1+v|^2)\RE(v)\IM(v)\dd x=2\int_{\R^3}f(|1+v|^2)\IM(v)\RE(v)\mathbf{1}_{\{|1+v|^2\leq 1-\delta\}}\dd x\\
    +2\int_{\R^3}f(|1+v|^2)\IM(v)\RE(v)\mathbf{1}_{\{||1+v|^2-1|<\delta\}}\dd x\\
    +2\int_{\R^3}f(|1+v|^2)\IM(v)\RE(v)\mathbf{1}_{\{|1+v|^2\geq 1+\delta\}}\dd x
    =:I_1+I_2+I_3,
\end{multline*}
with $\delta\in(0,1)$ such that \eqref{eq:Taylor F} is valid for $||1+v|^2-1|\leq \delta$. We treat of the terms separately. Note that on $\{|1+v|^2\leq 1-\delta\}$ one has $\RE(v)\in (-1-\sqrt{1-\delta},-1+\sqrt{1-\delta})$.
Hence, for $\eta=1-\sqrt{1-\delta}$ we obtain
\begin{equation*}
    |I_1| \leq \frac{C}{\eta^2}\int_{\R^3}\left|\RE(v)\right|^2 \dd x.
\end{equation*}
Upon using the local Lipschitz property of $f$ and $f(1)=0$ and Cauchy-Schwarz followed by Young inequality, one has
\begin{equation*}
\begin{aligned}
        |I_2|&\leq C\int_{\R^3}(|1+v|^2-1)|\RE(v)|\mathbf{1}_{\{||1+v|^2-1|<\delta\}}\dd x\\
    &\leq C\left(\int_{\R^3}(|1+v|^2-1)^2\mathbf{1}_{\{||1+v|^2-1|<\delta\}} \dd x+\|\RE(v)\|_{L^2}^2\right)\\
    &\leq C\int_{\R^3}F(|1+v|^2)\mathbf{1}_{\{||1+v|^2-1|<\delta\}}\dd x+C\|\RE(v)\|_{L^2}^2,
    \end{aligned}
\end{equation*}
where we used \eqref{F:convex} in the last inequality.
It remains to control $I_3$. By virtue of Assumption \ref{ass:global3D}, we have that $F(\rho)>0$ for all $\rho>1$ and there exist $C>0$, $R_0>1$ such that $F(\rho)\geq C\rho^{1+\alpha_1}$ for all $\rho\geq R_0$. It follows that
\begin{equation*}
    |I_3|\leq \frac{CR_0^{1+\alpha_1}}{m}\int_{\R^3}F(|\psi|^2)\mathbf{1}_{\{1+\delta\leq |\psi|^2\leq R_0\}}\dd x+C\int_{\R^3}F(|\psi|^2)\mathbf{1}_{\{|\psi|^2\geq R_0\}}\dd x,
\end{equation*}
where $m=\displaystyle\min_{\rho\in[1+\delta,R_0]}F(\rho)>0$. We conclude that there exists $C>0$ such that 
\begin{equation*}
\frac{\dd}{\dd t}M(t)\leq CC_0\left(\mathcal{H}(1+v)(t)+\int_{\R^3}F_{-}(|1+v|^2)\dd x+\|\RE(v)\|_{L^2}^2\right).
\end{equation*}
Further, using that $\supp(F_{-})\subset \{|1+v|^2<1-\delta\}\subset \{|\RE(v)|>\eta\}$, we infer
\begin{equation*}
    \int_{\R^3}F_{-}(|1+v|^2)\dd x\leq \frac{C}{\eta^2}\|\RE(v)\|_{L^2}^2.
\end{equation*}
Finally, there exists $C>0$ such that 
\begin{equation*}
\frac{\dd}{\dd t}M(t)\leq C M(t). 
\end{equation*}
Gronwall's Lemma then yields
\begin{equation*}
    M(1+v)(t)\leq \eul^{C t}M(1+v)(0),
    \end{equation*}
and from Lemma \ref{lem:EM} we infer that there exists $D=D(\En(1+v_0), \|\RE(v_0)\|_{L^2}^2)>0$ with
\begin{equation*}
    \En(1+v)(t)\leq D\eul^{C t}.
    \end{equation*}
 The statement follows by approximation and the continuous dependence on the initial data provided by Lemma \ref{lem:approx} and Theorem \ref{thm:mainlocal} respectively.
\end{proof}
Global existence then follows from Lemma \ref{lem:bound M} and Theorem \ref{thm:mainlocal} by means of the blow-up alternative,  completing the proof of Theorem \ref{thm:main3D}. 
\begin{remark}\label{rem:smallness}
While our proof of global well-posedness in the case of non-sign-definite total energy $\Hc$ does not require a smallness condition, more decay of $\RE(v_0)$ than provided by $v_0\in \Fone(\R^3)$ is assumed, namely $\RE(v_0)\in L^2(\R^3)$. The finite energy assumption only yields $v_0\in L^6(\R^3)$ and $|v|^2+2\RE(v_0)\in L^2(\R^3)$. 
\\
Under Assumption \ref{ass:global3D} and instead of $\RE(v_0)\in L^2(\R^3)$, one may alternatively assume that the initial data are such that $\Hc(1+v_0)$ and $\|\nabla\RE(v_0)\|_{L^2}^2$ are sufficiently small adapting \cite[Lemma 3.2]{Killip16} stated for cubic-quintic nonlinearities \eqref{eq:cubicquintic}. Moreover, as pointed out in \cite[Remark p. 2683]{Killip18} the same argument yields small data global well-posedness for the cubic-quintic nonlinearity where the quintic part is focusing and the cubic part defocusing, hence for $F$ being unbounded from below. Inspired, by this observation and the classical small data global well-posedness in $H^1$ for NLS eq, see e.g. \cite[Chapter 3.4]{Tao}, we prove that \eqref{eq:NLS3d} is globally well-posed in the energy space for small data provided that Assumption \ref{ass:D} holds.
\end{remark}
\begin{proposition}
    If Assumption \ref{ass:D} is satisfied, then there exists $\varepsilon>0$ only depending on $\delta>0$ as in \eqref{F:convex} such that if $\mathcal{H}(1+v_0)\leq \frac{1}{4}\eps$ and $\|\nabla v_0\|_{L^2}^2\leq \eps$, then the unique solution provided by Proposition \ref{prop:LWP3d} is global. 
\end{proposition}
This proves Theorem \ref{thm: smalldata GWP}. If $\supp(F_{-})\subset [0,1)$, it suffices to assume $\|\nabla \RE(v_0)\|_{L^2}^2$ small instead of  $\|\nabla v_0\|_{L^2}^2$ small. 
\begin{proof}
First we show that under the given assumptions one has $\En(1+u_0)\leq C{\eps}$ and second a continuity argument then yields that $\En(1+v)(t)$ remains bounded for all times. 
We claim that there exists $\eps>0$ such that if $\|\nabla v_0\|_{L^2}^2\leq \eps$ and $\mathcal{H}(1+v_0)\leq\frac{\varepsilon}{4}$, then $v_0\in \Fone$ and 
\begin{equation}\label{eq:energy threshold}
    \En(1+v_0)\leq C\left(\Hc(1+v_0)+\|\nabla v_0\|_{L^2}^2\right)=C\eps.
\end{equation}
The inequality is proven arguing as in Lemma \ref{lem:EM}. Indeed, instead of relying on the bound $\RE(v_0)\in L^2(\R^3)$, one exploits the bound 
\begin{equation*}
    \|\RE(v_0)\|_{L^6}^6\leq C \|\nabla \RE(v_0)\|_{L^2}^6
\end{equation*}
together with $|\RE(v)|\geq \eta>0$ for some $\eta>0$ whenever $|1+v|^2\in \supp(F_{-}\mathbf{1}_{\{|1+v|^2<1\}})$, where $\eta$ depends on $\delta>0$ as in \eqref{F:convex}. This yields
\begin{equation*}
    \left|\int_{\R^3}F_{-}(|1+v_0|^2)\mathbf{1}_{\{|1+v_0|^2<1-\delta\}}\mathrm{d}x\right|\leq \frac{C}{\eta^6}\|\RE(v_0)\|_{L^6}^6\leq \frac{C}{\eta^6}\|\nabla\RE(v_0)\|_{L^2}^6 \leq \frac{1}{8} \|\nabla\RE(v_0)\|_{L^2}^2
\end{equation*}
provided that $\|\nabla\RE(v_0)\|_{L^2}<< \eta^\frac{3}{2}$. Similarly, there exists $\nu>0$ only depending on $\delta>0$ as in \eqref{F:convex} such that in $\supp(F \mathbf{1}_{\{|1+v|^2>1}\})$, we have $|\RE(v_0)|>\nu$ or $|\IM(v_0)|>\nu$. Hence, 
\begin{equation*}
    \left|\int_{\R^3}F(|1+v_0|^2)\mathbf{1}_{\{|1+v_0|^2>1+\delta\}}\mathrm{d}x\right|\leq \frac{1}{8} \|\nabla v_0\|_{L^2}^2.
\end{equation*}
The inequality \eqref{eq:energy threshold} follows. Along the same lines one proves that 
\begin{align*}
    \En(1+v)(t)&\leq C\Hc(1+v)(t)+C'\|\nabla v\|_{L^2}^6\\
    &\leq C\Hc(1+v_0)+C'\En(1+v)^6(t)=\frac{C}{4}\eps+C'\En(1+v)^6(t).
\end{align*}
Provided that $\varepsilon>0$ is sufficiently small, a continuity argument yields that $\En(1+v)(t)$ remains bounded, hence by virtue of the blow-up alternative stated in Proposition \ref{prop:LWP3d} global existence follows.
\end{proof}
\section{Lipschitz continuity of the solution map}\label{sec:Lipschitz}
In this section, we provide the proof of Theorem \ref{thm:lipschitz}. Namely, we show that provided $f$ satisfies \eqref{ass:lipschitz} in addition to Assumption \ref{ass:N}, the solution map is Lipschitz continuous on bounded sets of $\Eb(\R^d)$.
\begin{proof}[Proof of Theorem \ref{thm:lipschitz}]
Let $R>0$ and $\psi_0^1,\psi_0^2\in \Eb(\R^d)$ such that $\En(\psi_0^i)\leq R$ for $i=1,2$. Then, for all $0<T<T^{\ast}(\mathcal{O}_{R})$ there exists $M>0$ such that the unique maximal solutions $\psi_1,\psi_2\in C([0,T];\Eb(\R^d))$ satisfy
\begin{equation*}
    \Zt(\psi_1)+\Zt(\psi_2)\leq M,
\end{equation*}
with $\Zt$ defined in \eqref{eq:Zt}. By virtue of \eqref{ineq:triangular}, it follows that
\begin{equation}\label{eq:triangular ineq}
\begin{aligned}
    &d_{\Eb}(\psi_1(t),\psi_2(t))
    \leq C(1+M) d_{\Eb}(\eitD \psi_0^1,\eitD\psi_0^2) \\
    &+ C(1+M) \left\|-i\int_0^t\eul^{\frac{\imu}{2}(t-s)\Delta}\left(\mathcal{N}(\psi_1(s))-\mathcal{N}(\psi_2(s))\right)\dd s \right\|_{L^{\infty}([0,T];H^1(\R^3))}\\
    &\leq C(1+M) d_{\Eb}(\psi_0^1,\psi_0^2)+ C(1+M) \left\|\mathcal{N}(\psi_1)-\mathcal{N}(\psi_2)\right\|_{N^1([0,T]\times \R^d)},
\end{aligned}
\end{equation}
where we used \eqref{eq:stabilitylinsol} to control the distance of the free solutions and the Strichartz estimate \eqref{eq:Strichartznonh} to control the nonlinear flow. Lemma \ref{coro:nonlinear flow} and Lemma \ref{lem:CDOID3D} for $d=2,3$ respectively yield that 
\begin{equation}\label{eq: Lipschitz N}
    \left\|\mathcal{N}(\psi_1)-\mathcal{N}(\psi_2)\right\|_{N^0([0,T]\times \R^d)}\leq C(1+M+M^{2\alpha}) T^{\theta}\sup_{t\in[0,T]}d_{\Eb}(\psi_1(t),\psi_2(t)).
\end{equation}
It remains to control $\nabla\mathcal{N}(\psi_1)-\nabla\mathcal{N}(\psi_2)$ in $N^0([0,T]\times \R^d)$. To that end, we recall that $\nabla\mathcal{N}(\psi_i)$ can be decomposed by means of the functions $G_{\mathrm{bd}}(\psi_i), G_{\mathrm{int}}(\psi_i)$ defined in \eqref{eq: Ginfty Gq}.
One has that 
\begin{equation}\label{eq: grad N Lipschitz}
\begin{aligned}
    &\left\|\nabla\mathcal{N}(\psi_1)-\nabla\mathcal{N}(\psi_2)\right\|_{N^0([0,T]\times \R^d)}\\
    &\leq \||G_{\mathrm{bd}}(\psi_1)||\nabla\psi_1-\nabla\psi_2|\|_{L^{\infty}([0,T];L^{2}(\R^d))}+\||G_{\mathrm{int}}(\psi_1)||\nabla\psi_1-\nabla\psi_2|\|_{N^0([0,T]\times \R^d)}\\
    &+\left\||G_{\mathrm{bd}}(\psi_1)-G_{\mathrm{bd}}(\psi_2)||\nabla\psi_2|\right\|_{N^0([0,T]\times \R^d)}+\left\||G_{\mathrm{int}}(\psi_1)-G_{\mathrm{int}}(\psi_2)||\nabla\psi_2|\right\|_{N^0([0,T]\times \R^d)}.
\end{aligned}
\end{equation}
Note that \eqref{eq: estG} yields that
\begin{equation*}
    |G_{\mathrm{bd}}(\psi_1)|\leq C, \quad \text \quad |G_{\mathrm{int}}(\psi_1)|\leq C(1+|\psi_1|^{2\alpha}).
\end{equation*}
Further, \eqref{ass:lipschitz} yields that $G_{\mathrm{bd}}$ and $ G_{\mathrm{int}}$ are locally Lipschitz, namely,
\begin{equation}
    \begin{aligned}
    \left|G_{\mathrm{bd}}(\psi_1)-G_{\mathrm{bd}}(\psi_2)\right|&\leq C \left||\psi_1|-|\psi_2|\right|,\\
    \left|G_{\mathrm{int}}(\psi_1)-G_{\mathrm{int}}(\psi_2)\right|&\leq C\left(1+|\psi_1|^{2\beta}+|\psi_2|^{2\beta}\right) \left||\psi_1|-|\psi_2|\right|,
    \end{aligned}
\end{equation}
wit $\beta=\max\{0,\alpha-\frac12\}$. As $|\psi_i|\geq 1$ on the support of $G_{\mathrm{int}}(\psi_i)$, we may assume in the following that $\beta \geq 1$.
\\
In the following, we distinguish to cases.
\\
\textbf{Case 1: $d=2$:} 
Consider the admissible pair $(q_1,r_1))=(\frac{2(\alpha+1)}{\alpha},2(\alpha+1))$, see also \eqref{eq:q1r1}. To bound the first line on the right hand side of \eqref{eq: grad N Lipschitz}, we observe that
\begin{equation*}
    \||G_{\mathrm{bd}}(\psi_1)||\nabla\psi_1-\nabla\psi_2|\|_{L^{1}([0,T];L^{2}(\R^2))}\leq C T \|\nabla\psi_1-\nabla\psi_2\|_{L^{\infty}([0,T];L^2(\R^2))},
\end{equation*}
and 
\begin{equation*}
    \||G_{\mathrm{int}}(\psi_1)||\nabla\psi_1-\nabla\psi_2|\|_{N^0([0,T]\times \R^2)}\leq T^{\frac{1}{q_1'}}\Zt(\psi_1)^{2\alpha}\|\nabla\psi_1-\nabla\psi_2\|_{L^{\infty}([0,T];L^2(\R^2))}.
\end{equation*}
To bound the first term of the second line on the right hand side of \eqref{eq: grad N Lipschitz}, one has
\begin{align*}
    &\left\||G_{\mathrm{bd}}(\psi_1)-G_{\mathrm{bd}}(\psi_2)||\nabla\psi_2|\right\|_{N^0([0,T]\times \R^2)}\leq C \left\||\psi_1|-|\psi_2||\nabla\psi_2| \right\|_{L^{\frac{4}{3}}([0,T];L^{\frac{4}{3}}(\R^2))}\\
    &\leq T^{\frac12} \left\||\psi_1|-|\psi_2| \right\|_{L^{\infty}([0,T];L^2(\R^2)}\|\nabla\psi_2\|_{L^4([0,T];L^4(\R^2))}\\
    &\leq T^{\frac{1}{2}}\left(1+T+T^{\frac{1}{q_1'}}\Zt(\psi_1)^{2\alpha}\right)\Zt(\psi)  \left\||\psi_1|-|\psi_2| \right\|_{L^{\infty}([0,T];L^2(\R^2))},
\end{align*}
where we used the Strichartz estimates \eqref{eq:Strichartz gradient}, \eqref{eq:Strichartznonh} and \eqref{eq:boundgradN} in the last inequality.
To bound the second term of the line on the right hand side of \eqref{eq: grad N Lipschitz}, we have that
\begin{align*}
    &\left\||G_{\mathrm{int}}(\psi_1)-G_{\mathrm{int}}(\psi_2)||\nabla\psi_2|\right\|_{N^0([0,T]\times \R^2)}\\
    &\leq C \left\|\left(1+|\psi_{1,\mathrm{int}}|^{2\beta}+|\psi_{2,\mathrm{int}}|^{2\beta}\right)\left||\psi_1|-|\psi_2|\right|\nabla\psi_2| \right\|_{N^0([0,T]\times \R^2)}\\
    &\leq \left(T^{\frac12}\|\nabla\psi\|_{L^4L^4}
    + T^{\frac13}\left(\||\psi_{1,\mathrm{int}}|^{2\beta}+|\psi_{1,\mathrm{int}}|^{2\beta}\|_{L_t^{\infty}L_x^6}\right)\|\nabla\psi\|_{L_t^3L_x^6}\right)\left\||\psi_1|-|\psi_2| \right\|_{L_{t}^{\infty}L_x^2}\\
    &\leq \left(T^{\frac{1}{2}}+T^{\frac{1}{3}}\left(\Zt(\psi_1)^{2\beta}+\Zt(\psi_2)^{2\beta}\right)\right)
    \left(1+T+T^{\frac{1}{q_1'}}\Zt(\psi_1)^{2\alpha}\right)\Zt(\psi) \\
    &\qquad \cdot\left\||\psi_1|-|\psi_2| \right\|_{L_t^{\infty}L_x^2}
\end{align*}
where we used the Strichartz estimates \eqref{eq:Strichartz gradient}, \eqref{eq:Strichartznonh} and \eqref{eq:boundgradN} in the last inequality.
Combining the above estimates, we obtain that there exists a sufficiently small $T_1=T_1(M)>0$ so that 
\begin{equation*}
    d_{\Eb}\left(\psi_1(t), \psi_2(t)\right)\leq C(1+M) d_{\Eb}(\psi_0^1,\psi_0^2) 
\end{equation*}
for all $t\in [0,T_1]$. Note that $T_1$ only depends on $M$, one may hence iterate the procedure $N:=\lceil \frac{T}{T_1}\rceil$ times to cover the time interval $[0,T]$. This completes the case $d=2$.
\\
\textbf{Case 2: $d=3$.} The proof for $d=3$ follows the same lines upon modifying the space-time norms so that the pairs of exponents are Strichartz admissible for $d=3$. In particular, one relies on the endpoint Strichartz estimate \eqref{eq:Strichartz gradient} to bound $\nabla\psi_2\in L^2([0,T];L^6(\R^3))$.
\end{proof}
If the solutions are global, i.e. $T^{\ast}(\mathcal{O}_{ R})=+\infty$, then Theorem \ref{thm:lipschitz} extends to the following.
\begin{corollary}\label{coro:lipschitz}
Under the Assumptions of Theorem \ref{thm:lipschitz}, if in addition $f$ is such that \eqref{eq:NLS} is globally well-posed then for any $R>0$, $T>0$, there exists $C>0$ such that for all  $\psi_0^i\in\Eb(\R^d)$, where $i=1,2$, with $\En(\psi_i)\leq R$ the respective unique solutions $\psi_i\in C(\R,\Eb(\R^d))$ satisfy \eqref{eq:Lipschitz main}.
\end{corollary}
\section*{Acknowledgements}
P.A. and P.M. acknowledge partial support from INdAM-GNAMPA. 
P.A. is partially supported by PRIN project 20204NT8W4 
"Nonlinear evolution PDEs, fluid dynamics and transport equations: theoretical foundations and applications'' and by the Italian Ministry of University and Research (MUR) through the Excellence Department Project awarded to GSSI, CUP D13C22003740001.
L.E.H. is funded by the Deutsche Forschungsgemeinschaft (DFG, German Research Foundation) – Project-ID 317210226 – SFB 1283.

\bibliographystyle{siam}
\bibliography{bibliomain}  
\end{document}